\documentclass[11pt,reqno]{amsart}

\usepackage{amsmath,amssymb,amsthm,amsfonts}   
\usepackage{bbm} 
\usepackage{graphicx}   
\usepackage{pstricks}   
\usepackage[colorlinks=true, linkcolor=blue, citecolor=magenta]{hyperref} 
\usepackage{url}
\usepackage{tasks}
\usepackage{enumerate}
\usepackage{caption,subcaption} 

\usepackage{mathtools}      
\mathtoolsset{showonlyrefs} 

\usepackage{mathrsfs}
\usepackage{multicol}   
\usepackage{fullpage}   

\newtheorem{thm}{Theorem}[section]
\newtheorem{cor}[thm]{Corollary}
\newtheorem{prop}[thm]{Proposition}
\newtheorem{lem}[thm]{Lemma}

\newtheorem*{acknowledgements*}{Acknowledgements}

\theoremstyle{definition}

\newtheorem{rem}[thm]{Remark}

\title{Multifractality and intermittency 
in the limit evolution of polygonal vortex filaments
}

\author[V. Banica]{Valeria Banica}
 \address[V. Banica]{
Laboratoire Jacques-Louis Lions (LJLL),  
Sorbonne Universite, CNRS, Universit\'e de Paris, 
France
and 
Institut Universitaire de France (IUF), 
France
 }
\email{Valeria.Banica@math.cnrs.fr}

\author[D. Eceizabarrena]{Daniel Eceizabarrena}
 \address[D. Eceizabarrena]{
 Department of Mathematics and Statistics,	
 University of Massachusetts Amherst, 
 United States
 }
\email{deceizabarre@umass.edu}

\author[A. R. Nahmod]{Andrea R. Nahmod}
 \address[A. Nahmod]{
Department of Mathematics and Statistics,
University of Massachusetts Amherst, 
United States 
 }
\email{nahmod@umass.edu}

\author[L. Vega]{Luis Vega}
 \address[L. Vega]{
 BCAM - Basque Center for Applied Mathematics, 
 Spain, 
 and 
 Departamento de Matem\'aticas, 
 Universidad del País Vasco (UPV/EHU), 
 Spain
 }
\email{lvega@bcamath.org}


\begin{document}

\begin{abstract}
With the aim of quantifying turbulent behaviors of vortex filaments,  
we study the multifractality and intermittency of 
the family of generalized Riemann's non-differentiable functions
\begin{equation}
R_{x_0}(t) = \sum_{n \neq 0} \frac{e^{2\pi i ( n^2 t + n x_0 ) } }{n^2},
\qquad 
x_0 \in [0,1].
\end{equation}
These functions represent, in a certain limit, the trajectory of regular polygonal vortex filaments that evolve according to the binormal flow.
When $x_0$ is rational,  
we show that $R_{x_0}$ is multifractal and intermittent
by completely determining the spectrum of singularities of $R_{x_0}$ 
and computing the $L^p$ norms of its Fourier high-pass filters, which are analogues of structure functions. 
We prove that $R_{x_0}$ has a multifractal behavior 
also when $x_0$ is irrational. 
The proofs rely on a careful design of Diophantine sets that depend on $x_0$, which we study by 
using the Duffin-Schaeffer theorem and the Mass Transference Principle. 
\end{abstract}

\maketitle


\section{Introduction}

Multifractality and intermittency are among the main properties expected in turbulent flows
but, as usual in the theory of turbulence, it is challenging to analyze them rigorously. 
The motivation of this article is to quantify the multifractal and intermittent behavior of regular polygonal vortex filaments that evolve with the binormal flow.  
This evolution is represented, in a certain limit, by the function 
$R_{x_0}: \mathbb R \to \mathbb C$ defined by 
\begin{equation}\label{eq:R_x0}
R_{x_0}(t) = \sum_{n \neq 0} \frac{e^{2\pi i ( n^2 t + n x_0 ) }   }{n^2},
\end{equation} 
for $x_0 \in [0,1]$ fixed.
This function is one of the possible generalizations of the classic Riemann's non-differentiable function, 
which is recovered when $x_0 = 0$, 
and it can also be seen as the solution to a periodic Cauchy problem for the free Schr\"odinger equation. 
In this article we study the multifractality and intermittency of $R_{x_0}$, which until now was unknown for $x_0 \neq 0$:
\begin{itemize}
	\item When $x_0 \in \mathbb Q$, 
	 we completely describe the multifractality of $R_{x_0}$ by computing its spectrum of singularities (Theorem~\ref{thm:Main_Theorem_Rationals_Spectrum}).
	 We also compute the $L^p$ norms of its Fourier high-pass filters to deduce its intermittency exponents (Theorem~\ref{thm:Main_Theorem_Rationals_Intermittency}) and show that $R_{x_0}$ is intermittent.
	\item When $x_0 \not\in\mathbb Q$, we give a result that proves multifractality (Theorem~\ref{thm:Main_Theorem_Irrationals}) and strongly suggests that the spectrum of singularities depends on the irrationality of $x_0$, 
	and hence that it is different from when $x_0 \in \mathbb Q$.  
	\end{itemize}
The main novelty in this article is a careful design of Diophantine sets 
and the use of  the Duffin-Schaeffer theorem and the Mass Transference Principle to compute their measure and dimension.
When $x_0 \in \mathbb Q$, we use the partial Duffin-Schaeffer theorem as proved by Duffin and Schaeffer in \cite{DuffinSchaeffer1941}, 
while when $x_0 \not\in\mathbb Q$ we need the full strength of the theorem as proved by Koukoulopoulos and Maynard \cite{KoukoulopoulosMaynard2020}.
We give an overview of these arguments in Section~\ref{sec:Diophantine_Approximation}.
Before that, 
we introduce the concepts of multifractality and intermittency in Section~\ref{sec:MultifractalityAndIntermittency}, 
we discuss the connection of $R_{x_0}$ and vortex filaments in Section~\ref{sec:VFE}
and we state our results in Sections~\ref{sec:Setup} and \ref{sec:Results}.

\subsection{Multifractality and intermittency}
\label{sec:MultifractalityAndIntermittency}

The concepts of multifractality and intermittency arise in 
the study of three dimensional turbulence of fluids and waves, 
both characterized by low regularity and a chaotic behavior. 
These are caused by an energy cascade by which the energy injected in large scales is transferred to small scales. 
In this setting, large eddies constantly split in smaller eddies, 
generating sharp changes in the velocity magnitude. 
Moreover, this cascade is not expected to be uniform in space, 
and the rate at which these eddies decrease depends on their location.

Mathematically speaking, an option to measure the irregularity of the velocity $v$ is to compute the local H\"older regularity, 
that is, the largest $\alpha = \alpha(x)$ such that $|v(x + h) - v(x)| \lesssim |h|^\alpha$ when $|h| \to 0$. 
The lack of uniformity in space suggests that the H\"older level sets $D_\alpha = \{ \, x \, : \, \alpha(x) = \alpha \,  \}$ should be non-empty, and of different size, for many values of $\alpha$.
In this context, the spectrum of singularities is defined as $d(\alpha) = \operatorname{dim}_{\mathcal H} D_\alpha$, 
where $\operatorname{dim}_{\mathcal H}$ is the Hausdorff dimension,
and the velocity $v$ is said to be \textbf{multifractal} if $d(\alpha)$ takes values in multiple H\"older regularities $\alpha$. 

On the other hand, intermittency is a measure of the likelihood of localized bursts or outlier events.  
One way to quantify it is by analyzing the structure functions 
$S_p(h) = \langle |v(x+h) - v(x)|^p \rangle$ 
of the velocity when the scale $h$ tends to zero. 
More precisely, defining the flatness as
\begin{equation}\label{eq:Flatness}
F_4(h) = \frac{S_4(h)}{S_2(h)^2}, \qquad \text{ for very small } h, 
\end{equation}
we have \textbf{small-scale intermittency}\footnote{
Proposed by Frisch \cite[p.122, (8.2)]{Frisch1995} and Anselmet et al.
 \cite{AnselmetGagneHopfingerAntonia1984}.
} if $\lim_{h \to 0} F_4(h) = +\infty$. 
Assuming the typical power law 
\begin{equation}\label{eq:Zeta_p}
S_p(h) 
\simeq |h|^{\zeta_p}, 
\end{equation} 
it is usual to rephrase the definition of intermittency as $\zeta_4 - 2\zeta_2 < 0$ for the intermittency exponent\footnote{
In this setting, intermittency is regarded as a nonlinear correction to Kolmogorov's theory
(see \cite[Section 2.4]{BuckmasterVicol})
which predicted the exponents $\zeta_p$ to be a linear function of $p$ and hence $\zeta_4 - 2\zeta_2 = 0$ and, in general, $\zeta_p - p\zeta_2/2 = 0$.
}
$\zeta_p$. 
This definition, and in particular \eqref{eq:Flatness},
 is inspired by the probabilistic concept of kurtosis\footnote{The fourth standardized moment, sometimes also referred to as \textit{tailedness}.},
which quantifies how large the tails of the underlying probability distribution are.
A large kurtosis implies fat tails, which suggests that outlier events are more likely than for a normal distribution,  
agreeing with the widespread idea of non-Gaussianity. 
More generally, moments $F_p(h) = S_p(h)/S_2(h)^{p/2}$ of order $p \geq 4$ can be used to measure the tails of a probability distribution (see \cite[p.124]{Frisch1995}) and therefore intermittency, 
so it is common in recent physics literature to measure $\zeta_p$ for different $p$ 
(see \cite{RicardFalcon2023} and references therein, also \cite{ApolinarioChevillardMourrat2022} for a numeric intermittent model). 
The intermittency condition is then rewritten as $\zeta_p - p \zeta_2 /2 < 0$, a behavior that corresponds to a sublinear $\zeta_p$.

\subsection{$R_{x_0}$ as the trajectory of polygonal vortex filaments}
\label{sec:VFE}

The binormal flow is a model introduced by Da Rios\footnote{
Explored also by Levi-Civita in \cite{LeviCivita1932}.
} 
in 1906 \cite{DaRios1906} as an approximation to the evolution of a vortex filament according to Euler equation
and whose validity has been precisely and rigorously described theoretically by Fontelos and Vega in \cite{FontelosVega2023}
 in the setting of the Navier-Stokes equations.
This model describes the motion of the filament $\boldsymbol{X}: \mathbb R \times \mathbb R \to \mathbb R^3$, $\boldsymbol{X} = \boldsymbol{X}(x,t)$ by the equation $\boldsymbol{X}_t = \boldsymbol{X}_x \times \boldsymbol{X}_{xx}$. 
Inspired by Jerrard and Smets \cite{JerrardSmets2015},
De la Hoz and Vega \cite{DelaHozVega2014} observed numerically that
if the initial filament 
$\boldsymbol{X}_M(x,0)$ 
is a regular polygon with $M$ corners at the integers $x \in \mathbb Z$, 
then the trajectory of the corners 
$\boldsymbol{X}_M(0,t)$ 
is a plane curve which, identifying the plane with $\mathbb C$ and when $M$ is large, looks like
\begin{equation}\label{eq:Riemann_Simplification}
\phi(t) = \sum_{n \in \mathbb Z} \frac{e^{2\pi i n^2 t} - 1}{n^2} = 2\pi it - \frac{\pi^2}{3} + R_0(t).
\end{equation}
Moreover, 
let $\boldsymbol{\chi}_M(x,0)$ be an infinite polygonal line that loops the polygon of $M$ sides a finite but large number of times and ends in two half-lines, symmetrized at $x=0$.
Banica and Vega rigorously proved in \cite{BanicaVega2022} that, under certain hypotheses, its binormal flow evolution $\boldsymbol{\chi}_M(x,t)$ obtained in \cite{BanicaVega2020} satisfies 
\begin{equation}\label{eq:Phi_x0_Definition}
\lim_{M \to \infty} M \, \boldsymbol{\chi}_M(x_0,t) 
= \phi_{x_0}(t) 
:= \sum_{n \in \mathbb Z} \frac{e^{2 \pi i n^2 t} - 1}{n^2} \, e^{2\pi i n x_0}, \qquad \forall x_0 \in [0,1].
\end{equation}
We show in Figures~\ref{fig:Figures} and \ref{fig:TrajectoriesTogether} 
the image of $\phi_{x_0}$ for some values of $x_0$. 
Like in \eqref{eq:Riemann_Simplification}, 
noticing that the Fourier series $\displaystyle{\sum_{n \neq 0} \frac{ e^{2\pi i n x} }{n^2}}$ is $2\pi^2 \left(x^2 - x + \frac16 \right)$,
we can write
\begin{equation}
\phi_{x_0}(t) =2\pi i t - 2 \pi^2 \Big(x_0^2 - x_0  + \frac{1}{6}\Big) + R_{x_0}(t), 
\end{equation} 
which shows that $\phi_{x_0}$ and $R_{x_0}$ have the same regularity as functions of $t$. 
In other words, 
$R_{x_0}$ captures the regularity of the limit trajectory of polygonal vortex filaments 
that evolve with the binormal flow. 
This connection motivates us to study the multifractality and intermittency of $R_{x_0}$.

\begin{figure}
\begin{subfigure}{0.24\textwidth}
\begin{center}
\includegraphics[width=0.8\linewidth]{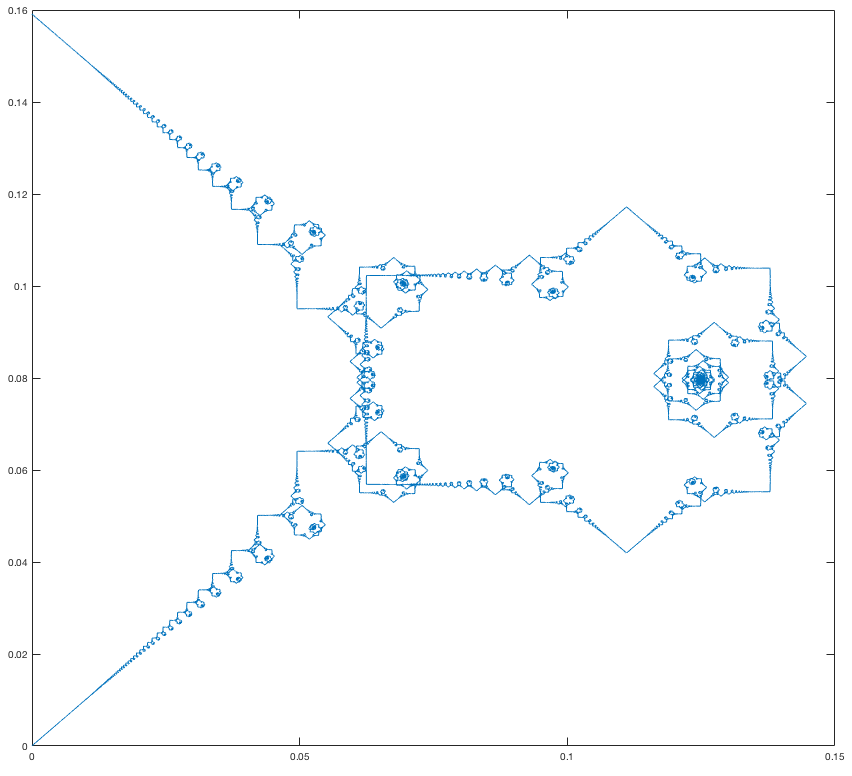}
\end{center}
\caption{$x_0 = 0$}
\end{subfigure}
\hfill
\begin{subfigure}{0.24\textwidth}
\includegraphics[width=\linewidth]{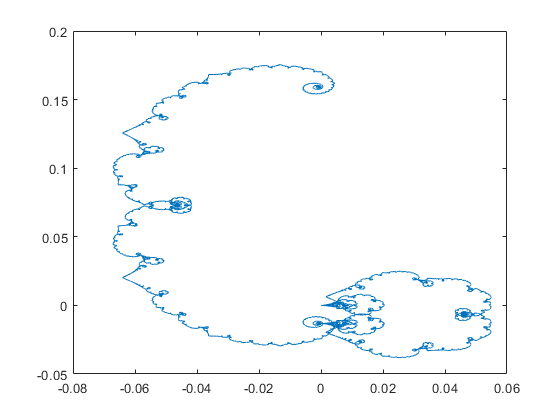}
\caption{$x_0 = 1/4$}
\end{subfigure}
\hfill
\begin{subfigure}{0.24\textwidth}
\includegraphics[width=\linewidth]{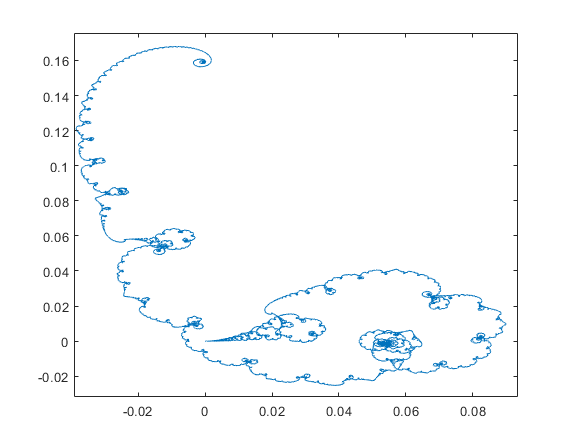}
\caption{$x_0 = 1/7$}
\end{subfigure}
\hfill
\begin{subfigure}{0.24\textwidth}
\includegraphics[width=\linewidth]{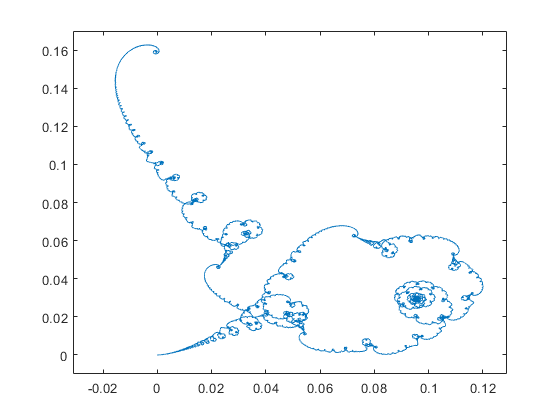}
\caption{$x_0 = 1/17$}
\end{subfigure}
\caption{Image of $\phi_{x_0}$, $t \in [0,1]$,
defined in \eqref{eq:Phi_x0_Definition}, 
 for some values of $x_0$.}
\label{fig:Figures}
\end{figure}

\begin{figure}
\includegraphics[scale=0.3]{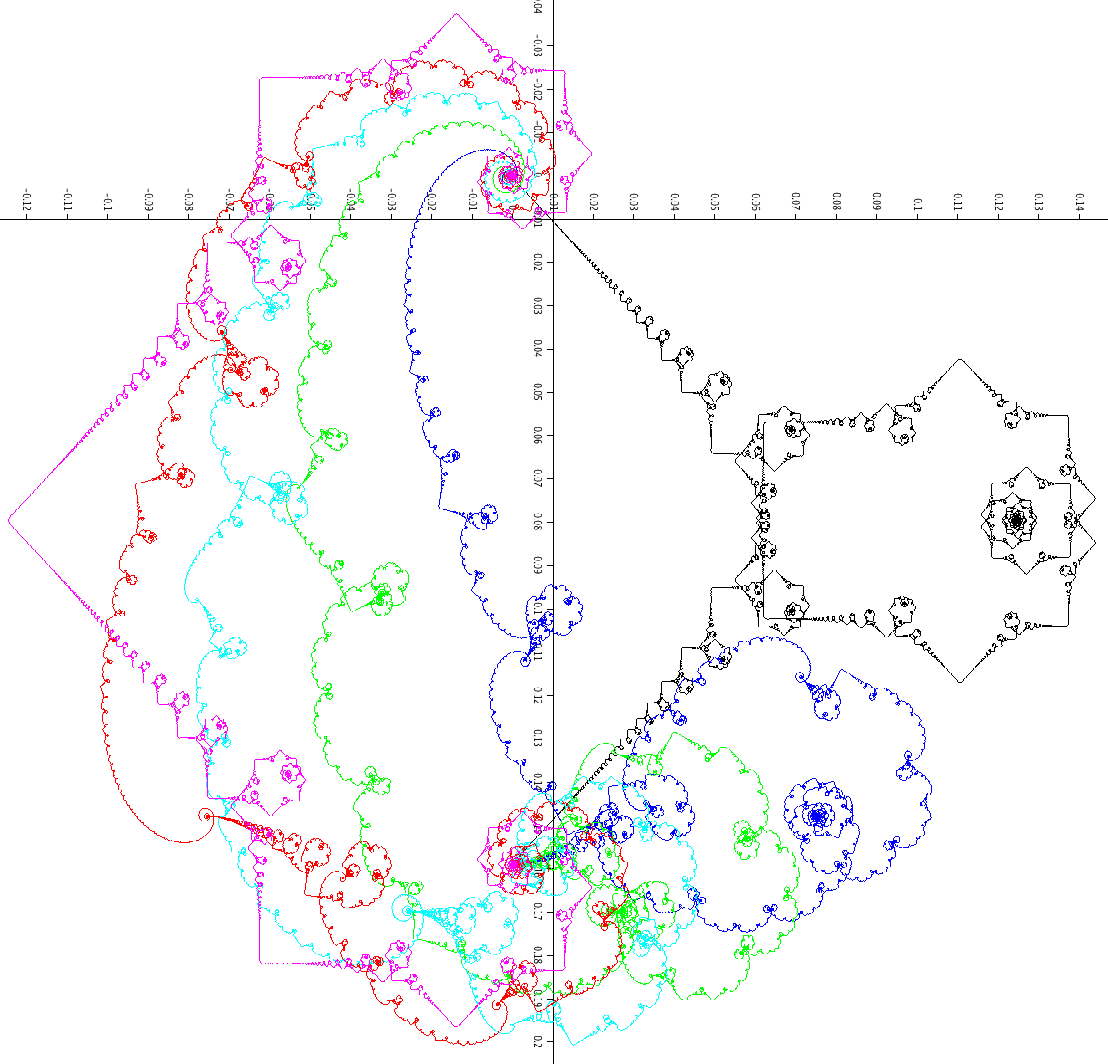}
\caption{The images of $\phi_{x_0}$, $t \in [0,1]$, for the values $x_0 = 0, 0.1, 0.2, 0.3, 0.4, 0.5$, 
from the rightmost to the leftmost.
}
\label{fig:TrajectoriesTogether}
\end{figure}

\subsection{Definitions and notation}\label{sec:Setup}

We now rigorously define the concepts discussed above. 

\subsubsection{Holder regularity}
 A function $f: \mathbb R \to \mathbb C$ is $\alpha$-H\"older at $t \in \mathbb R$, 
which we denote by $f \in \mathcal C^\alpha(t)$, if there exists a polynomial $P_t$ of degree at most $\alpha$ such that $ |f(t+h) - P_t(h)| \leq C |h|^\alpha$ for some constant $C > 0$ and for $h$ small enough.
In particular, if $0 < \alpha < 1$, the definition above becomes
\begin{equation}
f \in \mathcal C^\alpha(t) \quad \Longleftrightarrow \quad |f(t+h) - f(t)| \leq C |h|^\alpha, \quad \text{ for } h \text{ small enough}. 
\end{equation}
The local H\"older exponent of $f$ at $t$ is
$\alpha_f(t) = \sup \{  \,  \alpha \, : \, f \in \mathcal C^\alpha(t) \,   \}$.
We say $f$ is globally $\alpha$-H\"older if $f \in \mathcal C^\alpha(t)$ for all $t \in \mathbb R$.

\subsubsection{Spectrum of singularities}
The spectrum of singularities of $f$ is
\begin{equation}
d_f(\alpha) = \operatorname{dim}_{\mathcal H} \{ \, t \, : \, \alpha_f(t) = \alpha \, \}, 
\end{equation}
where $\operatorname{dim}_{\mathcal H}$ is the Hausdorff dimension\footnote{See 
\cite[Sections 3.1-3.2]{Falconer2014} for definitions and basic properties of Hausdorff measures and the Hausdorff dimension.}, 
and convene that $d(\alpha) = -\infty$ if $\{ \, t \, : \, \alpha_f(t) = \alpha \, \} = \emptyset$.

\subsubsection{Intermittency exponents}
 As discussed in \eqref{eq:Zeta_p}, the exponents $\zeta_p$ of the structure functions $S_p(h)$ describe the behavior of the increments of functions in small scales. 
Here we take the analogous approach of studying the 
high-frequency behavior of functions. 
Let $\Phi \in C^\infty (\mathbb R)$ be a cutoff function such that
$\Phi(x) = 0$ in a neighborhood of the origin 
and $\Phi(x) = 1$ for $|x| \geq 2$.  
For a periodic function $f$ with Fourier series $f(t) = \sum_{n \in \mathbb Z} a_n e^{2\pi i n t}$, 
define the
high-pass filter by 
\begin{equation}
P_{\geq N} f(t) = \sum_{n \in \mathbb Z} \Phi\Big(  \frac{n}{N} \Big) \,  a_n \, e^{2\pi i n t }, 
\qquad
N \in \mathbb N. 
\end{equation}
We treat the $L^p$ norms $\lVert P_{\geq N} f \rVert_p^p$ as the analytic and Fourier space analogues of the structure functions\footnote{We may think of the small scale $h$ to be represented by $1/N$, where $N$ is the frequency parameter}. 
Our analogous to the power law \eqref{eq:Zeta_p} is\footnote{
The heuristic exponent $\zeta_p$ in \eqref{eq:Zeta_p} and $\eta(p)$ defined in \eqref{eq:Eta_p_Intro} are a priori different. 
However, the definition of $\zeta_p$ can be made rigorous using $L^p$ norms so that it is equal to $\eta(p)$, as shown by Jaffard in \cite[Prop. 3.1]{Jaffard1997_1}.  
The exponent $\eta(p)$ is actually related to the Besov regularity of $f$.
Assuming $\lVert  P_{\geq N} f \rVert_p \simeq \lVert  P_{\simeq N} f \rVert_p$ (which is the case for $R_{x_0}$), where $P_{\simeq N} f$ denotes the band-pass filter defined with the cutoff $\Phi$ with the additional assumption of compact support, then $\eta(p) = \sup \{ \, s \, : \, f \in B^{s/p}_{p,\infty} \}$, 
where $f \in B^s_{p,q}$ if and only if $(2^{ks}\lVert P_{\simeq 2^k} f \rVert)_k \in \ell^q$. 
}

\begin{equation}\label{eq:Eta_p_Intro}
\eta_f(p) = \liminf_{N \to \infty} \frac{\log (\lVert  P_{\geq N} f \rVert_p^p ) }{\log (1/N)},
\end{equation}
which means that for any $\epsilon > 0$ we have
$ \lVert  P_{\geq N} f \rVert_p^p \leq  N^{-\eta_f(p) + \epsilon}$
for $N \gg_\epsilon 1$, 
and that this is optimal
in the sense that there is a subsequence $N_k \to \infty$ such that 
$ \lVert  P_{\geq N_k} f \rVert_p^p \geq  N_k^{-\eta_f(p) - \epsilon}$
for $k \gg_\epsilon 1$. 
We define the \textbf{$\boldsymbol{p}$-flatness} to be
\begin{equation}
F_p(N) = \frac{\lVert P_{\geq N} f \rVert_p^p}{\lVert P_{\geq N} f \rVert_2^p}, \qquad N \gg 1.
\end{equation}
The corresponding intermittency exponent\footnote{If the liminf in \eqref{eq:Eta_p_Intro} is a limit, then $\lVert P_{\geq N} f \rVert_p^p \simeq N^{-\eta_p}$ and hence $F_p(N) \simeq N^{-(\eta_f(p) - p\eta_f(2) /2 )}$.}  is $\eta_f(p) - p\, \eta_f(2)/2$.

\subsection{Results}\label{sec:Results}
To simplify notation, let us denote 
$\alpha_{R_{x_0}}(t) = \alpha_{x_0}(t)$, $d_{R_{x_0}}(\alpha) = d_{x_0}(\alpha)$
and $\eta_{R_{x_0}}(p) = \eta_{x_0}(p)$
for our function $R_{x_0}$ defined in \eqref{eq:R_x0}.

Since Weierstrass  \cite{Weierstrass1872} announced\footnote{
Weierstrass announced $R(t) = \sum_{n=1}^\infty \sin (n^2 t)/n^2$;
$R_0(t) = \sum_{n \neq 0}^\infty e^{2\pi i n^2 t}/n^2$ can be seen as its imaginary part.
} Riemann's non-differentiable function as the first candidate of a continuous and non-differentiable function in 1872, 
the regularity of $R_0$ has been studied by several authors.   
After Hardy \cite{Hardy1916} and Gerver \cite{Gerver1970,Gerver1971} proved that it is only almost nowhere differentiable (see also the simplified proof of Smith \cite{Smith1972}), Duistermaat \cite{Duistermaat1991} launched the study of its H\"older regularity. 
Jaffard completed the picture in his remarkable work \cite[Theorem 1]{Jaffard1996} (see also \cite{BrouckeVindas2021} for a recent alternative proof) by computing 
\begin{equation}\label{eq:Holder_For_Riemann}
\alpha_0(t) = \frac12 + \frac{1}{2\widetilde{\mu}(t)}, \qquad \text{ for } t \not\in \mathbb Q, 
\end{equation}
where
$\widetilde \mu(t)$ is the exponent of irrationality of $t$ restricted to denominators $q \not\equiv 2 \text{ (mod } 4)$\footnote{
Precisely, 
$\widetilde\mu(t) = \sup \{     \mu > 0 :  \big| t - \frac{p}{q} \big| \leq q^{-\mu} \text{ for infinitely many coprime pairs } (p,q) \in \mathbb N^2 \, 
\text{ with } q_n \not\equiv 2 \pmod{4}   \}.$
}. 
He combined this with an adaptation of the Jarn\'ik-Besicovitch theorem to prove
\begin{equation}\label{eq:Spectrum_At_0}
d_0(\alpha) = 
\left\{  
\begin{array}{ll}
4\alpha - 2, & 1/2 \leq \alpha \leq 3/4, \\
0, & \alpha = 3/2, \\
-\infty, & \text{ otherwise.}
\end{array}
\right. 
\end{equation}
Our first results concern the spectrum of singularities of $R_{x_0}$ for $x_0 \neq 0$. 
\begin{thm}\label{thm:Main_Theorem_Rationals_Spectrum}
Let $x_0 \in \mathbb Q$.
Then, 
\begin{equation}
d_{x_0}(\alpha) = 
\left\{
\begin{array}{ll}
4  \alpha - 2, & 1/2 \leq \alpha \leq 3/4, \\
0, & \alpha = 3/2, \\
-\infty, & \text{otherwise.}
\end{array} \right.
\end{equation}
\end{thm}

\begin{rem} 
\hfill
\begin{enumerate}
	\item[(a)] To prove Theorem~\ref{thm:Main_Theorem_Rationals_Spectrum},	
we adapt the classical approach due to Duistermaat \cite{Duistermaat1991} and Jaffard \cite{Jaffard1996} 
by carefully choosing subsets of the irrationals with novel Diophantine restrictions to disprove H\"older regularities.
However, 
the arguments in \cite{Jaffard1996} to compute their Hausdorff dimension do not suffice\footnote{The restriction for denominators in the case $x_0=0$ is essentially a parity condition, which is solved in \cite{Jaffard1996} by dividing the set by the factor 2. This does not generalize to the case $x_0 = P/Q$ where the condition for the denominator will be to be a multiple of $4Q$.}
 when $x_0 \neq 0$. 
We solve this by using the Duffin-Schaeffer theorem and the Mass Transference Principle; see Section~\ref{sec:Diophantine_Approximation} for the outline of the argument.
	\item[(b)] Even if $d_{x_0} = d_0$ for all $x_0 \in \mathbb Q$, we think that $\alpha_{x_0}(t) \neq \alpha_0(t)$.	
	However, Theorem~\ref{thm:Main_Theorem_Rationals_Spectrum} does not require computing $\alpha_{x_0}(t)$ for all $t \in \mathbb R$. 
	A full description of the sets $\{ \, t \, : \, \alpha_{x_0}(t) = \alpha \, \}$ is an interesting and challenging problem because when $x_0 \neq 0$ it is not clear how to characterize the H\"older regularity $\alpha_{x_0}(t)$ in terms of some irrationality exponent like in \eqref{eq:Holder_For_Riemann}.
	We do not pursue this problem here, which we leave for a future work. 
\end{enumerate}
\end{rem}

Let now $x_0 \not\in \mathbb Q$.
Let $p_n/q_n$ be its approximations by continued fractions, 
and define the exponents $\mu_n$ by $|x_0 - p_n/q_n| = 1/q_n^{\mu_n}$. 
Define the alternative\footnote{The usual exponent of irrationality is $\mu(x_0) = \limsup_{n \to \infty} \mu_n$.} exponent of irrationality 
\begin{equation}\label{eq:Exponent_Of_Irrationality_Alternative}
\sigma(x_0) = \limsup_{n \to \infty} \,  \{ \, \mu_n \, : \, q_n 
\not\in 4\mathbb N
\,   \}. 
\end{equation}
This exponent always exists and $\sigma(x_0) \geq 2$ (see Proposition~\ref{thm:Dimension_X_Mu}). 
Our result is the following.
\begin{thm}\label{thm:Main_Theorem_Irrationals}
Let $x_0 \not\in \mathbb Q$. 
Let $2 \leq \mu < 2\sigma(x_0)$, 
with $\sigma(x_0)$ as in \eqref{eq:Exponent_Of_Irrationality_Alternative}.
Then, for all $\delta > 0$, 
\begin{equation}\label{eq:Main_Theorem_Irrationals}
\frac{1}{\mu} \leq \operatorname{dim}_{\mathcal H} \left\{   \, t \, : \frac12 + \frac{1}{4\mu} - \delta  \leq \alpha_{x_0}(t) \leq \frac12 + \frac{1}{2\mu}    \right\}  \leq \frac{2}{\mu}. 
\end{equation}
\end{thm}
\begin{rem} 
\begin{enumerate}
	\item[(a)] We show in Figure~\ref{fig:Thm1.3} a graphic representation of Theorem~\ref{thm:Main_Theorem_Irrationals}.
	\item[(b)] Theorem~\ref{thm:Main_Theorem_Irrationals}  
	shows that $R_{x_0}$ is multifractal when $\sigma(x_0) > 2$. 
	\item[(c)] Theorem~\ref{thm:Main_Theorem_Irrationals} would be strengthened to $1/\mu \leq d_{x_0}(1/2 + 1/2\mu) \leq 2/\mu$ for $\mu < 2\sigma(x_0)$ if we could compute the dimension of some well-identified Diophantine sets, see Remark~\ref{rmk:Good_Diophantine_Set}. 
	This would give a nontrivial spectrum of singularities in an open interval for all $x_0 \not\in \mathbb Q$. 
	We leave this for a future work. 
	
	\item[(d)] The reasons to have an interval $(1/\mu, 2/\mu)$ for the dimension in \eqref{eq:Main_Theorem_Irrationals} seem to us deeper in nature. 
Unlike the upper bound  $2/\mu$, which follows from approximating $t$ with rationals $p/q$ with unrestricted $q\in \mathbb N$ and with error $q^{-\mu}$ (see the Jarn\'ik-Besicovitch theorem~\ref{thm:Jarnik_Besicovitch}), 
the lower bound depends on the nature of $x_0$ which imposes restrictions to $q$. 
When $x_0 = P/Q \in \mathbb Q$, we require $q \in 4Q\mathbb N$, which still results in a set of dimension $2/\mu$.
However, when $x_0 \not\in \mathbb Q$ we require $q$ be restricted to an exponentially growing sequence (given by the denominators of the continued fraction approximations of $x_0$). This restriction is much stronger and gives a set of $t$ of dimension $1/\mu$. 
These results follow from the Duffin-Schaeffer theorem and the Mass Transference Principle.

	\item[(e)] The theorem and its proof (see the heuristic discussion in Section~\ref{sec:Heuristics}) suggest that the spectrum of singularities may be $d_{x_0}(\alpha) = 4\alpha - 2$ in the range 
	$\frac12 + \frac{1}{4\sigma(x_0)} \leq \alpha \leq \frac34$, 
	and possibly something different outside of this range.
	In particular, we expect the segment of the spectrum in $5/8 \leq \alpha \leq 3/4$ to be present for all $x_0$. 
\end{enumerate}
\end{rem}

\begin{rem}
Our results suggest that the trajectories of the binormal flow do not have a generic behavior in terms of regularity. 
Indeed, if $X_n$ is a sequence of independent and identically distributed complex Gaussian random variables,
then the random function
\begin{equation}\label{eq:RandomSeries}
S(t) = \sum_{n=1}^\infty X_n \frac{e^{2\pi i n^2 t}}{n^2}
\end{equation}
has\footnote{
\cite[p.86, Theorem 2]{Kahane1985}  shows that almost surely $\alpha_S(t) \geq 3/4$ for all $t$, and
and \cite[p.104, Theorem 5]{Kahane1985} shows that almost surely $\alpha_S(t) \leq 3/4$ for all $t$. 
}  almost surely $\alpha_S(t) = 3/4$ for all $t \in \mathbb R$ \cite{Kahane1985}. 
Hence the generic behavior of \eqref{eq:RandomSeries} is monofractal. 
In contrast, the fine structure of the linear phase $nx_0$ of $R_{x_0}$ causes a multifractal behavior.  
\end{rem}

Regarding intermittency, we compute the $L^p$ norms of the Fourier high-pass filters of $R_{x_0}$
and the intermittency exponents $\eta_{x_0}(p)$ 
when $x_0 \in \mathbb Q$,
 from which we deduce that $R_{x_0}$ is intermittent. 
\begin{thm}\label{thm:Main_Theorem_Rationals_Intermittency}
Let $x_0 \in \mathbb Q$. Let $1 < p < \infty$. Then, 
\begin{equation}\label{eq:Main_Theorem_Rationals_Lp}
\big\lVert  P_{\geq N} R_{x_0} \big\rVert_p^p \simeq 
\left\{  \begin{array}{ll}
N^{-\frac{p}{2} - 1}, & p > 4, \\
N^{-3} \, \log N, & p = 4, \\
N^{-3p/4}, & p < 4, 
\end{array}
\right.
\end{equation}
and therefore 
\begin{equation}
\eta_{x_0}(p)
= \lim_{N \to \infty} \frac{\log (\lVert  P_{\geq N} f \rVert_p^p ) }{\log (1/N)}
= \left\{ \begin{array}{ll}
p/2 + 1, & p > 4, \\
3p/4, & p \leq 4.
\end{array}
\right.
\end{equation}
Consequently, $\lim_{N \to \infty} F_p(N) = +\infty$ for $p \geq 4$. 
In particular, $R_{x_0}$ is intermittent. 
\end{thm}

\begin{rem} \hfill
\begin{enumerate} 
\item[(a)] The $p=4$ intermittency exponent in \eqref{eq:Main_Theorem_Rationals_Lp} is $\eta(4) - 2\eta(2) = 0$, 
but the fact that $\lVert  P_{\geq N} R_{x_0}\rVert_4^4$ does not follow a pure power law makes
$F_4(N) \simeq \log N$. 
For $p > 4$, we have $\eta(p) - p\eta(2)/2  = 1 -  p/4 < 0$, 
so $R_{x_0}$ is intermittent in small scales when $x_0 \in \mathbb Q$. 

\item[(b)] The upper bound in  \eqref{eq:Main_Theorem_Rationals_Lp} in Theorem~\ref{thm:Main_Theorem_Rationals_Intermittency} holds for all $x_0 \in [0,1]$. The theorem shows that  this is optimal when $x_0 \in \mathbb Q$, but we do not expect it to be optimal when $x_0 \not\in \mathbb Q$.
	We suspect that the exact behavior, and hence $\eta_{x_0}(p)$, depends on the irrationality of $x_0$. 
	We aim to study this question in a future work.
	
\end{enumerate}
\end{rem}

\begin{figure}[h]
\includegraphics[width=0.55\linewidth]{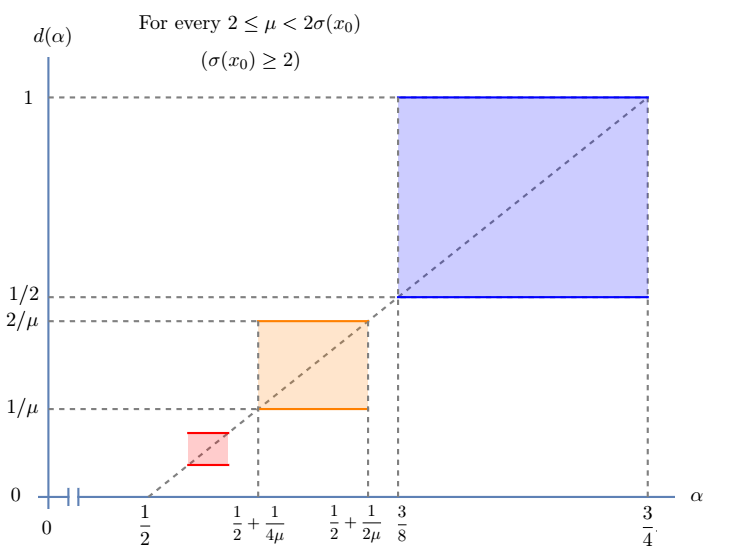}
\caption{A graphic representation of Theorem~\ref{thm:Main_Theorem_Irrationals}. 
We have a continuum of Whitney-type boxes parametrized by $\mu$ along the dashed diagonal line $d(\alpha) = 4\alpha - 2$. 
The graph of $d_{x_0}(\alpha)$ has at least a point in each of those boxes. }
\label{fig:Thm1.3}
\end{figure}

\subsection{Related literature on the analytic study of Riemann's non-differentiable function}\label{sec:AnalyticLiterature}
Beyond the literature for the original Riemann's function $R_0$, 
the closest work to the study of $R_{x_0}$ is by Oskolkov and Chakhkiev \cite{OskolkovChakhkiev2010}. 
They studied the regularity of $R_{x_0}(t)$ 
almost everywhere as a function of two variables $(x_0,t)$,
which is not fine enough to capture multifractal properties.   

Alternatively, 
there are many works studying $R_{x_0}(t)$ as a function of $x_0$ with $t$ fixed,
motivated by the fact that $R_{x_0}$ is the solution to an initial value problem for the periodic free Schr\"odinger equation. 
From this perspective, 
Kapitanski and Rodnianski \cite{KapitanskiRodnianski1999} studied the Besov regularity of the fundamental solution\footnote{
Which, up to constants, is either $\partial_t R_{x_0}(t)$ or $\partial_{x_0}^2R_{x_0}(t)$.} as a function of $x$ with $t$ fixed. 
This approach is also intimately related to the Talbot effect in optics which, 
as proposed by Berry and Klein \cite{BerryKlein1996}, 
is approximated by the fundamental solution to the periodic free Schr\"odinger equation. 
Pursuing the related phenomenon of \textit{quantization}\footnote{See the article by Olver \cite{Olver2010} for an instructive account of quantization.},
the geometry of the profiles of Schr\"odinger solutions have been studied for fixed $t$ by Berry \cite{Berry1996} and Rodnianski \cite{Rodnianski2000}.
Following the numeric works of Chen and Olver \cite{ChenOlver2013,ChenOlver2014},
this perspective has also been extended to the nonlinear setting and other dispersive relations 
by Chousonis, Erdogan and Tzirakis \cite{ErdoganTzirakis2013,ChousionisErdoganTzirakis2015} 
and Boulton, Farmakis and Pelloni \cite{BoultonFarmakisPelloni2021,BoultonFarmakisPelloni2023}.

There is a literature for other natural generalizations of Riemann's function, like
\begin{equation}
F(t) = \sum_{n=1}^\infty \frac{e^{2\pi i P(n) t}}{n^\alpha}, 
\qquad P \text{ a polynomial}, 
\qquad \alpha > 1, 
\end{equation}
For $P(n) = n^2$, 
Jaffard \cite{Jaffard1996} gave his results for all $\alpha > 1$. 
Chamizo and C\'ordoba \cite{ChamizoCordoba1999} studied the Minkowski dimension of their graphs. 
Seuret and Ubis \cite{SeuretUbis2017} studied the non-convergent case $\alpha < 1$, 
using a local $L^2$ exponent.  
Chamizo and Ubis \cite{ChamizoUbis2007,ChamizoUbis2014} 
 studied the spectrum of singularities for general polynomials $P$.
 Further generalizations concerning fractional integrals of modular forms were studied by Pastor \cite{Pastor2019}.

\subsection{Structure of the article}\label{sec:Structure}

In Section~\ref{sec:Diophantine_Approximation} we discuss the  general strategy we follow to prove our theorems, stressing the new ideas related to Diophantine sets with restrictions, the Duffin-Schaeffer theorem and the Mass Transference Principle. 
In Section~\ref{sec:Preliminary_Results} we prove 
 preliminary results for the local H\"older regularity of $R_{x_0}$, 
 in particular the behavior around rational points $t$.
In Section~\ref{sec:Rational_Holder} we compute the spectrum of singularities of $R_{x_0}$ when $x_0 \in \mathbb Q$ and prove Theorem~\ref{thm:Main_Theorem_Rationals_Spectrum}. 
In Section~\ref{sec:Irrationals} we prove Theorem~\ref{thm:Main_Theorem_Irrationals}. 
In Section~\ref{sec:Rational_High_Pass_Filters}
we prove Theorem~\ref{thm:Main_Theorem_Rationals_Intermittency} 
by computing the $L^p$ norms of the high-pass filters of $R_{x_0}$. 
The proofs of some auxiliary results are postponed to Appendices~\ref{sec:SumsEulerFunctions} and \ref{sec:Asymptotic_Alternative} to avoid breaking the continuity of the main arguments.

\section{An overview on the general arguments and on Diophantine approximation}\label{sec:Diophantine_Approximation}

\subsection{General argument} 
An important part of the arguments in this article relies on Diophantine approximation.
We will work with both the exponent of irrationality 
\begin{equation}\label{eq:Exponent_Of_Irrationality}
\mu(x) = \sup \Big\{ \, \mu > 0 \, : \Big| x - \frac{p}{q} \Big| \leq \frac{1}{q^\mu} \text{ for infinitely many coprime pairs } (p,q) \in \mathbb N \times \mathbb N \,   \Big\}, 
\end{equation}
and the Lebesgue and Hausdorff measure properties of the related sets 
\begin{equation}\label{eq:A_mu_Intro}
A_\mu = \Big\{ \, x \in \mathbb [0,1] \, \mid \, \Big| x - \frac{p}{q} \Big| \leq \frac{1}{q^\mu} \text{ for infinitely many coprime pairs } (p,q) \in \mathbb N \times \mathbb N \,   \Big\}, 
\end{equation}
where the case $\mu = \infty$ is understood as $A_\infty = \bigcap_{\mu \geq 2} A_\mu$. 
In a somewhat hand-waving way, $\mu(x) = \mu$ 
means that $|x - p/q| \simeq 1/q^\mu$ infinitely often, which ceases to be true for any larger $\mu$. 

With these concepts in hand, 
the classic way to study the regularity of $R_{x_0}$ (used by Duistermaat, Jaffard and subsequent authors) is to first compute the asymptotic behavior of $R_{x_0}$ around rationals. 
Using the Poisson summation formula
we will get a leading order expression of the form
\begin{equation}\label{eq:Heuristic_Asymptotic}
 R_{x_0}\Big( \frac{p}{q} + h \Big) - R_{x_0}\Big( \frac{p}{q}\Big)   \sim \frac{\sqrt{h}}{q}G_q \sim \frac{\sqrt{h}}{\sqrt{q}},
\end{equation}
where  $G_q$ includes a quadratic Gauss sum of period $q$, hence $|G_q| \sim \sqrt q$ whenever it does not cancel.
This shows that in most rationals the regularity of $R_{x_0}$ is 1/2. 
Let now $t\not\in \mathbb Q$ with irrationality exponent $\mu(t) = \mu$. 
Then, essentially $|t - p/q| \simeq 1/q^{\mu}$, 
so choosing $h = t-p/q$ we get 
 \begin{equation}
 R_{x_0}\Big( t \Big) - R_{x_0}\Big( t - h \Big) 
\sim \frac{\sqrt{h}}{\sqrt{q}}
\sim h^{\frac12 + \frac{1}{2\mu}}.
\end{equation}
This suggests that $\alpha_{x_0}(t) = \frac12 + \frac{1}{2\mu}$. 
Combining this with the Jarnik-Besicovitch theorem, which says that $\dim_{\mathcal H } A_\mu = 2/\mu$, we get the desired $d(\alpha) = 4\alpha - 2$ in the range $1/2 \leq \alpha \leq 3/4$. 

This argument is essentially valid up to assuming $G_q \neq 0$ in \eqref{eq:Heuristic_Asymptotic}. 
This, however, does not always hold. 
Apart from a parity condition on $q$ coming from the Gauss sums (present already in previous works),
an additional condition arises that depends on $x_0$. 
For example, if $x_0 = P/Q \in \mathbb Q$, this condition has the form of $Q \mid q$. 
In terms of the sets $A_\mu$, 
this means that we need to restrict the denominators of the approximations to a subset of the natural numbers. 
So let $\mathcal Q \subset \mathbb N$, and define
\begin{equation}\label{eq:A_Mu_Q}
A_{\mu, \mathcal Q}
= \Big\{ \, x \in [0,1] \, : \, \Big| x - \frac{p}{q} \Big| \leq \frac{1}{q^\mu} \text{ for infinitely many coprime pairs } (p,q) \in \mathbb N \times \mathcal Q  \,   \Big\}.
\end{equation}
Clearly $A_{\mu, \mathcal Q} \subset A_\mu$,
but a priori it could be much smaller. 
Does $A_{\mu, \mathcal Q}$ preserve the measure of $A_\mu$?
Previous works need to work with situations analogue to $Q=2$, 
but here we need to argue for all $Q \in \mathbb N$. 
For that, at the level of the Lebesgue measure
we will use the Duffin-Schaeffer theorem, 
while we will compute Hausdorff measures and dimensions 
 via the Mass Transference Principle.

\subsection{Lebesgue measure: Dirichlet approximation and the Duffin-Schaeffer theorem}

Both the Dirichlet approximation theorem 
and the theory of continued fractions
imply $A_2 = \mathbb [0,1] \setminus \mathbb Q$.
However, neither of them give enough information about the sequence of denominators they produce, 
so they cannot be used to determine the size of the set $A_{2,\mathcal Q  }  \subset A_2$.
The recently proved Duffin-Schaeffer conjecture gives an answer to this kind of questions. 
\begin{thm}[Duffin-Schaeffer theorem \cite{KoukoulopoulosMaynard2020}]\label{thm:Duffin_Schaeffer}
Let $\psi: \mathbb N \to [0,\infty)$ be a function. Define
\begin{equation}
A_{\psi}
= \Big\{ \, x \in [0,1] \, : \, \Big| x - \frac{p}{q} \Big| \leq \psi(q) \text{ for infinitely many coprime pairs } (p,q) \in \mathbb N \times \mathbb N  \,   \Big\}.
\end{equation}
Let $\varphi$ denote the Euler totient function\footnote{The Euler totient function:  for $q \in \mathbb N$, $\varphi(q)$ is the number of natural numbers $ i \leq q$ such that $\operatorname{gcd}(q,i)=1$.}. 
Then, we have the following dichotomy: 
\begin{enumerate}
	\item If $\sum_{q = 1}^\infty \varphi(q) \psi(q) = \infty$, then $|A_\psi| = 1$.  
	
	\item If $\sum_{q = 1}^\infty \varphi(q) \psi(q) < \infty$, then $|A_\psi| = 0$.
\end{enumerate}
\end{thm}
The relevant part of this theorem is $(a)$, since $(b)$ follows from the canonical limsup covering 
\begin{equation}\label{eq:Canonical_Cover_Limsup}
A_\psi \subset  \bigcup_{ q = Q }^\infty  \bigcup_{\substack{ 1 \leq p  \leq q \\ \, (p,q) = 1 }} \Big( \frac{p}{q} - \psi(q), \, \frac{p}{q} + \psi(q) \Big), \quad \forall \,  Q \in \mathbb N
\quad \Longrightarrow \quad 
|A_\psi| \leq \sum_{q = Q}^\infty \varphi(q) \psi(q), \quad \forall \, Q \in \mathbb N. 
\end{equation}
On the other hand, 
as opposed to the classic theorem by Khinchin\footnote{Khinchin's theorem states that if $\psi: \mathbb N \to [0,\infty)$ is a function such that $q^2\psi(q)$ is decreasing and $\sum_{q=1}^\infty q\,\psi(q) = \infty$, then the set $\{ \, x \in [0,1] \, : \, |x - p/q| \leq \psi(q) \text{ for infinitely many pairs } (p,q) \in \mathbb N \times \mathbb N \,  \}$ has Lebesgue measure 1. } \cite[Theorem 32]{Khinchin1964},
the arbitrariness of $\psi$ allows to restrict the denominators to a set $\mathcal Q \subset \mathbb N$ just by setting $\psi(q) = 0$ when $q \not\in \mathcal Q$. 
In particular, $A_{\mu,\mathcal Q} = A_\psi$ if we define $\psi(q) = \mathbbm{1}_{\mathcal Q}(q) / q^\mu$, where $\mathbbm{1}_{\mathcal Q}$ is the indicator function of the set $\mathcal Q$. 
Hence, the relevant sum for the sets $A_{\mu,\mathcal Q}$ is
\begin{equation}\label{eq:Our_Relevant_Sum}
\sum_{q=1}^\infty \varphi(q) \psi(q) = \sum_{q \in \mathcal Q} \frac{\varphi(q)}{q^\mu}.
\end{equation}
In particular, it is fundamental to understand the behavior of the Euler totient function $\varphi$ on $\mathcal Q$. 

The complete proof of the Duffin-Schaeffer theorem was given recently by Koukoulopoulos and Maynard 
\cite[Theorem 1]{KoukoulopoulosMaynard2020}, 
but Duffin and Schaeffer \cite{DuffinSchaeffer1941} proved back in 1941 that the result holds under the additional assumption that there exists $c > 0$ such that
\begin{equation}\label{eq:Duffin_Schaeffer_Condition}
\sum_{q=1}^N \varphi(q) \, \psi(q) \geq c \sum_{q=1}^N q \, \psi(q), \qquad \text{ for infinitely many } N \in \mathbb N. 
\end{equation} 
In the setting of $A_{\mu,\mathcal Q}$, 
this condition is immediately satisfied by sets $\mathcal Q$ for which there is a $c > 0$ such that $\varphi(q) > c \,  q$ for all $q \in \mathcal Q$. Examples of this are:
\begin{itemize}
	\item $\mathcal Q = \mathbb P$ the set of prime numbers, and
	\item $\mathcal Q = \{ \, M^n \, : \, n \in \mathbb N \,   \}$ where $M \in \mathbb N$, that is, the set of power of a given number $M$. 
\end{itemize} 
It follows from our computations in Appendix~\ref{sec:SumsEulerFunctions} that the condition \eqref{eq:Duffin_Schaeffer_Condition} is also satisfied by 
\begin{itemize}
	\item $\mathcal Q = \{ \, Mn \, : \, n \in \mathbb N \,   \}$ where $M \in \mathbb N$, that is,  the set of multiples of a given number $M$. 
\end{itemize}
To prove Theorem~\ref{thm:Main_Theorem_Rationals_Spectrum} for $x_0 = P/Q$, we restrict the denominators to the latter set with $M = 4Q$; in particular, the 1941 result by Duffin and Schaeffer  \cite{DuffinSchaeffer1941} suffices.  
However, in the case of $x_0 \not\in \mathbb Q$ we need to restrict the denominators to an exponentially growing sequence $q_n$ for which we do not know if \eqref{eq:Duffin_Schaeffer_Condition} holds. 
Hence, in this case we need the full power of the result by Koukoulopoulos and Maynard \cite{KoukoulopoulosMaynard2020}. 
This might give an indication of the difficulty to settle the case $x_0 \not\in \mathbb Q$.

\subsection{Hausdorff dimension: the Jarn\'ik-Besicovitch theorem and the Mass Transference Principle}

We mentioned that 
$A_2 = [0,1] \setminus \mathbb Q$, 
and it follows from the argument in \eqref{eq:Canonical_Cover_Limsup}
 that $|A_\mu| = 0$ for $\mu > 2$. 
But how small is $A_\mu$ is when $\mu > 2$?
A measure theoretic answer to that is the Jarn\'ik and Besicovitch theorem from the 1930s (see \cite[Section 10.3]{Falconer2014} for a modern version).
\begin{thm}[Jarn\'ik-Besicovitch theorem]\label{thm:Jarnik_Besicovitch}
Let $\mu > 2$ and let $A_\mu$ be defined as in \eqref{eq:A_mu_Intro}. 
Then, $\operatorname{dim}_{\mathcal H} A_\mu = 2/ \mu$ and $\mathcal H^{2/\mu}(A_\mu) = \infty$. 
\end{thm}

In this article we need to adapt this result to $A_{\mu,\mathcal Q}$. 
First, using the Duffin-Schaeffer Theorem~\ref{thm:Duffin_Schaeffer} we will be able to find the largest $\mu_0 \geq 1$ such that $|A_{\mu_0, \mathcal Q}| = 1$, 
so that $|A_{\mu, \mathcal Q}| = 0$ for all $\mu > \mu_0$. 
To compute the Hausdorff dimension of those zero-measure sets, 
we will use a theorem by Beresnevich and Velani, called the Mass Transference Principle \cite[Theorem 2]{BeresnevichVelani2006}. 
We state here its application to the unit cube and to Hausdorff measures. 
\begin{thm}[Mass Transference Principle \cite{BeresnevichVelani2006}]\label{thm:Mass_Transference_Principle}
Let $B_n = B_n(x_n,r_n)$ be a sequence of balls in $[0,1]^d$ 
such that $\lim_{n \to \infty} r_n = 0$. 
Let $\alpha < d$ and let $B_n^\alpha = B_n(x_n,r_n^{\alpha/d})$ be the dilation of $B_n$ centered at $x_n$ by the exponent $\alpha$. Suppose that $ X^\alpha := \limsup_{n \to \infty} B_n^\alpha $ is of full Lebesgue measure, that is, 
$|X^\alpha| = 1$. 
Then, calling $X := \limsup_{n \to \infty} B_n$, we have $\operatorname{dim}_{\mathcal H} X \geq \alpha$ 
and $\mathcal H^\alpha( X ) = \infty$. 
\end{thm}

To illustrate the power of the Mass Transference Principle, 
let us explain how the Jarnik-Besicovitch Theorem \ref{thm:Jarnik_Besicovitch} follows as a simple corollary of the Dirichlet approximation theorem. 
From the definition of $A_\mu$ we can write\footnote{The expression in \eqref{eq:A_Mu_Limsup} is not in the form of a limsup of balls.
It follows, however, that the limsup of any enumeration whatsoever of the balls considered in the construction gives the same set.}
\begin{equation}\label{eq:A_Mu_Limsup}
A_\mu = \limsup_{q \to \infty} \bigcup_{1 \leq p \leq q, (p,q) = 1} B\Big( \frac{p}{q}, \, \frac{1}{q^\mu} \Big). 
\end{equation}
Choose $\alpha = 2/\mu$ so that $(A_\mu)^\alpha = A_{\mu\alpha} =  A_2$, which by the Dirichlet approximation theorem has full measure. 
Then, the Mass Transference Principle implies
$\operatorname{dim}_{\mathcal H} A_\mu \geq 2/\mu$ and $\mathcal H^{2/\mu} (A_\mu) = \infty$. 
The upper bound follows from the canonical cover of $A_\mu$ in \eqref{eq:A_Mu_Limsup}, 
proceeding like in \eqref{eq:Canonical_Cover_Limsup}.
 
 For $A_{\mu,\mathcal Q}$,  
 once we find the largest $\mu_0$ for which $|A_{\mu_0,\mathcal Q}|=1$ using the Duffin-Schaeffer theorem,
we will choose $\alpha = \mu_0 / \mu$ so that the property $(A_{\mu,\mathcal Q})^\alpha = A_{\mu \alpha, \mathcal Q } = A_{\mu_0, \mathcal Q}$ has full measure, 
and the Mass Transference Principle will then imply $\operatorname{dim}_{\mathcal H} A_{\mu, \mathcal Q} \geq \mu_0/\mu$.


\section{Preliminary results on the local regularity of $R_{x_0}$}
\label{sec:Preliminary_Results}

In this section we carry over to $R_{x_0}$
regularity results that are by now classical for $R_0$. 
In Section \ref{sec:General_Lower_Bound} we prove that
$R_{x_0}$ is globally $C^{1/2}$.
In Section~\ref{sec:Asymptotic_Behavior_Rationals}
we compute the asymptotic behavior of $R_{x_0}$ around rationals. 
In Section~\ref{sec:Lower_Bounds} we give a lower bound for $\alpha_{x_0}(t)$ that is independent of $x_0$.

\subsection{A global H\"older regularity result}\label{sec:General_Lower_Bound}
Duistermaat \cite[Lemma 4.1.]{Duistermaat1991} proved that $R_0$ is globally $C^{1/2}(t)$. 
The same holds for all $x_0 \in \mathbb R$. 
We include the proof for completeness.
\begin{prop}\label{thm:DuistermaatLowerBound}
Let $x_0 \in \mathbb R$. Then, $\alpha_{x_0}(t) \geq 1/2$ for all $t \in\mathbb R$. That is, $R_{x_0}$ is globally $C^{1/2}$.
\end{prop}
\begin{proof}
For $h \neq 0$, let $N \in \mathbb N$ such that $\frac{1}{(N+1)^2} \leq  |h| < \frac{1}{N^2}  $, 
and write
\begin{equation}
R_{x_0}(t+h) - R_{x_0}(t) 
= \sum_{|n| \leq N} \frac{e^{2\pi i n^2 t}\, e^{2\pi i n x_0} }{n^2} \Big( e^{2\pi i n^2 h} - 1 \Big)
+ \sum_{|n| > N} \frac{e^{2\pi i n^2 t}\, e^{2\pi i n x_0} }{n^2} \Big( e^{2\pi i n^2 h} - 1 \Big).
\end{equation}
Since $|e^{ix} - 1| \leq |x|$ for all $x \in \mathbb R$, we bound
\begin{equation}
\Big| \sum_{|n| \leq N} \frac{e^{2\pi i n^2 t}\, e^{2\pi i n x_0} }{n^2} \Big( e^{2\pi i n^2 h} - 1 \Big) \Big|
\leq  \sum_{|n| \leq N} \frac{ \big| e^{2\pi i n^2 h} - 1 \big|}{n^2}
\leq 2|h|N 
< 2 |h| \frac{1}{\sqrt{|h|}} = 2 \sqrt{|h|}. 
\end{equation}
For the other sum, we trivially bound $\big| e^{2\pi i n^2 h } -  1 \big| \leq 2$ to get
\begin{equation}
\Big|  \sum_{|n| > N} \frac{e^{2\pi i n^2 t}\, e^{2\pi i n x_0} }{n^2} \Big( e^{2\pi i n^2 h} - 1 \Big)  \Big|
\leq  2\, \sum_{n=N+1}^\infty \frac{2}{n^2} 
\leq \frac{4}{N} 
\leq \frac{8}{N+1}
\leq 8 \sqrt{|h|}. 
\end{equation}
Hence $\big| R_{x_0}(t+h) - R_{x_0}(t) \big| \leq 10  |h|^{1/2}$. 
This holds for all $t$, so $R_{x_0} \in C^{1/2}(t)$ for all $t \in \mathbb R$. 
\end{proof}

\subsection{Asymptotic behavior of $R_{x_0}$ around rational $t$}\label{sec:Asymptotic_Behavior_Rationals}
The building block for all results in this article is the behavior of $R_{x_0}$ around rationals, which we compute explicitly. 
\begin{prop}\label{thm:AsymototicRationals}
Let $x_0 \in \mathbb R$.
Let $p,q \in \mathbb N$ be such that $(p,q)=1$. Then, 
\begin{equation}
R_{x_0}\left( \frac{p}{q} + h  \right)  - R_{x_0}\left( \frac{p}{q} \right) =  - 2\pi i h + \frac{\sqrt{|h|}}{q} \,  \sum_{m \in \mathbb Z} G(p,m,q)  \, F_\pm \left( \frac{x_0 - m/q}{\sqrt{h}} \right), 
\qquad \text{ for } \,  h \neq 0, 
\end{equation}
where $F_\pm = F_+$ if $h > 0$ and $F_\pm = F_-$ if $h < 0$, and
\begin{equation}
G(p,m,q) = \sum_{r=0}^{q-1} e^{2\pi i \frac{pr^2 + mr}{q}}, \qquad 
F_\pm(\xi) = \int_{\mathbb R} \frac{e^{\pm 2\pi i x^2} - 1}{x^2} \, e^{2\pi i x \xi} \, dx.
\end{equation}
The function $F_\pm$ is bounded and continuous, 
$F_\pm(0) = 2\pi (-1 \pm i)$, 
and 
\begin{equation}
 F_\pm(\xi) =  (1 \pm i) \, \frac{e^{\mp \pi i \xi^2/2}}{\xi^2} + O\left( \frac{1}{\xi^4} \right)  = O\left( \frac{1}{\xi^2} \right), \qquad \text{ as } \quad  \xi \to \infty.  
\end{equation}
\end{prop}

\begin{proof}
We follow the classical approach, 
which can be traced back to Smith \cite{Smith1972},
of using the Poisson summation formula. 
From the definition of $R_{x_0}$, 
complete first the sum to $n \in \mathbb Z$ to write
\begin{equation}
R_{x_0}\left( \frac{p}{q} + h  \right)  - R_{x_0}\left( \frac{p}{q} \right) 
=  - 2\pi i h +  \sum_{n \in \mathbb Z} \frac{e^{2\pi i n^2 h} - 1}{n^2} \,  e^{2\pi i \frac{pn^2}{q} } \, e^{2\pi i n x_0}, 
\end{equation} 
where we must interpret the term $n=0$ as the value of $\frac{e^{2\pi i n^2 h} - 1}{n^2} \simeq 2\pi i h$ as $n \to 0$. 
Split the sum modulo $q$ by writing $n = mq+r$ and
\begin{equation}\label{eq:PSF1}
\sum_{n \in \mathbb Z} \frac{e^{2\pi i n^2 h} - 1}{n^2} \,  e^{2\pi i \frac{pn^2}{q} } \, e^{2\pi i n x_0}
= \sum_{r=0}^{q-1}  e^{2\pi i \frac{p r^2}{q} } \, \sum_{m \in \mathbb Z} \frac{e^{2\pi i (mq+r)^2 h} - 1}{(mq+r)^2} \, e^{2\pi i ( mq + r) x_0}. 
\end{equation}
Use the Poisson summation formula for the function
\begin{equation}
f(y) = \frac{e^{2\pi i (yq+r)^2 h} - 1}{(yq+r)^2} \, e^{2\pi i (yq + r) x_0}, 
\end{equation}
for which, changing variables $(yq+r) \sqrt{|h|} = z$, we have
\begin{equation}
\widehat{f}(\xi) 
 = \frac{\sqrt{|h|}}{q} \, e^{2\pi i r \xi/q  } \, \int \frac{e^{2\pi i \operatorname{sgn}(h) z^2}-1}{z^2} \, e^{2\pi i \frac{z}{\sqrt{|h|}} (x_0 - \xi/q ) }\, dz 
 =  \frac{\sqrt{|h|}}{q} \, e^{2\pi i r  \xi/q  } \, F_\pm \Big( \frac{x_0 - \xi/q}{\sqrt{|h|}} \Big).
\end{equation}
Therefore, 
\begin{equation}
\eqref{eq:PSF1}  = \sum_{r=0}^{q-1} e^{2\pi i \frac{p r^2}{q} } \, \sum_{m \in \mathbb Z} \frac{\sqrt{|h|}}{q} \, e^{2\pi i r m/q  } \, F_\pm \Big( \frac{x_0 - m/q}{\sqrt{|h|}} \Big) \\
=  \frac{\sqrt{|h|}}{q} \,  \sum_{m \in \mathbb Z} G(p,m,q)  \, F_\pm \Big( \frac{x_0 - m/q}{\sqrt{|h|}} \Big).
\end{equation}
The properties for $F_\pm$ 
follow by integration by parts and the value of the Fresnel integral. 
\end{proof}

The main term in Proposition~\ref{thm:AsymototicRationals}
corresponds to $m \in \mathbb Z$ such that $x_0 - m/q$ is closest to 0. 
Define
\begin{equation}\label{eq:x_q_AND_m_q}
\left\{
\begin{array}{l}
m_q = \operatorname{argmin}_{m \in \mathbb Z} \big |x_0 - \frac{m}{q} \big|, \\
x_q = x_0 - \frac{m_q}{q},
\end{array}
\right.
\qquad \text{ so that } \qquad 
|x_q| 
= \Big| x_0 - \frac{m_q}{q} \Big| 
= \operatorname{dist} \Big( x_0, \frac{\mathbb Z}{q} \Big) 
\leq \frac{1}{2q}.
\end{equation}
Then, shifting the sum, 
\begin{equation}
R_{x_0}\Big( \frac{p}{q} + h  \Big)  - R_{x_0}\Big( \frac{p}{q} \Big)  =   - 2\pi i h + \frac{\sqrt{|h|}}{q} \,  G(p,m_q,q) F_\pm\Big( \frac{x_q}{\sqrt{|h|}} \Big)  
 + \frac{\sqrt{|h|}}{q} \, \sum_{m \neq 0} G(p,m_q + m, q) \, F_\pm \Big( \frac{x_q - m/q}{\sqrt{|h|}} \Big). 
\end{equation}
Let us now bound the sum as an error term. 
As long as $(p,q) = 1$, it is a well-known property of Gauss sums that $|G(p,m,q)| \leq \sqrt{2q}$ for all $m \in \mathbb N$, so 
\begin{equation}
\frac{\sqrt{|h|}}{q} \Big|  \sum_{m \neq 0} G(p,m_q + m, q) \, F_\pm \Big( \frac{x_q - m/q}{\sqrt{|h|}} \Big)  \Big|
 \leq 
 2\, \frac{\sqrt{|h|}}{\sqrt{q}} \, \sum_{m \neq 0} \Big|  F_\pm \Big( \frac{x_q - m/q}{\sqrt{|h|}} \Big) \Big|. 
\end{equation}
Since $|x_q| \leq 1/(2q)$ and $m \neq 0$, we have $|x_q - m/q| \simeq |m|/q$.
This suggests separating two cases:
\begin{itemize}
	\item If $q \sqrt{|h|} < 1$, we use the property $F_\pm(x) = O(x^{-2})$ to bound
\begin{equation}
\sum_{m \neq 0} \Big|  F_\pm\Big( \frac{x_q - m/q}{\sqrt{|h|}} \Big) \Big| 
 \lesssim   \sum_{m \neq 0} \frac{|h|}{|x_q - m/q|^2}
 \simeq q^2 \, |h|\, \sum_{m \neq 0} \frac{1}{m^2}
 \simeq  q^2 |h|. 
\end{equation}	

\item If $q \sqrt{|h|} \geq 1$, we split the sum as 
	\begin{equation}
\begin{split}
\sum_{m \neq 0} \Big|  F_\pm \Big( \frac{x_q - m/q}{\sqrt{|h|}} \Big) \Big| 
& = \sum_{|m| \leq q\sqrt{|h|}} \Big|  F_\pm \Big( \frac{x_q - m/q}{\sqrt{|h|}} \Big) \Big| + \sum_{|m| \geq q\sqrt{|h|}} \Big|  F_\pm \Big( \frac{x_q - m/q}{\sqrt{|h|}} \Big) \Big| \\
& \leq \sum_{|m| \leq q\sqrt{|h|}} C + \sum_{|m| \geq q\sqrt{|h|}} \frac{|h|}{|x_q - m/q|^2} \\
& \simeq q \sqrt{|h|} + q^2 |h| \sum_{|m| \geq q\sqrt{|h|}} \frac{1}{m^2}  \simeq  q \, \sqrt{|h|}. 
\end{split}
\end{equation}
\end{itemize} 
These two bounds can be written simultaneously as
\begin{equation}
\sum_{m \neq 0} \Big|  F_\pm \Big( \frac{x_q - m/q}{\sqrt{|h|}} \Big) \Big|  \lesssim \min \big(  q \, \sqrt{|h|},  q^2 |h| \big),
\end{equation}
where the underlying constant is universal. 
Multiply by $\sqrt{|h|} / \sqrt{q}$ to get the following corollary.
\begin{cor}\label{thm:CorollaryAsymptotic}
Let $x_0 \in \mathbb R$.
Let $p,q \in \mathbb N$ be such that $(p,q)=1$. Then, 
\begin{equation}
R_{x_0}\Big( \frac{p}{q} + h  \Big)  - R_{x_0}\Big( \frac{p}{q} \Big)  
= - 2\pi i h + \frac{\sqrt{|h|}}{q} \,  G(p,m_q,q) F_\pm \Big( \frac{x_q}{\sqrt{|h|}} \Big)  + O\left( \min \left( \sqrt{q} \, h , q^{3/2}\, h^{3/2}  \right) \right),
\end{equation}
where the underlying constant of the $O$ is independent of $p, q$ and $x_0$.  
\end{cor}

\begin{rem}\label{thm:Remark_MainTerm}
The difference between $x_0 = 0$ and $x_0 \neq 0$ is clear from Corollary~\ref{thm:CorollaryAsymptotic}.
\begin{itemize}
	\item
	 If $x_0 = 0$, we have $x_q = 0 = m_q$ for all $q$. 
The main term is 
$|h|^{1/2} q^{-1} \, G(p,0,q) \, F_\pm(0)$, so there is a clear dichotomy: 
$R_0$ is differentiable at $p/q$ if and only if $G(p,0,q) = 0$, which happens if and only if $q \equiv 2 \pmod{4}$; 
in all other rationals, $R_{x_0}$ is $C^{1/2}$.  
	
	\item
	 If $x_0 \neq 0$, it is in general false that $x_q = 0$, so to determine the differentiability of $R_{x_0}$ we need to control the magnitude of $F_\pm(x_q/\sqrt{|h|})$. 
\end{itemize}
\end{rem}

\subsection{Lower bounds for the local H\"older regularity}
\label{sec:Lower_Bounds}

We now give lower bounds for $\alpha_{x_0}(t)$
that do not depend on $x_0$.  
In Section~\ref{sec:Lower_Holder_Rationals} we work with $t \in \mathbb Q$, 
and in Section~\ref{sec:Lower_Holder_Irrationals} with $t \not\in \mathbb Q$.  

\subsubsection{At rational points}\label{sec:Lower_Holder_Rationals}
There is a dichotomy in the H\"older regularity of $R_{x_0}$ at rational points. 
\begin{prop}\label{thm:Holder_Lower_Bound_At_Rationals}
Let $x_0 \in \mathbb R$ and $t \in \mathbb Q$. Then, either $\alpha_{x_0}(t) = 1/2$ or $\alpha_{x_0}(t) = 3/2$.  
\end{prop}
\begin{proof}
Let $t = p/q$ with $(p,q)=1$. 
If $q$ is fixed, we get 
$\min \left( \sqrt{q} \, |h| , q^{3/2}\, |h|^{3/2}  \right) = q^{3/2} |h|^{3/2}$ for small enough $|h|$, 
so from Corollary~\ref{thm:CorollaryAsymptotic} we get
\begin{equation}\label{eq:From_Corollary}
 R_{x_0}\Big( \frac{p}{q} + h  \Big)  - R_{x_0}\Big( \frac{p}{q} \Big) = - 2\pi i h + \frac{\sqrt{|h|}}{q} \, G(p,m_q,q) F_\pm \Big( \frac{x_q}{\sqrt{|h|}} \Big) + O\Big( q^{3/2} h^{3/2} \Big). 
\end{equation}
Then, differentiability completely depends on 
the Gauss sum $G(p,m_q,q)$ and on $x_q$.
\begin{itemize}
	\item[\textbf{Case 1}] If $G(p,m_q,q) =0$, then $\big| R_{x_0}\big( \frac{p}{q} + h  \big)  - R_{x_0}\big( \frac{p}{q} \big)  + 2\pi i h \big| \lesssim_q h^{3/2}$, so $\alpha_{x_0}(p/q) \geq 3/2$. 
	\item[\textbf{Case 2}] If $G(p,m_q,q) \neq 0$ and $x_q \neq 0$.
	Then, $|G(p,m_q,q)| \simeq \sqrt{q}$ and $\lim_{h \to 0} x_q/\sqrt{|h|} = \infty$, so $\big| F_\pm \big( x_q/\sqrt{|h|} \big) \big| \lesssim h/x_q^2$. Hence, $\alpha_{x_0}(p/q) \geq 3/2$ because
	\begin{equation}
	 R_{x_0}\Big( \frac{p}{q} + h  \Big)  - R_{x_0}\Big( \frac{p}{q} \Big) = - 2\pi i h + O\Big(   \frac{\sqrt{h}}{\sqrt{q}} \frac{h}{x_q^2} + q^{3/2}h^{3/2} \Big)
	 = - 2\pi i h + O_q\big(  h^{3/2} \big).
	\end{equation}
	
	\item[\textbf{Case 3}] If $G(p,m_q,q) \neq 0$ and $x_q = 0$, 
	we have $|G(p,m_q,q)| \simeq \sqrt{q}$, so from \eqref{eq:From_Corollary} we get
		\begin{equation}
		\Big|   R_{x_0}\Big( \frac{p}{q} + h  \Big)  - R_{x_0}\Big( \frac{p}{q} \Big) \Big| 
		\geq   \frac{\sqrt{|h|}}{q} |G(p,m_q,q)| |F_\pm(0)| + O_q( h )
		\simeq \frac{\sqrt{h}}{\sqrt{q}} + O_q( h )
		\gtrsim_q h^{1/2}
		\end{equation}		  
		for $h \ll_q 1$. Together with Proposition~\ref{thm:DuistermaatLowerBound},
		this implies $\alpha_{x_0}(p/q) = 1/2$. 
\end{itemize}
That Cases 1 and 2 actually imply $\alpha_{x_0}(t) = 3/2$ is a bit more technical; we postpone the proof to Proposition~\ref{thm:Rationals32} in Appendix~\ref{sec:Asymptotic_Alternative}. 
\end{proof}

\subsubsection{At irrational points}\label{sec:Lower_Holder_Irrationals}
We give a lower bound $\alpha_{x_0}(t)$ that depends on the exponent of irrationality of $t$, but not on $x_0$.

\begin{prop}\label{thm:Lower_Bound_For_Holder_Regularity}
Let $x_0 \in \mathbb R$ and $t \in \mathbb R \setminus \mathbb Q$. 
Let $\mu(t)$ be the exponent of irrationality of $t$. Then, $\alpha_{x_0}(t) \geq \frac12 + \frac1{2\mu(t)}$. 
\end{prop}
The proof of this result, which we include for completeness, 
closely follows the procedure by Chamizo and Ubis \cite[Proof of Theorem 2.3]{ChamizoUbis2014}.
\begin{rem}
Similar to what happens for $x_0 = 0$, where $\alpha_0(t) = 1/2 + 1/2\widetilde \mu(t) \geq 1/2 + 1/2\mu(t)$ 
 (see  \eqref{eq:Holder_For_Riemann}), 
we do not expect the bound in Proposition~\ref{thm:Lower_Bound_For_Holder_Regularity} 
to be optimal for all $t \not\in \mathbb Q$. 
However, it will be enough to compute the spectrum of singularities. 
\end{rem}

\begin{proof}
In view of Proposition~\ref{thm:DuistermaatLowerBound}, 
there is nothing to prove if $\mu(t) = \infty$, so assume $\mu(t) < \infty$. 
Let $p_n/q_n$ be the $n$-th approximation by continued fractions of $t$. 
Center the asymptotic behavior in Corollary~\ref{thm:CorollaryAsymptotic} at $p_n/q_n$, 
and bound it from above by 
\begin{equation}\label{eq:AsymptoticPnQn}
\Big| R_{x_0}\Big( \frac{p_n}{q_n} + h  \Big)  - R_{x_0}\Big( \frac{p_n}{q_n} \Big)   \Big| \lesssim 
 \frac{\sqrt{|h|}}{\sqrt{q_n}} + |h|  + \min \Big( \sqrt{q_n} \, h , q_n^{3/2}\, h^{3/2}  \Big),
\end{equation}
where we used that $|G(p_n, m_{q_n}, q_n)| \leq \sqrt{2q_n}$ for all $n \in \mathbb N$ and $|F(x)| \lesssim 1$ for all $x \in \mathbb R$. 

Let $h \neq 0$ be small enough. 
The sequence $| t -  p_n/q_n|$ is strictly decreasing,  
so choose $n$ 
such that 
\begin{equation}\label{eq:HBetweenConvergents}
\left| t - \frac{p_n}{q_n} \right| \leq |h| < \left| t - \frac{p_{n-1}}{q_{n-1}} \right|.
\end{equation} 
Then,
from \eqref{eq:AsymptoticPnQn}, \eqref{eq:HBetweenConvergents} and $|t - p_n/q_n + h| \leq 2|h|$, we get
\begin{equation}\label{eq:AsymptoticIrrational}
\begin{split}
& \left| R_{x_0}\left( t + h  \right)  - R_{x_0}\left( t \right)   \right|  \\
& \qquad \qquad \qquad \leq \left| R_{x_0}\left(  \frac{p_n}{q_n} + t - \frac{p_n}{q_n} + h  \right)  - R_{x_0}\left( \frac{p_n}{q_n} \right)   \right|  + \left| R_{x_0}\left( \frac{p_n}{q_n} + t - \frac{p_n}{q_n}  \right)  - R_{x_0}\left( \frac{p_n}{q_n} \right)   \right| \\
& \qquad \qquad \qquad \lesssim  \frac{\sqrt{|h|}}{\sqrt{q_n}} + |h|  + \min \left( \sqrt{q_n} \, |h| , q_n^{3/2}\, |h|^{3/2}  \right).
\end{split}
\end{equation}
Next we compute the dependence between $q_n$ and $h$. 
By the property of continued fractions
\begin{equation}
\frac{1}{q_n^{\mu_n}} = \Big| t - \frac{p_n}{q_n} \Big| \leq \frac{1}{q_{n+1}q_n}, 
\end{equation}
we get $1/q_n \leq 1/q_{n+1}^{1/(\mu_n - 1)}$ for all $n  \in \mathbb N$. 
Then, from \eqref{eq:HBetweenConvergents} we get
\begin{equation}\label{eq:HBetweenMu}
\frac{1}{q_n^{\mu_n}} \leq |h| < \frac{1}{q_{n-1}^{\mu_{n-1}}} \leq \frac{1}{q_n^{\mu_{n-1} / (\mu_{n-1} - 1 ) }}. 
\end{equation}
We now bound each term in \eqref{eq:AsymptoticIrrational} using \eqref{eq:HBetweenMu}.
\begin{itemize}
	\item For the first term, by \eqref{eq:HBetweenMu},
	$
	\sqrt{|h|}/\sqrt{q_n} \leq  |h|^{\frac12 + \frac{1}{2\mu_n}}.
	$
	
	\item The fact that $\mu_n \geq 2$ implies $\frac12 + \frac{1}{2\mu_n} \leq \frac34$, so 
	$ |h| \leq |h|^{3/4} \leq  |h|^{\frac12 + \frac{1}{2\mu_n}}$ 
	and the second term is absorbed by the first one. 
	
	\item For the third term, we write the minimum as
	\begin{equation}
	\min (\sqrt{q_n} \, |h|, q_n^{3/2} \, |h|^{3/2}) = \left\{ \begin{array}{ll}
	\sqrt{q_n} \, |h|, & \text{ when } |h| \geq 1/q_n^2, \\
	q_n^{3/2} \, |h|^{3/2}& \text{ when } |h| \leq 1/q_n^2.
	\end{array}
	\right.
	\end{equation}
	So we have two regions: 
	\begin{itemize}
		\item When $|h| \geq 1/q_n^2$, use \eqref{eq:HBetweenMu} to bound  
		\begin{equation}
		\sqrt{q_n} \, |h| \leq \frac{|h|}{|h|^{(\mu_{n-1} - 1)/2\mu_{n-1}}} = |h|^{\frac12 +  \frac{1}{2\mu_{n-1}}}.
		\end{equation}
		
		\item When $|h| \leq 1/q_n^2$, we directly have $q_n \leq |h|^{-1/2}$, so 
		\begin{equation}
		q_n^{3/2} \, |h|^{3/2} = |h|^{3/2 - 3/4} = |h|^{3/4} \leq |h|^{\frac12 +  \frac{1}{2\mu_{n-1}}},
		\end{equation}
		where in the last inequality we used $\frac12 + \frac{1}{2\mu_{n-1}} \leq \frac34$ as before. 
\end{itemize}		
\end{itemize}
Gathering all cases,  we get 
\begin{equation}
| R_{x_0}(t + h) -  R_{x_0}(t) | \leq |h|^{\frac12 + \frac{1}{2\mu_n}} +  |h|^{ \frac12 + \frac{1}{2\mu_{n-1}} }.
\end{equation} 
From the definition of the exponent of irrationality
 $\mu(t) = \limsup_{n \to \infty} \mu_n$, for any $\delta > 0$  
 there exists $N_\delta \in \mathbb N$  such that $\mu_n \leq \mu(t) + \delta$ for all $n \geq N_\delta$. 
Then, since $|h| < 1$, we have $|h|^{\frac12 + \frac{1}{2\mu_n}} \leq |h|^{\frac12 + \frac{1}{2\mu(t) + 2\delta}} $ for all $n \geq N_\delta$. 
Renaming $\delta$, we get $N_\delta \in \mathbb N$ such that 
\begin{equation}
| R_{x_0}(t + h) -  R_{x_0}(t) | \leq |h|^{\frac12 + \frac{1}{2\mu(t)} - \delta}, \qquad \text{ for  all } \quad  |h| \leq \Big| t - \frac{p_{N_\delta}}{q_{N_\delta}} \Big|, 
\end{equation}
so $\alpha_{x_0}(t) \geq \frac12 + \frac{1}{2\mu(t)} - \delta$. Since this holds for all $\delta > 0$, we conclude that $\alpha_{x_0}(t) \geq \frac12 + \frac{1}{2\mu(t)}$. 
\end{proof}

\section{Proof of Theorem~\ref{thm:Main_Theorem_Rationals_Spectrum}: Spectrum of singularities when $x_0 \in \mathbb Q$}
\label{sec:Rational_Holder}

In this section we prove Theorem~\ref{thm:Main_Theorem_Rationals_Spectrum}.
Let us fix $x_0 = P/Q$ such that $(P,Q) = 1$. 
To compute the spectrum of singularities $d_{x_0}$, 
we first characterize the rational points $t$ where $R_{x_0}$ is not differentiable, 
and then we give an upper bound for the regularity $\alpha_{x_0}(t)$ at irrational $t$.

\subsection{At rational points $t$}
In the proof of Proposition~\ref{thm:Holder_Lower_Bound_At_Rationals}
we established that $R_{x_0}$ is not differentiable at $t=p/q$ if and only if
$ G(p,m_q, q) \neq 0$
and 
$x_q = \operatorname{dist} (x_0,\mathbb Z /q ) = 0$.
 We characterize this in the following proposition.

\begin{prop}\label{thm:Non-differentiability_Rationals}
Let $x_0 = P/Q$ with $\operatorname{gcd}(P,Q) = 1$, and let $p,q \in \mathbb N$ such that $\operatorname{gcd}(p,q) = 1$. 
Then, $R_{x_0}$ is non-differentiable at $t = p/q$ if and only if
\begin{itemize}
	\item $q = kQ$ with $k \equiv 0,1,3 \pmod{4}$, in the case $Q \equiv 1 \pmod{2}$.
	\item $q = kQ$ with $k \equiv 0 \pmod{2}$, in the case $Q \equiv 0 \pmod{4}$.
	\item $q = kQ$ with $k \in \mathbb Z$,  in the case $Q \equiv 2 \pmod{4}$.
\end{itemize} 
In all such cases, the asymptotic behavior is 
\begin{equation}\label{eq:Asymptotic_Nondifferentiable}
	R_{x_0}\left( \frac{p}{q} + h  \right)  - R_{x_0}\left( \frac{p}{q} \right) 
 =    c \, e^{2\pi i \phi_{p,q,x_0}} \, F_\pm(0) \, \frac{\sqrt{|h|}}{\sqrt{q}} - 2\pi i h +  O\left( \min \left( \sqrt{q} \, h , q^{3/2}\, h^{3/2}  \right) \right).   
	\end{equation}
	where $c=1$ or $c = \sqrt{2}$ depending on parity conditions of $Q$ and $q$. 
	In particular, $\alpha_{x_0}(t) = 1/2$.
\end{prop}
\begin{proof}
In view of the proof of Proposition~\ref{thm:Holder_Lower_Bound_At_Rationals}, 
we must identify the conditions for $G(p,m_q,q) \neq 0$ and $x_q = 0$. 
Since $x_q = \operatorname{dist}(P/Q, \mathbb Z / q)$, 
we have $x_q = 0$ when there exists $m_q \in \mathbb Z$ such that 
\begin{equation}
\frac{P}{Q} = \frac{m_q}{q} \quad \Longleftrightarrow \quad Pq = m_qQ.
\end{equation}
Since $\operatorname{gcd}(P,Q) = 1$, then necessarily $Q | q$.
Reversely, if $q = kQ$, then picking $m_q = kP$ we have $m_q/q = P/Q$. 
In short, 
\begin{equation}
x_q = 0 \quad  \Longleftrightarrow \quad q \text{ is a multiple of } Q. 
\end{equation}
So let $q = kQ$ for some $k \in \mathbb N$. Then, $m_q = kP$. 
Let us characterize the second condition $G(p,m_q,q) = G(p,kP, kQ) \neq 0$. 
It is well-known that 
\begin{equation}\label{eq:GaussSumCharacterization}
G(a,b,c) \neq 0 \quad \Longleftrightarrow \quad  \text{ either } \left\{ \begin{array}{ll}
c \text{ is odd, or } \\
c \text{ is even and } \frac{c}{2} \equiv b \pmod{2}.
\end{array}
\right.
\end{equation}
 We separate cases:
\begin{itemize}
	\item Suppose $Q$ is odd. Then, according to \eqref{eq:GaussSumCharacterization}, we need either
	\begin{itemize}
		\item $kQ$ odd, which holds if and only if $k$ is odd, or
		\item $kQ$ even, which holds if and only if $k$ is even, and $kQ/2 \equiv kP \pmod{2}$. Since $Q$ is odd and $k$ is even, this is equivalent to $k/2 \equiv 0 \pmod{2}$, which means $k \equiv 0 \pmod{4}$. 
	\end{itemize}
	Therefore, if $q = kQ$, the Gauss sum $G(p,m_q,q) \neq 0$ if and only if $k \equiv 0,1,3 \pmod{4}$. 
	
	\item Suppose $Q \equiv 0 \pmod{4}$. Since $q = kQ$ is even, by \eqref{eq:GaussSumCharacterization} we need $kQ/2 \equiv kP \pmod{2}$. Since $Q$ is a multiple of 4, this is equivalent to $ kP \equiv 0 \pmod{2}$. 
	But since $Q$ is even, then $P$ must be odd. Therefore, $k$ must be even. 
	In short, if $q = kQ$, we have $G(p,m_q,q) \neq 0$ if and only if $k$ is even. 

	\item Suppose  $Q \equiv 2 \pmod{4}$. Since $q = kQ$ is even, by \eqref{eq:GaussSumCharacterization} we need $kQ/2 \equiv kP \pmod{2}$. Now both $Q/2$ and $P$ are odd, so this is equivalent to $k \equiv k \pmod{2}$, which is of course true. 
	Therefore, if $q = kQ$, we have $G(p,m_q,q) \neq 0$ for all $k \in \mathbb Z$. 
\end{itemize}

Once  all cases have been identified, \eqref{eq:Asymptotic_Nondifferentiable} follows from Corollary~\ref{thm:CorollaryAsymptotic}
and from the fact that if $G(p,m_q,q) \neq 0$ we have $|G(p,m_q,q)| = c\sqrt{q}$ with $c = 1$ or $c = \sqrt{2}$. 
\end{proof}

\subsection{A general upper bound for irrational $t$}
We begin the study of $t \not\in \mathbb Q$ by giving a general upper bound for $\alpha_{x_0}(t)$ for $t \not\in \mathbb Q$.
The proof uses an alternative asymptotic expression around rationals that we postpone to Appendix~\ref{sec:Asymptotic_Alternative}.  
\begin{prop}\label{thm:UpperBound34MainText}
Let $x_0 \in \mathbb Q$ and $t \not\in \mathbb Q$. Then, $\alpha_{x_0}(t) \leq  3/4$. 
\end{prop}
\begin{proof}
See Appendix~\ref{sec:Asymptotic_Alternative}, Proposition~\ref{thm:UpperBound34}. 
\end{proof}

\subsection{Upper bounds depending on the irrationality of $t$}
\label{sec:Rational_x0_Irrational_t}
We now aim at an upper bound for $\alpha_{x_0}(t)$
that depends on the irrationality of $t$ at the level of Proposition~\ref{thm:Lower_Bound_For_Holder_Regularity}.
The idea is to approximate $t$ by rationals $p/q$ where $R_{x_0}$ is non-differentiable, which we characterized in Proposition~\ref{thm:Non-differentiability_Rationals}.
To avoid treating different cases depending on the parity of $Q$,  
let us restrict\footnote{We lose nothing with this reduction when computing the spectrum of singularities, but it may be problematic if we aim to compute the H\"older regularity $\alpha_{x_0}(t)$ for all $t$.   } $q \in 4Q \mathbb N$, 
such that the three conditions in Proposition~\ref{thm:Non-differentiability_Rationals} are simultaneously satisfied
and \eqref{eq:Asymptotic_Nondifferentiable} holds.
%

Let $\mu \in [2, \infty)$. 
Define the classic Diophantine set
\begin{equation}
A_\mu = \left\{ \, t \in (0,1) \setminus \mathbb Q \, : \,   \big| t - \frac{p}{q} \big| \leq \frac{1}{q^\mu} \, \, \text{ for i. m. coprime pairs } (p,q) \in \mathbb N \times \mathbb N    \,    \right\}
\end{equation}
and for $0 < a < 1$ small enough define the restricted Diophantine set
\begin{equation}
 A_{\mu,Q} = \left\{ \, t \in (0,1) \setminus \mathbb Q \, : \,   \big| t - \frac{p}{q} \big| \leq \frac{a}{q^\mu} \, \, \text{ for i. m. coprime pairs } (p,q) \in \mathbb N \times 4Q\mathbb N   \,    \right\}.
\end{equation}
For $\mu = \infty$ we define 
$A_\infty = \bigcap_{\mu \geq 2} A_\mu$ and 
$A_{\infty, Q} = \bigcap_{\mu \geq 2} A_{\mu, Q}$. 
Clearly, $  A_{\mu,Q} \subset A_\mu$. 
Our first step is to give an upper bound for $\alpha_{x_0}(t)$ for $t \in  A_{\mu,Q}$. 
\begin{prop}\label{thm:Upper_Bound_For_Holder_regularity}
Let $\mu \geq 2$ and $t \in A_{\mu,Q}$. 
Then, 
$\alpha_{x_0}(t) \leq \frac12 + \frac{1}{2\mu}$. 
\end{prop}
\begin{proof}
We begin with the case $\mu < \infty$. 
If $t \in  A_{\mu,Q}$, 
there is a sequence of irreducible fractions $p_n/q_n$ with $q_n \in 4Q\mathbb N$, for which  
we can use \eqref{eq:Asymptotic_Nondifferentiable} and write
\begin{equation}\label{eq:Asymptotic_EvenConvergents}
	R_{x_0}\left( t  \right)  - R_{x_0}\Big( \frac{p_n}{q_n} \Big) 
 =    c \, e^{2\pi i \phi_{n,x_0}} \,  \frac{\sqrt{|h_n|}}{\sqrt{q_n}} - 2\pi i h_n +  O\left( \min \left( \sqrt{q_n} \, h_n , q_n^{3/2}\, h_n^{3/2}  \right) \right),
	\end{equation}
where we absorbed $F(0)$ into $c$ and we defined $h_n$ and $\mu_n$ as 
\begin{equation}\label{eq:Approximation_Distance}
h_n = t - \frac{p_n}{q_n}, \quad 
|h_n|
= \frac{1}{q_n^{\mu_n}} \leq \frac{a}{q_n^\mu} < \frac{1}{q_n^\mu}. 
\end{equation}
We now absorb the second and third terms in \eqref{eq:Asymptotic_EvenConvergents} in the first term.
First, $\mu \geq 2$ implies $q_n^2 |h_n| \leq 1$,  
so 
$
	\min ( \sqrt{q_n} \, |h_n| , q_n^{3/2}\, |h_n|^{3/2} ) = q_n^{3/2}\, |h_n|^{3/2}. 
$
Letting $C$  be the universal constant in the $O$ in \eqref{eq:Asymptotic_EvenConvergents},  
\begin{equation}
C\, q_n^{3/2} |h_n|^{3/2} \leq \frac{c}{4}  \frac{\sqrt{|h_n|}}{\sqrt{q_n}}
\qquad  \Longleftrightarrow \qquad 
q_n^2 |h_n| \leq \frac{ c}{4C}, 
\end{equation}
and since $q_n^2 |h_n| \leq a q_n^{2 - \mu } \leq a$, 
it suffices to ask $a \leq c/(4C)$. 
Regarding the second term, we have
\begin{equation}
2\pi |h_n| \leq \frac{c}{4} \,   \frac{  \sqrt{|h_n|}}{\sqrt{q_n} }
 \qquad \Longleftrightarrow  \qquad 
 q_n \, |h_n| \leq \Big( \frac{c}{8\pi} \Big)^2 
\end{equation}
This holds for large $n$ because $q_n^2 |h_n| \leq 1$ implies $q_n \, |h_n| \leq 1/q_n$, and because $\limsup_{n \to \infty} q_n = \infty$ (otherwise $q_n$ would be bounded and hence the sequence $p_n/q_n$ would be finite). 
All together, using the reverse triangle inequality in \eqref{eq:Asymptotic_EvenConvergents} and the bound for $h_n$ in \eqref{eq:Approximation_Distance}
\begin{equation}\label{eq:Asymptotic_EvenConvergents_LowerBound}
	\Big| R_{x_0}\left( t  \right)  - R_{x_0}\Big( \frac{p_n}{q_n} \Big)  \Big| \geq    \frac{c}{2}  \,  \frac{\sqrt{|h_n|}}{\sqrt{q_n}} \geq \frac{c}{2}  \,  |h_n|^{\frac12 + \frac{1}{2\mu}}, 
	\qquad \forall n \gg 1. 
\end{equation}
This means that $R_{x_0}$ cannot be better than $\mathcal C^{\frac12 + \frac{1}{2\mu}}$ at $t$, 
thus concluding the proof for $\mu < \infty$. 

If $t \in A_{ \infty,Q}$, by definition $t \in A_{\mu, Q}$ for all $\mu \geq 2$, hence we just proved that $\alpha_{x_0}(t) \leq 1/2 + 1/(2\mu)$ for all $\mu \geq 2$. Taking the limit $\mu \to \infty$ we get $\alpha_{x_0}(t) \leq 1/2$. 
\end{proof}

To prove Theorem~\ref{thm:Main_Theorem_Rationals_Spectrum}, 
we need to compute 
 $\operatorname{dim}_{\mathcal H}\{ \, t \, : \, \alpha_{x_0}(t) = \alpha  \,   \}$
with prescribed $\alpha$. 
For that, we need to complement Proposition~\ref{thm:Upper_Bound_For_Holder_regularity}
by proving that for $t  \in A_{\mu,Q}$ we also have $\alpha_{x_0}(t) \geq \frac12 + \frac{1}{2\mu}$. 
By Proposition~\ref{thm:Lower_Bound_For_Holder_Regularity}, 
it would suffice to prove that  $t \in  A_{\mu,Q}$ has irrationality $\mu (t) = \mu$. 
Unfortunately, when $\mu < \infty$ this need not be true.
To fix this, for $2 \leq \mu < \infty$ define the companion sets 
\begin{equation}
 B_\mu   = A_\mu \setminus \bigcup_{\epsilon > 0} A_{\mu + \epsilon} 
 = \Big\{ \, t \in  A_\mu \, \mid \,   \forall \epsilon > 0, \,  \big| t - \frac{p}{q} \big| \leq \frac{1}{q^{\mu + \epsilon}} \, \, \text{ only for finitely many } \frac{p}{q} \,    \Big\},
\end{equation} 
and
\begin{equation}\label{eq:B_Mu_Q}
 B_{\mu,Q}  = A_{\mu,Q}  \setminus \bigcup_{\epsilon > 0} A_{\mu + \epsilon}  
 = \Big\{ \, t \in  A_{\mu,Q} \, \mid \,   \forall \epsilon > 0, \,  \big| t - \frac{p}{q} \big| \leq \frac{1}{q^{\mu + \epsilon}} \, \, \text{ only for finitely many } \frac{p}{q} \,    \Big\},
\end{equation} 
which have the properties we need. 
\begin{prop}\label{thm:Holder_Regularity_In_Bmu}
Let $2 \leq \mu < \infty$. Then, 
\begin{enumerate}[(i)]
	\item $  B_{\mu,Q} \subset B_\mu \subset \{   \,  t \in \mathbb R \setminus \mathbb Q  \, : \, \mu(t) = \mu \,      \}$.   
	\item If $t \in B_{\mu,Q}$, then $\alpha_{x_0}(t) = \frac12 + \frac{1}{2\mu}$.    
\item If $t \in A_{\infty,Q}$, then $\alpha_{x_0}(t) = 1/2$.  
\end{enumerate}
\end{prop}
\begin{proof}
$(i)$ First,  $ B_{\mu,Q} \subset B_\mu$ because $ A_{\mu,Q} \subset A_\mu$. 
The second inclusion is a consequence of the definition of the irrationality exponent in \eqref{eq:Exponent_Of_Irrationality}. 
Indeed, $t \in B_\mu \subset A_\mu$ directly implies that $\mu(t) \geq \mu$. 
On the other hand, for all $\epsilon > 0$, $t \in B_\mu$ implies $t \notin A_{\mu + \epsilon}$, so $t$ can be approximated with the exponent $\mu + \epsilon$ only with finitely many fractions, and thus $\mu(t) \leq \mu + \epsilon$. 
Consequently, $\mu(t) \leq \mu$.

$(ii)$ By $(i)$, $t \in  B_{\mu,Q}$ implies $\mu(t) = \mu$, so by Proposition~\ref{thm:Lower_Bound_For_Holder_Regularity} we get $\alpha_{x_0} (t) \geq \frac12 + \frac{1}{2\mu}$. 
At the same time,  $t \in B_{\mu,Q} \subset  A_{\mu,Q}$, so Proposition~\ref{thm:Upper_Bound_For_Holder_regularity} implies $\alpha_{x_0}(t) \leq \frac12 + \frac{1}{2\mu}$. 

$(iii)$ It follows directly from Propositions~\ref{thm:DuistermaatLowerBound} and \ref{thm:Upper_Bound_For_Holder_regularity}. 
\end{proof}

\begin{cor}\label{thm:Sandwich_Sets}
Let $2 < \mu < \infty$. Then, for all $\epsilon > 0$, 
\begin{equation}
 B_{\mu,Q} \subset  \left\{ \, t \in (0,1) \, : \, \alpha_{x_0}(t) = \frac12 + \frac{1}{2\mu}  \,  \right\} \subset A_{\mu - \epsilon}. 
\end{equation}
For $\mu = 2$ we have the slightly more precise
\begin{equation}
 B_{2,Q} \subset  \{ \, t \in (0,1) \, : \, \alpha_{x_0}(t) =  3/4 \,  \}  \subset A_2.
\end{equation}
For $\mu = \infty$, 
\begin{equation}
 A_{\infty,Q} \subset  \{ \, t \in (0,1) \, : \, \alpha_{x_0}(t) =  1/2  \,  \}  \subset A_\infty \cup \mathbb Q.
\end{equation}
\end{cor}
\begin{proof}
Left inclusions follow from Proposition~\ref{thm:Holder_Regularity_In_Bmu} for all $\mu \geq 2$, so we only need to prove the right inclusions. 
When $\mu = 2$, it follows from the Dirichlet approximation theorem, which states that $\mathbb R \setminus \mathbb Q \subset A_2$, 
and Proposition~\ref{thm:Holder_Lower_Bound_At_Rationals}, in which we proved that if $t$ is rational, then either $\alpha_{x_0}(t) = 1/2$ or $\alpha_{x_0}(t) \geq 3/2$. 
Thus,  $\{ \, t \in (0,1) \, : \, \alpha_{x_0}(t) =  3/4 \,  \}  \subset (0,1) \setminus \mathbb Q \subset A_2$. 
Suppose now that $2 < \mu < \infty$ and that $\alpha_{x_0}(t) = \frac12 + \frac{1}{2\mu}$.
By Proposition~\ref{thm:Lower_Bound_For_Holder_Regularity}, $\alpha_{x_0}(t) \geq \frac12 + \frac{1}{2\mu(t)}$, so we get
$\mu \leq \mu(t)$.
In particular, given any $\epsilon > 0$, we have $\mu - \epsilon < \mu(t)$, so  
$ \big| t - \frac{p}{q} \big| \leq 1/q^{\mu - \epsilon}$ 
for infinitely many coprime pairs $ (p,q) \in \mathbb N \times \mathbb N$, 
which means that $t \in A_{\mu - \epsilon}$. 
Finally, for $\mu = \infty$, if $t \not\in \mathbb Q$ is such that $\alpha_{x_0}(t) = 1/2$, then by Proposition~\ref{thm:Lower_Bound_For_Holder_Regularity} we get $\mu(t) = \infty$, 
which implies that $t \in A_\mu$ for all $\mu \geq 2$, hence $t \in A_\infty$. 
\end{proof}

Now, to prove Theorem~\ref{thm:Main_Theorem_Rationals_Spectrum}
it suffices to compute $\operatorname{dim}_{\mathcal H}  A_\mu$ and $\operatorname{dim}_{\mathcal H}  B_{\mu,Q}$.  
\begin{prop}\label{thm:Dimensions}
For  $2 \leq \mu < \infty$,  
$\operatorname{dim}_{\mathcal H} A_\mu 
= \operatorname{dim}_{\mathcal H} B_{\mu,Q}
= 2/\mu$. 
Also, $\operatorname{dim}_{\mathcal H} A_\infty = 0$. 
\end{prop}
Form this result, whose proof we postpone, we can prove 
Theorem~\ref{thm:Main_Theorem_Rationals_Spectrum}
as a corollary.
\begin{thm}\label{thm:Spectrum_Rationals}
Let $x_0 \in \mathbb Q$. Then, the spectrum of singularities of $R_{x_0}$ is
\begin{equation}
d_{x_0}(\alpha) = 
\left\{
\begin{array}{ll}
4  \alpha - 2, & 1/2 \leq \alpha \leq 3/4, \\
0, & \alpha = 3/2, \\
-\infty, & \text{otherwise.}
\end{array} \right.
\end{equation}
\end{thm}
\begin{proof}
Proposition~\ref{thm:DuistermaatLowerBound} implies $d(\alpha) = -\infty$ when $\alpha < 1/2$, while Propositions~\ref{thm:Holder_Lower_Bound_At_Rationals} and \ref{thm:UpperBound34MainText} imply that $d_{x_0}(3/2) = 0$ and $d_{x_0}(\alpha) = -\infty$ if $\alpha > 3/4$ and $\alpha \neq 3/2$.
When $1/2 \leq \alpha \leq 3/4$, 
it follows from
Corollary~\ref{thm:Sandwich_Sets}, Proposition~\ref{thm:Dimensions} and the periodicity of $R_{x_0}$. 
First, 
$d_{x_0}(1/2) \leq  \operatorname{dim}_{\mathcal H} ( A_\infty \cup \mathbb Q ) = 0$
because $\operatorname{dim}_{\mathcal H} \mathbb Q = \operatorname{dim}_{\mathcal H}  A_\infty = 0$. 
On the other hand, for  $2 \leq \mu < \infty$ we get
\begin{equation}
\frac{2}{\mu} \leq 
d_{x_0} \left( \frac12 + \frac{1}{2\mu} \right) \leq \frac{2}{\mu - \epsilon},  \qquad \forall \epsilon > 0 
\qquad \Longrightarrow \qquad
d_{x_0}\left( \frac12 + \frac{1}{2\mu} \right) = \frac{2}{\mu}. 
\end{equation}
which gives the result for $1/2 < \alpha \leq 3/4$ by renaming $\alpha =  1/2 + 1/(2\mu)$.
\end{proof}

Let us now prove Proposition~\ref{thm:Dimensions}. 
\begin{proof}[Proof of Proposition~\ref{thm:Dimensions}]
We have $A_2 = (0,1) \setminus \mathbb Q $ by Dirichlet approximation, so $\operatorname{dim}_{\mathcal H} A_2 = 1$. 
For $\mu > 2$ we have $\operatorname{dim}_{\mathcal H} A_\mu = 2/\mu$ by the Jarnik-Besicovitch Theorem~\ref{thm:Jarnik_Besicovitch}. 
Also, $A_\infty \subset A_\mu$ for all $\mu \geq 2$, 
so $\operatorname{dim}_{\mathcal H} A_\infty \leq 2/\mu$ 
for all $\mu \geq 2$,
hence $\operatorname{dim}_{\mathcal H} A_\infty = 0$. 
So we only need to prove that $\operatorname{dim}_{\mathcal H} B_{\mu,Q} = 2/\mu$ for $2 \leq \mu < \infty$. 
Moreover, 
\begin{equation}
 B_{\mu,Q} = A_{\mu,Q} \setminus \bigcup_{\epsilon > 0} A_{\mu + \epsilon} \subset  A_{\mu,Q} \subset A_\mu, 
\end{equation}
which implies $\dim_{\mathcal H}  B_{\mu,Q} \leq  \dim_{\mathcal H}  A_\mu = 2/\mu$. 
Hence it suffices to prove that $\operatorname{dim}_{\mathcal H} B_{\mu,Q} \geq 2/\mu$. 
This claim follows 
from $\mathcal H^{2/\mu}( A_{\mu,Q} ) >  0$. 
Indeed, 
we first remark that the sets $A_\mu$ are nested, in the sense that 
$A_{\sigma} \subset A_\mu$ when $\sigma > \mu$.
We can therefore write
\begin{equation}
 \bigcup_{\epsilon > 0} A_{\mu + \epsilon} = \bigcup_{n \in \mathbb N} A_{\mu + \frac{1}{n}}.
\end{equation}
By the Jarnik-Besicovitch Theorem~\ref{thm:Jarnik_Besicovitch}, $\dim_{\mathcal H} A_{\mu +  1/n} = 2/(\mu + 1/n) < 2/\mu$, so $\mathcal H^{2/\mu}(A_{\mu + 1/n}) = 0$ for all $n \in \mathbb N$, 
hence
\begin{equation}
\mathcal H^{2/\mu} \Big( \bigcup_{\epsilon  > 0} A_{\mu + \epsilon} \Big)
= \mathcal H^{2/\mu} \Big( \bigcup_{n \in \mathbb N} A_{\mu + \frac{1}{n}} \Big)
 = \lim_{n \to \infty} \mathcal H^{2/\mu} \big( A_{\mu + \frac{1}{n}} \big) = 0. 
\end{equation}
Therefore, 
\begin{equation}
\mathcal H^{2/\mu} \big( B_{\mu,Q} \big)
= \mathcal H^{2/\mu} \Big(  A_{\mu,Q} \setminus \bigcup_{\epsilon > 0} A_{\mu + \epsilon} \Big) = \mathcal H^{2/\mu} (  A_{\mu,Q}  )  - \mathcal H^{2/\mu} \Big( \bigcup_{\epsilon  > 0} A_{\mu + \epsilon} \Big)
= \mathcal H^{2/\mu} \left(  A_{\mu,Q}  \right),  
\end{equation}
so $\mathcal H^{2/\mu}( A_{\mu,Q}) > 0$ implies $\mathcal H^{2/\mu}( B_{\mu,Q}) > 0$,
hence  $\dim_{\mathcal H}  B_{\mu,Q} \geq 2/\mu$. 

Let us thus prove 
$\mathcal H^{2/\mu}( A_{\mu,Q}) > 0$, 
for which we follow the procedure outlined in Section~\ref{sec:Diophantine_Approximation} with the set of denominators $\mathcal Q = 4Q\mathbb N$. 
We first detect the largest $\mu$ such that $A_{\mu,Q}$
has full Lebesgue measure
using the Duffin-Schaeffer Theorem~\ref{thm:Duffin_Schaeffer}. 
Define
\begin{equation}
\psi_{\mu,Q}(q) = a\, \frac{ \mathbbm{1}_{ 4Q \mathbb N}(q)}{q^{\mu}}, 
\end{equation}
where $a> 0$ comes from the definition of $A_{\mu,Q}$ and 
$\mathbbm{1}_{ 4Q \mathbb N}(q)$ is the indicator function of $4Q\mathbb N$,
\begin{equation}
\mathbbm{1}_{ 4Q \mathbb N}(q) = 
\left\{ 
\begin{array}{ll}
1, & \text{ if } 4 Q \, \mid \, q, \\
0, & \text{ otherwise.}
\end{array}
\right.
\end{equation}
Then, we have $A_{\mu,Q} = A_{\psi_{\mu,Q}}$, where
\begin{equation}\label{eq:A_with_Psi}
A_{\psi_{\mu,Q}}
= \Big\{ \, t \in [0,1]  \, : \,   \Big| t - \frac{p}{q} \Big| \leq \psi_{\mu,Q}(q) \, \, \text{ for i. m. coprime pairs }  (p,q) \in  \mathbb N \times \mathbb N  \,  \Big\}
\end{equation}
has the form needed for the Duffin-Schaeffer Theorem~\ref{thm:Duffin_Schaeffer}. 
Indeed, the inclusion $\subset$ follows directly from the definition of $\psi_{\mu,Q}$. 
For the inclusion $\supset$, observe first that if $t \in A_{\psi_{\mu,Q}}$ with $\mu > 1$, then $t \not\in \mathbb Q$. 
Now, if a coprime pair $(p,q) \in \mathbb N^2$ satisfies $|t - p/q| \leq \psi_{\mu,Q}(q)$, then $q \in 4Q\mathbb N$ because otherwise we get the contradiction
\begin{equation}
0 < \Big| t - \frac{p}{q} \Big| \leq \psi_{\mu,Q}(q) = a \, \frac{\mathbbm{1}_{4Q\mathbb N}(q)}{q^\mu} = 0.
\end{equation}
In this setting, the Duffin-Schaeffer theorem says that 
$A_{\mu,Q}$ has Lebesgue measure 1  if and only if 
\begin{equation}\label{eq:The_Sum}
\sum_{q=1}^\infty \varphi(q) \,  \psi_{\mu,Q}(q) 
= \frac{a}{(4Q)^\mu} \, \sum_{n=1}^\infty \frac{\varphi(4Qn)}{ n^\mu } 
= \infty, 
\end{equation}
and has zero measure otherwise. 
Using this characterization, we prove now 
\begin{equation}\label{eq:CaseMu2}
| A_{\mu,Q} | = \left\{  \begin{array}{ll}
1, & \mu \leq 2, \\
0, & \mu > 2, 
\end{array}
\right.
\end{equation}
independently of $a$. 
To detect the critical $\mu=2$, 
trivially bound $\varphi(n) < n$ so that
\begin{equation}
	\sum_{n=1}^\infty \frac{\varphi(4Qn)}{ n^\mu } < \sum_{n=1}^\infty \frac{4Qn}{ n^\mu }
	= 4Q\, \sum_{n=1}^\infty \frac{1}{ n^{\mu-1} } < \infty, \qquad \text{ if } \, \, \mu > 2.
	\end{equation}
However, this argument fails when $\mu = 2$.
What is more, 
denote by $\mathbb P$ the set of primes so that
	\begin{equation}
		\sum_{n=1}^\infty \,  \frac{\varphi(4Qn)}{n^2} >  \sum_{p \in \mathbb P, \, p > 4Q} \,  \frac{\varphi(4Qp)}{p^2}
		\end{equation}		
	If $p \in \mathbb P$ and $p > 4Q$, then $\gcd (p,4Q) = 1$ because $p \nmid 4Q$ (for if $p \mid 4Q$ 		then $p \leq 4Q$). Therefore, $\varphi(4Qp) = \varphi (4Q) \, \varphi(p) = \varphi (4Q) \, (p-1) > \varphi (4Q) \, p/2$, so 
	\begin{equation}
		\sum_{n=1}^\infty \,  \frac{\varphi(4Qn)}{n^2} 
		> \frac{\varphi(4Q)}{2} \, \sum_{p \in \mathbb P, \, p > 4Q} \,  \frac{1}{p} = \infty,
	\end{equation}
	because the sum of the reciprocals of the prime numbers diverges\footnote{
This argument shows that the strategy used here to compute the dimension of $A_{\mu,\mathcal Q}$
also works if we restrict the denominators 
to the primes $\mathcal Q = \mathbb P$ in the first place. This situation arises when computing the spectrum of singularities of trajectories of polygonal lines with non-zero rational torsion, studied in \cite{BanicaVega2022}.}. 
The Duffin-Schaeffer Theorem~\ref{thm:Duffin_Schaeffer} thus implies that $ | A_{2,Q}| = 1$ and,
in particular, $\operatorname{dim}_{\mathcal H} A_{2,Q} = 1$.
From this we immediately get $ | A_{\mu,Q}| = 1$ when $\mu < 2$ because $A_{2,Q} \subset A_{\mu,Q}$.

Once we know \eqref{eq:CaseMu2},
 we use the Mass Transference Principle Theorem~\ref{thm:Mass_Transference_Principle}
to compute the dimension of $A_{\mu,Q}$ for $\mu > 2$.
Write first 
\begin{equation}
 A_{\mu,Q} = \limsup_{q \to \infty } \bigcup_{p \leq q, \, (p,q) = 1} B \Big( \, \frac{p}{q}, \psi_{\mu,Q}(q) \Big).
\end{equation}
Let $\beta = 2/\mu$  so that
\begin{equation}
\psi_{\mu,Q}(q)^\beta 
= \Big(a \, \frac{\mathbbm{1}_{4Q\mathbb N}(q)}{q^{\mu}} \Big)^\beta
= a^{\beta} \, \frac{\mathbbm{1}_{4Q\mathbb N}(q)}{q^{\mu \beta}}
= a^{2/\mu} \, \frac{\mathbbm{1}_{4Q\mathbb N}(q)}{q^2}
 = \psi_{2,Q}(q), 
\end{equation}
with a new underlying constant $a^{2/\mu}$.
Therefore, 
\begin{equation}
( A_{\mu,Q})^\beta
:= \limsup_{q \to \infty } \bigcup_{p \leq q, \, (p,q) = 1} B \Big( \,  \frac{p}{q}, \psi_{\mu,Q}(q)^\beta \Big)
= \limsup_{q \to \infty } \bigcup_{p \leq q, \, (p,q) = 1} B \Big( \,  \frac{p}{q}, \psi_{2,Q}(q) \Big)
= A_{2,Q}.
\end{equation}
Observe that $\beta$ is chosen to be the largest possible exponent that gives $ |  ( A_{\mu,Q})^\beta| = | (A_{\mu \beta,Q})| = 1$. 
Since \eqref{eq:CaseMu2} is independent of $a$, we get $ |  (A_{\mu,Q})^{2/\mu} | = | A_{2,Q} | = 1$, 
and the Mass Transference Principle Theorem~\ref{thm:Mass_Transference_Principle} implies that
$ \mathcal H^{2/\mu} \big(  A_{\mu,Q}  \big) =  \infty$. The proof is complete. 
\end{proof}

\section{Proof of Theorem~\ref{thm:Main_Theorem_Irrationals}: Spectrum of singularities when $x_0 \not\in \mathbb Q$}
\label{sec:Irrationals}

In this section we work with $x_0 \not\in \mathbb Q$ 
and prove Theorem~\ref{thm:Main_Theorem_Irrationals}.
Following the strategy for $x_0 \in \mathbb Q$, 
we first study the H\"older regularity at rational $t$ in Section~\ref{sec:x0_Irrational_Rationals}, 
and at irrational $t$ in Section~\ref{sec:x0_Irrational_Irrationals}

\subsection{Regularity at rational $\boldsymbol{t}$}\label{sec:x0_Irrational_Rationals}

Let $t = p/q$ an irreducible fraction. 
With Corollary~\ref{thm:CorollaryAsymptotic} in mind, 
we now have $x_q = \operatorname{dist} ( x_0, \mathbb Z/q  ) \neq 0$. 
Since $q$ is fixed, 
$\lim_{h \to 0} x_q / |h|^{1/2} = \infty$, so $F_\pm(x) = O(x^{-2})$ implies 
$F_\pm( x_q/\sqrt{|h|} ) \lesssim |h|/x_q^2 $ when $h \to 0$.
Also $|G(p,m_q,q)| \leq \sqrt{2q}$ for all $m_q$.
Hence, 
\begin{equation}
\Big| \, R_{x_0}\Big( \frac{p}{q} + h \Big) - R_{x_0}\Big( \frac{p}{q} \Big) + 2\pi i h \,  \Big| \lesssim  \left(  \frac{1}{\sqrt{q} \, x_q^2} + q^{3/2} \right) \, h^{3/2}.
\end{equation}
This regularity is actually the best we can get.
\begin{prop}
Let $x_0 \in \mathbb R \setminus \mathbb Q$ and let $t \in \mathbb Q$. 
Then,
$\alpha_{x_0}(t) = 3/2$. 
\end{prop}
We postpone the proof of $\alpha_{x_0}(t) \leq 3/2$ to Proposition~\ref{thm:Rationals32}.
In any case, this means that when $x_0 \notin \mathbb Q$, $R_{x_0}$ is more regular at rational points than when $x_0 \in \mathbb Q$.

\subsection{Regularity at irrational $\boldsymbol{t}$}
\label{sec:x0_Irrational_Irrationals}
Let now $t \notin \mathbb Q$. 
Again, we aim at an upper bound for $\alpha_{x_0}(t)$ that complements the lower bound in
Proposition~\ref{thm:Lower_Bound_For_Holder_Regularity}.
by approximating $t \not\in \mathbb Q$ by rationals $p_n/q_n$ and using the asymptotic behavior in Corollary~\ref{thm:CorollaryAsymptotic}. 
However, now $x_0 \not\in \mathbb Q$ implies $x_{q_n} \neq 0$, 
so we cannot directly assume $F_\pm(x_{q_n}/\sqrt{|h_{q_n}|}) \simeq F_\pm(0) \simeq 1$ anymore. 
Therefore, it is fundamental to understand the behavior of the quotient $x_{q_n}/\sqrt{|h_{q_n}|}$. 

\subsubsection{Heuristics}\label{sec:Heuristics}
Let $q \in \mathbb N$ and define the exponents $\mu_q$ and $\sigma_q$ as usual, 
\begin{equation}
x_q = \operatorname{dist} \Big( x_0, \frac{\mathbb Z}{ q} \Big) = \frac{1}{q^{\sigma_q}}, 
\qquad |h_q|  = \operatorname{dist} \Big( t, \frac{\mathbb Z}{ q} \Big) = \frac{1}{q^{\mu_q}}, 
\qquad \Longrightarrow \qquad 
\frac{x_q}{\sqrt{|h_q| }} = \frac{1}{q^{\sigma_q - \mu_q / 2}}. 
\end{equation}
If $\sigma_q - \mu_q/2 > c > 0$ holds for a sequence $q_n$, 
we should recover the behavior when $x_0 \in \mathbb Q$ because
\begin{equation}\label{eq:Heuristic_Argument_Of_F}
\lim_{n \to \infty} \big( \sigma_{q_n} - \frac{\mu_{q_n} }{2} \big) \geq c > 0 
\quad \Longrightarrow \quad 
\lim_{n \to \infty} \frac{x_{q_n}}{\sqrt{|h_{q_n}|}} = 0 
\quad \Longrightarrow \quad 
F_\pm\Big( \frac{x_{q_n}}{\sqrt{|h_{q_n}|}} \Big) \simeq F_\pm(0), \quad n \gg 1. 
\end{equation}
The main term in the asymptotic behavior for $R_{x_0} (t) - R_{x_0}(p_n/q_n)$
in Corollary~\ref{thm:CorollaryAsymptotic} would then be 
\begin{equation}
\text{ Main Term } = \frac{\sqrt{|h_{q_n}|}}{q_n} G(p_n,m_{q_n},q_n) F_\pm(0)
 \simeq \frac{\sqrt{|h_{q_n}|}}{\sqrt{q_n} }
 \simeq h_{q_n}^{\frac12 + \frac{1}{2\mu_{q_n}}}
\end{equation}
if we assume the necessary parity conditions so that $|G(p_n,m_{q_n},q_n)| \simeq \sqrt{q_n}$. 
Recalling the definition of the exponent of irrationality $\mu(\cdot)$ in  \eqref{eq:Exponent_Of_Irrationality},  
we may think of $\sigma_{q_n} \to \mu(x_0)$ and $\mu_{q_n} \to \mu(t)$, 
so these heuristic computations suggest that
$\alpha_{x_0}(t) \leq \frac12 + \frac{1}{2\mu(t)}$ for $t$ such that $\mu(t) \leq 2\mu(x_0)$. 
Since Proposition~\ref{thm:Lower_Bound_For_Holder_Regularity}
 gives $\alpha_{x_0}(t) \geq \frac12 + \frac{1}{2\mu(t)}$,
we may expect that
 \begin{equation}\label{eq:Heuristics_Irrationals}
 \alpha_{x_0}(t) = \frac12 + \frac{1}{2\mu(t)}, \qquad \text{ if } \quad 2 \leq \mu(t) \leq 2\mu(x_0),
 \end{equation}
or at least for a big subset of such $t$. 
 It is less clear what to expect when $\mu(t) > 2\mu(x_0)$, 
since \eqref{eq:Heuristic_Argument_Of_F}
need not hold. 
Actually, 
if $\sigma_{q_n} - \mu_{q_n}/2 < c < 0$ for all sequences, then since $F_\pm(x) = x^{-2} + O(x^{-4})$, 
\begin{equation}
\lim_{n \to \infty} \frac{x_{q_n}}{\sqrt{|h_{q_n}|}} 
= \lim_{n \to \infty} q_n^{\mu_{q_n}/2 - \sigma_{q_n}} = \infty 
\qquad \Longrightarrow \qquad 
F_\pm\Big( \frac{x_{q_n}}{\sqrt{|h_{q_n}|}} \Big) 
\simeq \frac{1}{q_n^{\mu_{q_n} - 2\sigma_{q_n}} } 
= |h_{q_n}|^{ 1 - \frac{2\sigma_{q_n}}{\mu_{q_n}} },
\end{equation}
which in turn would make the main term in $R_{x_0} (t) - R_{x_0}(p_n/q_n)$ be
\begin{equation}
\text{ Main Term } 
=  \frac{\sqrt{h_{q_n}}}{q_n} G(p_n,m_{q_n},q_n) F_\pm\Big( \frac{x_{q_n}}{\sqrt{|h_{q_n}|}} \Big)
\simeq h_{q_n}^{ \frac12 + \frac{1}{2\mu_{q_n} } } \, h_{q_n}^{ 1 - \frac{2\sigma_{q_n}}{\mu_{q_n}} }
\simeq h_{q_n}^{\frac32 - \frac{4\sigma_{q_n} - 1}{2\mu_{q_n}}},
\end{equation}
which corresponds to an exponent $\frac32 -  \frac{4\mu(x_0) - 1}{2\mu(t)}$. 
Together with lower bound in 
Proposition~\ref{thm:Lower_Bound_For_Holder_Regularity}, 
 we would get
$ \frac12 + \frac{1}{2\mu(t)} \leq \alpha_{x_0}(t) \leq \frac32 -  \frac{4\mu(x_0) - 1}{2\mu(t)} $, 
which leaves an open interval for $ \alpha_{x_0}(t) $. 

The main difficulty to materialize the ideas
leading to \eqref{eq:Heuristics_Irrationals}
is that we need the sequence $q_n$ to generate good approximations of both $x_0$ and $t$ simultaneously, 
which a priori may be not possible. 
In the following lines 
we show how we can partially dodge this problem 
to prove Theorem~\ref{thm:Main_Theorem_Irrationals}.

\subsubsection{Proof of Theorem~\ref{thm:Main_Theorem_Irrationals}}
Let $\sigma \geq 2$. 
Recalling the definition of the sets $A_{\mu, \mathcal Q}$ in \eqref{eq:A_Mu_Q}, define 
\begin{equation}
A_{\sigma,  \, \mathbb N \setminus 4\mathbb N} 
= \left\{ \, x \in [0,1] \, : \Big| x - \frac{b}{q} \Big| < \frac{1}{ q^\sigma} \text{ for infinitely many coprime pairs } (b,q) \in \mathbb N \times (\mathbb N \setminus 4\mathbb N) \,   \right\}.
\end{equation}
We first prove that the restriction in the denominators\footnote{This condition, which will be apparent later, comes from parity the conditions for the Gauss sums not to vanish.} 
does not affect the Hausdorff dimension.   
\begin{prop}\label{thm:Dimension_X_Mu}
Let $\sigma \geq 2$. Then, $\operatorname{dim}_{\mathcal H} A_{\sigma,  \, \mathbb N \setminus 4\mathbb N}  = 2/\sigma$. 
Moreover, 
$A_{2,  \, \mathbb N \setminus 4\mathbb N} = (0,1) \setminus \mathbb Q$, hence $|A_{2,  \, \mathbb N \setminus 4\mathbb N} | = 1$. 
If $\sigma > 2$, then $\mathcal H^{2/\sigma}(A_{\sigma,  \, \mathbb N \setminus 4\mathbb N} ) = \infty$ . 
\end{prop}
\begin{proof}
The proof for the upper bound for the Hausdorff dimension is standard. 
Writing 
\begin{equation}
A_{\sigma,  \, \mathbb N \setminus 4\mathbb N}
= \limsup_{   q \to \infty \, \, (q \not\in 4\mathbb N) } \bigcup_{ 1 \leq b < q, \, (b.q)=1 } B\Big( \frac{b}{q}, \frac{1}{q^\sigma} \Big) 
= \bigcap_{Q = 1}^\infty  \bigcup_{ q \geq Q, \,  q \not\in 4\mathbb N }  \Bigg(   \bigcup_{ 1 \leq b < q, \, (b.q)=1  } B\Big( \frac{b}{q}, \frac{1}{q^\sigma} \Big)  \Bigg), 
\end{equation} 
we get an upper bound for the Hausdorff measures using the canonical cover
\begin{equation}\label{eq:Canonical_Cover_Irrationals}
A_{\sigma,  \, \mathbb N \setminus 4\mathbb N}
 \subset \bigcup_{ q \geq Q, \, q \not\in 4\mathbb N }  \Big(   \bigcup_{ 1 \leq b < q } B\Big( \frac{b}{q}, \frac{1}{q^\sigma} \Big)  \Big), \quad \forall Q \in \mathbb N
\quad \Longrightarrow \quad 
\mathcal H ^\beta(A_{\sigma,  \, \mathbb N \setminus 4\mathbb N}) \leq \lim_{Q \to \infty} \sum_{ q \geq  Q } \frac{1}{q^{\sigma\beta - 1}}. 
\end{equation}
Thus, $\mathcal H^\beta(A_{\sigma,  \, \mathbb N \setminus 4\mathbb N}) = 0$ when $ \sigma\beta - 1 > 1$, and consequently 
$\operatorname{dim}_{\mathcal H} A_{\sigma,  \, \mathbb N \setminus 4\mathbb N} \leq 2/\sigma$. 

For the lower bound
we follow the procedure discussed in Section~\ref{sec:Diophantine_Approximation}, 
though unlike in the proof of Proposition~\ref{thm:Dimensions} we do not need the Duffin-Schaeffer theorem here. 
We first study the Lebesgue measure of $A_{\sigma,  \, \mathbb N \setminus 4\mathbb N}$.
From \eqref{eq:Canonical_Cover_Irrationals} with $\beta = 1$, 
we directly get $|A_{\sigma,  \, \mathbb N \setminus 4\mathbb N}|=0$ when $\sigma > 2$. 
When $\sigma = 2$, we get $A_{2,  \, \mathbb N \setminus 4\mathbb N} = A_2 = (0,1) \setminus \mathbb Q$. 
Indeed, if $b_n / q_n$ is the sequence of approximations by continued fractions of $x \in (0,1) \setminus \mathbb Q$,
two consecutive denominators $q_n$ and $q_{n+1}$ are never both even\footnote{If $x = [a_0; a_1, a_2, \ldots]$ is a continued fraction, then $q_0 =1$, $q_1 = a_1$ and $q_n = a_n q_{n-1} + q_{n-2}$ for $n \geq 2$. If $q_N$ and $q_{N+1}$ were both even for some $N$, then $q_{N-1}$ would also be, and by induction $q_0 =1$ would be even.}. 
This means that there is a subsequence 
$b_{n_k}/q_{n_k}$ such that $|x - b_{n_k}/q_{n_k}| < 1/q_{n_k}^2$ and $q_{n_k}$ is odd for all $k\in \mathbb N$. 
In particular, $q_{n_k} \not\in 4\mathbb N$, so 
$(0,1) \setminus \mathbb Q \subset A_{2,  \, \mathbb N \setminus 4\mathbb N}$. 
Hence, 
\begin{equation}\label{eq:X_Sigma_Lebesgue_Measure}
|A_{\sigma,  \, \mathbb N \setminus 4\mathbb N}| = \left\{  \begin{array}{ll}
1, & \sigma \leq 2, \\
0, & \sigma > 2,
\end{array}
\right.
\end{equation}
With this in hand, we use the Mass Transference Principle Theorem~\ref{thm:Mass_Transference_Principle}. 
For $\beta > 0$,  
\begin{equation}
(A_{\sigma,  \, \mathbb N \setminus 4\mathbb N})^\beta 
= \limsup_{  \substack{ q \to \infty \\ q \not\in 4\mathbb N } } \bigcup_{  1 \leq b < q, \, (b,q) = 1  } B\Big( \frac{b}{q}, \Big( \frac{1}{q^\sigma} \Big)^\beta \Big) 
= \limsup_{  \substack{ q \to \infty \\ q \not\in 4\mathbb N } } \bigcup_{ 1 \leq b < q, \, (b,q) = 1  } B\Big( \frac{b}{q},  \frac{1}{q^{\sigma \beta}}  \Big)
= A_{\sigma \beta,  \, \mathbb N \setminus 4\mathbb N}.
\end{equation}
Thus, choosing $\beta = 2/\sigma$ 
we get 
$ (A_{\sigma,  \, \mathbb N \setminus 4\mathbb N})^{2/\sigma} 
=
 A_{2,  \, \mathbb N \setminus 4\mathbb N}$, 
hence by \eqref{eq:X_Sigma_Lebesgue_Measure}
we get
$|(A_{\sigma,  \, \mathbb N \setminus 4\mathbb N})^{2/\sigma} |  = 1$. 
The Mass Transference Principle implies $\operatorname{dim}_{\mathcal H} A_{\sigma,  \, \mathbb N \setminus 4\mathbb N} \geq 2/\sigma$ and $\mathcal H^{2/\sigma} (A_{\sigma,  \, \mathbb N \setminus 4\mathbb N}) = \infty$. 
\end{proof}

Let $x_0 \in A_{\sigma,  \, \mathbb N \setminus 4\mathbb N}$. 
Then there exists a sequence of pairs $(b_n, q_n) \in \mathbb N \times (\mathbb N \setminus 4\mathbb N)$ such that 
$|x_0 - b_n/q_n| < 1/q_n^\sigma$ and moreover
$b_n/q_n$ are all approximations by continued fractions. 
Define 
\begin{equation}
\mathcal Q_{x_0} = \{ \, q_n \, : \, n \in \mathbb N \,   \}
\end{equation} 
to be the set of such denominators. 
This sequence exists because:
\begin{itemize}
	\item if $\sigma = 2$, there is a subsequence of continued fraction approximations with odd denominator, in particular with $q_n \not\in 4\mathbb N$.  
	\item if $\sigma > 2$, by definition there exist a sequence of pairs $(b_n, q_n) \in \mathbb N \times (\mathbb N \setminus 4\mathbb N)$ such that 
	\begin{equation}
\Big| x_0 - \frac{b_n}{q_n} \Big| < \frac{1}{q_n^\mu} \leq \frac{1}{2q_n^2}, \qquad \text{ for large enough } n \in \mathbb N. 
\end{equation}
By a theorem of Khinchin \cite[Theorem 19]{Khinchin1964}, 
all such $b_n/q_n$ are continued fraction approximations of $x_0$. 
\end{itemize}
Since all such $q_n$ are the denominators of continued fraction approximations, 
the sequence $q_n$ grows exponentially.\footnote{We actually have $q_n \geq 2^{n/2}$. 
To see this, rename this sequence as a subsequence $(b_{n_k}/q_{n_k})_k$ of the continued fraction convergents of $x_0$. By the properties of the continued fractions, $q_{n_k} \geq 2^{n_k/2}$. Since $n_k \geq k$, we get $q_{n_k} \geq 2^{k/2}$. }
Following again the notation in \eqref{eq:A_Mu_Q}
in Section~\ref{sec:Diophantine_Approximation}, 
for $\mu \geq 1$ and $0 < c < 1/2$, let\footnote{
 When $\mu = \infty$ the definition is adapted as usual as $A_{\infty,Q_{x_0}} = \cap_\mu A_{\mu,Q_{x_0}}$.
 Proofs for forthcoming results are written for $\mu < \infty$, but the simpler $\mu = \infty$ case is proved the same way we did in Section~\ref{sec:Rational_x0_Irrational_t}.} 
\begin{equation}
A_{\mu,\mathcal Q_{x_0}} 
 = \left\{  \, t \in [0,1] \, : \Big| t - \frac{p}{q} \Big| < \frac{c}{q^\mu} \text{ for infinitely many coprime pairs  } (p,q) \in  \mathbb N \times \mathcal Q_{x_0} \,    \right\}.
\end{equation}

\begin{prop}\label{thm:Dimension_A_Mu_Q_x0}
For $\mu \geq 1$, $\operatorname{dim}_{\mathcal H} (A_{\mu,\mathcal Q_{x_0}} ) =1/\mu$. 
\end{prop}
\begin{proof}
As in the proof of Proposition~\ref{thm:Dimension_X_Mu}, 
the upper bound follows from the limsup expression
$ A_{\mu,\mathcal Q_{x_0}}  = \limsup_{n \to \infty} \bigcup_{  1 \leq p \leq q_n,  \,  (p,q_n) = 1 } B( p/q_n, c/q_n^\mu )$
and its canonical covering 
\begin{equation}\label{eq:Canonical_Covering_A_MuQx0}
A_{\mu,\mathcal Q_{x_0}}   \subset  \bigcup_{n \geq N}  \bigcup_{1 \leq p \leq q_n} B\Big( \frac{p}{q_n}, \, \frac{c}{q_n^\mu}  \Big), \quad \forall N \in \mathbb N
\quad \Longrightarrow \quad
\mathcal H^\beta \big( A_{\mu,\mathcal Q_{x_0}}   \big) \leq c^\beta \lim_{N \to \infty}  \sum_{n = N}^\infty \frac{1}{q_n^{\mu \beta - 1}}. 
\end{equation}
Since $q_n \geq 2^{n/2}$, the series converges if and only if $\mu \beta - 1 > 0$. 
Thus, $\mathcal H^\beta (A_{ \mu,\mathcal Q_{x_0}}) = 0$ for all $\beta > 1/\mu$, 
hence $\operatorname{dim}_{\mathcal H} (A_{\mu,\mathcal Q_{x_0}}) \leq 1/\mu$. 

For the lower bound we follow again the procedure in Section~\ref{sec:Diophantine_Approximation}. 
First we compute the Lebesgue measure of $A_{\mu,\mathcal Q_{x_0}}$.
From \eqref{eq:Canonical_Covering_A_MuQx0} with $\beta = 1$ 
we get $|A_{\mu,\mathcal Q_{x_0}}| = 0$ if $\mu > 1$. 
When $\mu \leq 1$, we need the full strength of the Duffin-Schaeffer theorem proved by Koukoulopoulos and Maynard \cite{KoukoulopoulosMaynard2020} (see Theorem~\ref{thm:Duffin_Schaeffer} in this paper).
Indeed, we have $|A_{\mu,\mathcal Q_{x_0}}| = 1$ if and only if
$\sum_{n=1}^\infty \varphi(q_n)/q_n^\mu = \infty$, 
and otherwise $|A_{\mu,\mathcal Q_{x_0}}| = 0$.
If $\mu < 1$, we use one of the classic properties of Euler's totient function, 
namely that for $\epsilon = (1-\mu)/2 > 0$ there exists $N \in \mathbb N$
such that $\varphi(n) \geq n^{1-\epsilon}$ for all $n \geq N$. 
In particular, there exists $K \in \mathbb N$ such that 
\begin{equation}
\sum_{n=1}^\infty \frac{\varphi(q_n)}{q_n^\mu} 
\geq 
\sum_{n=K}^\infty \frac{\varphi(q_n)}{q_n^\mu} 
\geq 
\sum_{n=K}^\infty q_n^{1- \mu - \epsilon}
\geq 
\sum_{n=K}^\infty 1 
= \infty, 
\end{equation} 
so $|A_{\mu,\mathcal Q_{x_0}}| = 1$ if $\mu < 1$. 
For $\mu=1$, none of these arguments work, 
and we need to know the behavior of $\varphi(q_n)$ for $q_n \in \mathcal Q_{x_0}$, of which we have little control.
So independently of $c > 0$, 
\begin{equation}\label{eq:Measure_A_Mu_Q}
|A_{\mu,\mathcal Q_{x_0}}| = 
\left\{  \begin{array}{ll}
1, & \mu < 1, \\
?, & \mu = 1, \\
0, & \mu > 1. 
\end{array}
\right.
\end{equation}
Even not knowing $|A_{1,\mathcal Q_{x_0}}| $, 
the Mass Transference Principle Theorem~\ref{thm:Mass_Transference_Principle}
allows us to compute 
the Hausdorff dimension of $A_{\mu,\mathcal Q_{x_0}}$ 
from \eqref{eq:Measure_A_Mu_Q}. 
As usual,  
dilate the set with an exponent $\beta>0$:
\begin{equation}
(A_{\mu,\mathcal Q_{x_0}})^\beta 
= \limsup_{n \to \infty} \bigcup_{  1 \leq p \leq q_n } B \Big( \frac{p}{q_n}, \Big( \frac{c}{q_n^\mu} \Big)^\beta  \Big)
= \limsup_{n \to \infty} \bigcup_{1 \leq p \leq q_n } B \Big( \frac{p}{q_n}, \frac{c^\beta}{q_n^{\mu\beta}}  \Big)
= A_{\mu \beta,\mathcal Q_{x_0}}, 
\end{equation} 
with a new constant $c^\beta$.  
Since \eqref{eq:Measure_A_Mu_Q} is independent of $c$,  
we have $|(A_{\mu,\mathcal Q_{x_0}})^\beta| = |A_{\mu \beta,\mathcal Q_{x_0}}| = 1$
if $\mu \beta < 1$, 
and the Mass Transference Principle implies 
$\operatorname{dim}_{\mathcal H} A_{\mu,\mathcal Q_{x_0}} \geq \beta$. 
Taking $\beta \to 1/\mu$, we deduce   
$\operatorname{dim}_{\mathcal H} A_{\mu,\mathcal Q_{x_0}} \geq 1/\mu$. 
\end{proof}

As in Proposition~\ref{thm:Holder_Regularity_In_Bmu} and in the definition of $B_{\mu,\mathcal Q}$ in \eqref{eq:B_Mu_Q}, 
to get information about $\alpha_{x_0}(t)$ for $t \in A_{\mu,\mathcal Q_{x_0}}$
we need to restrict their exponent of irrationality.  
We do this by removing sets $A_{\mu + \epsilon}$ defined in \eqref{eq:A_mu_Intro}. 
However, compared to Proposition~\ref{thm:Holder_Regularity_In_Bmu} we have two fundamental difficulties:
\begin{enumerate}
	\item[(a)] The dimensions $\operatorname{dim}_{\mathcal H} A_\mu = 2/\mu > 1/\mu =  \operatorname{dim}_{\mathcal H} A_{\mu,\mathcal Q_{x_0}}$  do not match anymore.  
	\item[(b)] Because do not know the Lebesgue measure of $A_{1,\mathcal Q_{x_0}}$ 
in \eqref{eq:Measure_A_Mu_Q}, 
we cannot conclude that 
$\mathcal H^{1/\mu} (A_{\mu,\mathcal Q_{x_0}}) = \infty$ if $\mu > 1$. 
\end{enumerate}
To overcome these difficulties, 
let $\delta_1, \delta_2 > 0$ and define the set 
\begin{equation}\label{eq:B_Mu_Q_x0}
B_{\mu, \mathcal Q_{x_0}} ^{\delta_1, \delta_2}
= \Big(  A_{\mu, \mathcal Q_{x_0}} \setminus A_{\mu + \delta_1,\mathcal Q_{x_0}}  \Big) 
 \setminus \Big( \bigcup_{\epsilon >0} A_{2\mu + \delta_2 + \epsilon} \Big).
\end{equation}

\begin{rem}[Explanation of the definition of $B_{\mu, \mathcal Q_{x_0}} ^{\delta_1, \delta_2}$]\label{rmk:Good_Diophantine_Set}
The role of $\delta_2$ is to avoid the problem (b) above, 
while $\delta_1$ has a technical role when controlling the behavior of $F_\pm(x_{q_n}/\sqrt{h_{q_n}})$ in \eqref{eq:Controlling_F}.
Last, 
we remove $A_{2\mu+\epsilon}$ instead of $A_{\mu + \epsilon}$
to avoid problem (a) and to ensure that $B_{\mu, \mathcal Q_{x_0}} ^{\delta_1, \delta_2}$ is not too small.
The downside of this is that we can only get $\mu(t) \in [\mu, 2\mu + \delta_2]$ for the exponent of irrationality of $t \in B_{\mu, \mathcal Q_{x_0}} ^{\delta_1, \delta_2}$. 
If instead we worked with the set
\begin{equation}\label{eq:B_s_Q_x0_Tilde}
\widetilde B_{\mu, \mathcal Q_{x_0}}^{\delta_1}
=  \Big(  A_{\mu, \mathcal Q_{x_0}} \setminus A_{\mu + \delta_1,\mathcal Q_{x_0}}  \Big)  \setminus \Big( \bigcup_{\epsilon >0} A_{\mu + \epsilon} \Big)
\end{equation}
we would deduce $\mu(t) =\mu$ and therefore $\alpha_{x_0}(t) = 1/2 + 1/(2\mu)$. 
However, we do not know how to compute the dimension of $\widetilde B^{\delta_1}_{\mu, \mathcal Q_{x_0}}$.

\end{rem}

\begin{prop} \label{thm:Holder_Regularity_For_Irrationals}
Let $\mu \geq 1$. Then, 
\begin{enumerate}
	\item [(a)] $\operatorname{dim}_{\mathcal H} B^{\delta_1, \delta_2}_{\mu, \mathcal Q_{x_0}} = 1/\mu$. 
	\item [(b)] If $t \in B^{\delta_1, \delta_2}_{\mu, \mathcal Q_{x_0}}$, then $\alpha_{x_0}(t) \geq \frac12 + \frac{1}{4\mu + 2\delta_2}$. 

	\item[(c)] If $2 \leq \mu < 2\sigma - \delta_1$ and  $t \in B^{\delta_1, \delta_2}_{\mu, \mathcal Q_{x_0}}$, then $\alpha_{x_0}(t) \leq \frac12 + \frac{1}{2\mu}$.
\end{enumerate}
\end{prop}

\begin{proof}[Proof of Proposition~\ref{thm:Holder_Regularity_For_Irrationals}]
$(a)$  
The inclusion $B^{\delta_1, \delta_2}_{\mu, \mathcal Q_{x_0}} \subset A_{\mu, \mathcal Q_{x_0}}$ directly implies 
$ \operatorname{dim}_{\mathcal H} B^{\delta_1, \delta_2}_{\mu, \mathcal Q_{x_0}} \leq 1/\mu$. 
We prove the lower bound following the proof of Proposition~\ref{thm:Dimensions} in a few steps:
\begin{itemize}
	
	\item[(a.1)] Since $\operatorname{dim}_{\mathcal H}  A_{\mu + \delta_1, \mathcal Q_{x_0} } = 1/(\mu + \delta_1) < 1/\mu$, 
	we have $\operatorname{dim}_{\mathcal H}  ( A_{\mu, \mathcal Q_{x_0}} \setminus A_{\mu + \delta_1, \mathcal Q_{x_0} } ) = 1/\mu$.

	\item[(a.2)] The sets $A_\mu$ are nested, 
so by the Jarnik-Besicovitch Theorem~\ref{thm:Jarnik_Besicovitch}
\begin{equation}\label{eq:Dimension_With_Delta}
\operatorname{dim}_{\mathcal H}  \Big(\bigcup_{\epsilon > 0} A_{2\mu + \delta_2 + \epsilon} \Big)
= \sup_{n \in \mathbb N} \left\{ \operatorname{dim}_{\mathcal H} \Big( A_{2\mu + \delta_2 + \frac{1}{n}}  \Big) \right\}
= \sup_{n \in \mathbb N}  \frac{2}{2\mu + \delta_2 + \frac{1}{n}} 
= \frac{1}{\mu + \delta_2 / 2}. 
\end{equation}
Moreover, 
$\mathcal H^\gamma  \big(\bigcup_{\epsilon > 0} A_{2\mu + \delta_2 + \epsilon} \big) 
= \lim_{n  \to \infty} \mathcal H^\gamma \big( A_{2\mu + \delta_2 + 1/n} \big) = 0$ for all $\gamma \geq 1/(\mu + \delta_2/2)$.

\end{itemize}
Take $\gamma$ such that 
$1/ (\mu + \delta_2 /2) < \gamma < 1/\mu$. 
From (a.1)
we get $\mathcal H^{\gamma}  ( A_{\mu, \mathcal Q_{x_0}} \setminus A_{\mu + \delta_1, \mathcal Q_{x_0} } ) = \infty$, 
and from (a.2) we have 
$\mathcal H^{\gamma} \big(\bigcup_{\epsilon > 0} A_{2\mu + \delta_2 + \epsilon} \big) = 0 $,  
so
\begin{equation}
\mathcal H^\gamma (B^{\delta_1, \delta_2}_{\mu, \mathcal Q_{x_0}})
= \mathcal H^\gamma  ( A_{\mu, \mathcal Q_{x_0}} \setminus A_{\mu + \delta_1, \mathcal Q_{x_0} } )- \mathcal H^\gamma \Big(\bigcup_{\epsilon > 0} A_{2\mu + \delta  + \epsilon} \Big) > 0.
\end{equation}
Consequently $\operatorname{dim}_{\mathcal H} B^{\delta_1,\delta_2}_{\mu, \mathcal Q_{x_0}} \geq \gamma$,
and taking $\gamma \to 1/\mu$ we conclude  $\operatorname{dim}_{\mathcal H} B^{\delta_1,\delta_2}_{\mu, \mathcal Q_{x_0}} \geq 1/\mu$.

$(b)$ Let $t \in B^{\delta_1, \delta_2}_{\mu, \mathcal Q_{x_0}}$. 
Then,
$t \notin \bigcup_{\epsilon > 0} A_{2\mu + \delta_2 + \epsilon}$ implies $\mu(t) \leq 2\mu + \delta_2$, 
where $\mu(t)$ is the exponent of irrationality of $t$. 
Combining this with Proposition~\ref{thm:Lower_Bound_For_Holder_Regularity} we get
$ \alpha_{x_0}(t) \geq \frac12 + \frac{1}{2\mu(t)} \geq \frac12 + \frac{1}{4\mu + 2\delta_2}$.

$(c)$ Let $t \in B^{\delta_1, \delta_2}_{\mu, \mathcal Q_{x_0}}$. 
Since $t \in  A_{\mu, \mathcal Q_{x_0}} \setminus A_{\mu + \delta_1, \mathcal Q_{x_0} } $, 
there is a subsequence of denominators $(q_{n_k})_k  \subset \mathcal Q_{x_0}$ 
such that $c/q_{n_k}^{\mu + \delta_1} \leq \big|  t - p_{n_k}/q_{n_k} \big| < c/q_{n_k}^\mu$ for $k \in \mathbb N$. 
Define the errors $h_{n_k}$ and $x_{n_k}$, and the exponent $\mu_{n_k}$ as
\begin{equation}\label{eq:Errors_Irrationals}
h_{n_k} =  t - \frac{p_{n_k}}{q_{n_k}}, 
\qquad 
|h_{n_k}| 
= \frac{1}{q_{n_k}^{\mu_{n_k}}}
\qquad  \text{ and } \qquad 
x_{n_k} = \Big| x_0 - \frac{b_{n_k}}{q_{n_k}} \Big| < \frac{1}{q_{n_k}^{\sigma}}. 
\end{equation}
From the condition above, since $c < 1$, 
we immediately get that for any $\epsilon >0$,
\begin{equation}\label{eq:s_n_k_For_Large_k}
\mu < \mu_{n_k} \leq \mu + \delta_1 + \epsilon, \qquad \forall k \gg_\epsilon 1. 
\end{equation}
By the asymptotic expansion in Corollary~\ref{thm:CorollaryAsymptotic}, we have 
\begin{equation}\label{eq:Main_And_Error_Irrationals}
R_{x_0}(t) - R_{x_0}\Big( \frac{p_{n_k}}{q_{n_k}}  \Big) 
= \frac{|h_{n_k}|^{1/2}}{q_{n_k}} \, G(p_{n_k}, b_{n_k}, q_{n_k}) \, F_\pm \Big(  \frac{x_{n_k}}{\sqrt{h_{n_k}}}  \Big) - 2\pi i h_{n_k} 
+ \text{Error},
\end{equation}
where $\text{Error} = O\Big( \min \big( q_{n_k}^{3/2}\, h_{n_k}^{3/2}, q_{n_k}^{1/2} \, h_{n_k} \big) \Big)$. 
Let us treat the elements in this expression separately.
\begin{itemize}
	\item Since $q_{n_k} \not\in 4\mathbb N$, we have $|G(p_{n_k},b_{n_k}, q_{n_k})| \geq \sqrt{q_{n_k}}$ for $k \in \mathbb N$. 
	Indeed, if $q_{n_k}$ is odd, then $|G(p_{n_k},b_{n_k}, q_{n_k})| = \sqrt{q_{n_k}}$.  
	If $q_{n_k} \equiv 2 \pmod{4}$, then $b_{n_k}$ is odd, so $q_{n_k}/2 \equiv b_{n_k} \pmod{2}$ and hence
$|G(p_{n_k},b_{n_k}, q_{n_k})| = \sqrt{2 q_{n_k}}$. 
	 Also, by \eqref{eq:Errors_Irrationals} and  \eqref{eq:s_n_k_For_Large_k},
	\begin{equation}\label{eq:Controlling_F}
	\frac{x_{n_k}}{\sqrt{|h_{n_k}|}}  = x_{n_k} \, q_{n_k}^{\mu_{n_k}/2} 
	< \frac{q_{n_k}^{\mu_{n_k}/2}}{q_{n_k}^\sigma}
	\leq \frac{q_{n_k}^{ \frac{\mu}{2} + \frac{\delta_1}{2} + \frac{\epsilon}{2}}}{q_{n_k}^\sigma}
	=   \frac{1}{q_{n_k}^{\sigma - \frac{\mu}{2} - \frac{\delta_1}{2} - \frac{\epsilon}{2}}}. 
	\end{equation}
	Hence, if $2\sigma > \mu + \delta_1$, we can choose $\epsilon = \sigma - \mu/2 - \delta_1 / 2 > 0$ and we get 
	\begin{equation}
	\lim_{k \to \infty} \frac{x_{n_k}}{\sqrt{|h_{n_k}|}}  
	\leq \lim_{k \to \infty} \frac{1}{q_{n_k}^{\sigma - \mu/2 - \delta_1/2  - \epsilon / 2}} 
	=  \lim_{k \to \infty} \frac{1}{q_{n_k}^{(\sigma - \mu/2 - \delta_1/2)/ 2}} 
	= 0.
	\end{equation}
	Since $F_\pm$ is continuous, we get $	| F_\pm (  x_{n_k}/|h_{n_k}|^{1/2} ) | \geq |F_\pm(0)| / 2 \simeq 1$ for all $ k \gg 1$. 
	Therefore, 
	\begin{equation}
	\text{Main term} 
	= \Big| \frac{\sqrt{|h_{n_k}|}}{q_{n_k}} \, G(p_{n_k}, b_{n_k}, q_{n_k}) \, F \Big(  \frac{x_{n_k}}{|h_{n_k}|^{1/2}}  \Big) \Big|
	\simeq \frac{\sqrt{|h_{n_k}|}}{\sqrt{q_{n_k}}}, \qquad \forall k \gg 1. 
	\end{equation}
	
	\item The term $2\pi i h_{n_k}$ is absorbed by the Main Term if $ |h_{n_k}| \ll \sqrt{|h_{n_k}|}/\sqrt{q_{n_k}}$, which is equivalent to $|h_{n_k}| \ll 1/q_{n_k}$.  
	If $ \mu > 1$, we get precisely $|h_{n_k}| < c/q_{n_k}^\mu \ll 1/q_{n_k}$. 

	\item Regarding the error term, we can write
	\begin{equation}
		 q_{n_k}^{1/2} |h_{n_k}| = \frac{\sqrt{|h_{n_k}|}}{\sqrt{q_{n_k}}} \, (q_{n_k}^2 |h_{n_k}|)^{1/2}, 
		 \qquad 
		 q_{n_k}^{3/2} |h_{n_k}|^{3/2} = \frac{\sqrt{|h_{n_k}|}}{\sqrt{q_{n_k}}} \, q_{n_k}^2 |h_{n_k}|. 	
\end{equation}		
Since $\text{Error} \leq C \, \min \big( q_{n_k}^{3/2}\, |h_{n_k}|^{3/2}, q_{n_k}^{1/2} \, |h_{n_k}| \big) $ for some $C > 0$, 
the error is absorbed by the Main Term if $q_{n_k}^2\, |h_{n_k}| \leq c$ for a small enough, but universal constant $c$. 
Choosing $c>0$ in the definition of $A_{\mu,\mathcal Q_{x_0}}$,  
the condition $|h_{n_k}| \leq c/q_{n_k}^\mu \leq c/q_{n_k}^2$ is satisfied if $\mu \geq 2$. 
\end{itemize}
Hence, if $2 \leq \mu  < 2\sigma - \delta_1$
and $t \in B^{\delta_1,\delta_2}_{\mu,\mathcal Q_{x_0}}$, then
$| R_{x_0}(t) - R_{x_0}( p_{n_k}/q_{n_k} ) | \gtrsim \sqrt{|h_{n_k}|}/\sqrt{q_{n_k}}$ for all $k \gg 1$. 
From \eqref{eq:s_n_k_For_Large_k} we have $1/\sqrt{ q_{n_k} } = |h_{n_k}|^{1/(2\mu_{n_k})} > |h_{n_k}|^{1/(2\mu)}$, so
$| R_{x_0}(t) - R_{x_0}( p_{n_k}/q_{n_k}  ) | \gtrsim |h_{n_k}|^{\frac12 + \frac{1}{2\mu}}$ for all $k \gg 1$, 
which implies $\alpha_{x_0}(t) \leq \frac12 + \frac{1}{2\mu}$.
\end{proof}

From Proposition~\ref{thm:Holder_Regularity_For_Irrationals} we can deduce the main part of Theorem~\ref{thm:Main_Theorem_Irrationals}. 
\begin{thm}\label{thm:Theorem_For_Irrationals}
Let $\sigma \geq 2$ and let $x_0 \in A_{\sigma,  \, \mathbb N \setminus 4\mathbb N}$. 
Let $2 \leq  \mu  < 2\sigma$. Then, for all $\delta > 0$, 
\begin{equation}
\frac{1}{\mu} \leq \operatorname{dim}_{\mathcal H} \left\{   \, t \, : \frac12 + \frac{1}{4\mu} - \delta  \leq \alpha_{x_0}(t) \leq \frac12 + \frac{1}{2\mu}    \right\}  \leq \frac{2}{\mu}.
\end{equation}
\end{thm} 
\begin{proof}
Choose $\delta_2 > 0$ 
and any $\delta_1 < 2\sigma - \mu$. 
Hence, $2 \leq \mu < 2\sigma - \delta_1 $
and Proposition~\ref{thm:Holder_Regularity_For_Irrationals} implies 
\begin{equation}
B^{\delta_1,\delta_2}_{\mu,\mathcal Q_{x_0}} 
 \subset \left\{   \, t \, : \frac12 + \frac{1}{4\mu + 2\delta_2} \leq \alpha_{x_0}(t) \leq \frac12 + \frac{1}{2\mu}    \right\}.
\end{equation}
Since $\operatorname{dim}_{\mathcal H} B^{\delta_1, \delta_2}_{\mu,\mathcal Q_{x_0}}  = 1/\mu$
and $\delta_2$ is arbitrary, 
we get the lower bound.
Let us now prove the upper bound. 
If $\alpha_{x_0}(t) \leq \frac12 + \frac{1}{2\mu}$, 
by Proposition~\ref{thm:Lower_Bound_For_Holder_Regularity} we get
$ \frac12 + \frac{1}{2\mu(t)} \leq \alpha_{x_0}(t) \leq \frac12 + \frac{1}{2\mu} $, 
hence $\mu(t) \geq \mu$. This implies $t \in A_{\mu-\epsilon}$ for all $\epsilon >0$,
so by the Jarnik-Besicovitch Theorem~\ref{thm:Jarnik_Besicovitch} we get
\begin{equation}
 \operatorname{dim}_{\mathcal H} \left\{   \, t \, : \frac12 + \frac{1}{4\mu} - \delta \leq \alpha_{x_0}(t) \leq \frac12 + \frac{1}{2\mu}    \right\}
  \leq
   \operatorname{dim}_{\mathcal H}  A_{\mu - \epsilon}
   = \frac{2}{\mu - \epsilon} 
\end{equation}
for all $\delta \geq 0$. We conclude by taking the limit $\epsilon \to 0$. 
\end{proof}

To get the precise statement of Theorem~\ref{thm:Main_Theorem_Irrationals}, 
we only need to relate the sets $A_{\sigma,  \, \mathbb N \setminus 4\mathbb N}$ with the exponent 
$\sigma(x_0) = \limsup_{n \to \infty} \{ \, \mu_n \, : \, q_n \not\in 4\mathbb N \, \}$ defined in \eqref{eq:Exponent_Of_Irrationality_Alternative}. 
We proceed as follows.
Since $\{A_{\sigma,  \, \mathbb N \setminus 4\mathbb N}\}_{\sigma \geq 2}$ is a nested family
and $A_{2,  \, \mathbb N \setminus 4\mathbb N} = (0,1) \setminus \mathbb Q$, 
for every $x_0 \in (0,1) \setminus \mathbb Q$ there exists $\widetilde{\sigma}(x_0) = \sup\{ \, \sigma \, : \, x_0 \in A_{\sigma,  \, \mathbb N \setminus 4\mathbb N} \,  \}$. 
Let us check that $\sigma(x_0) = \widetilde \sigma (x_0)$.
Indeed, call $\widetilde \sigma(x_0) = \widetilde\sigma$. 

$\bullet\, \, $ If $\widetilde\sigma > 2$. 
Then for $\epsilon > 0$ small enough there exists a sequence $b_k/q_k$ such that $q_k \not\in 4\mathbb N$ and $|x_0 -  b_k/q_k| < 1/q_k^{\widetilde\sigma - \epsilon} < 1/(2q_k^2)$. 
By Khinchin's theorem \cite[Theorem 19]{Khinchin1964}, $b_k/q_k$ is an approximation by continued fraction, for which $|x_0 -  b_k/q_k| = 1/q_k^{\mu_k} < 1/q_k^{\widetilde\sigma - \epsilon}$,
 and therefore $\mu_k \geq \widetilde\sigma - \epsilon$. 
This implies $\sigma(x_0) \geq \widetilde\sigma - \epsilon $ for all $\epsilon > 0$, hence $ \sigma(x_0) \geq \widetilde\sigma$.
On the other hand, 
for all approximations by continued fractions with $q_n \not\in 4\mathbb N$ with large enough $n$
we have
$|x_0 -  b_n/q_n| = 1/q_n^{\mu_n} > 1/q_n^{\widetilde\sigma + \epsilon}$, hence $\mu_n  \leq \widetilde\sigma + \epsilon$. 
This holds for all $\epsilon > 0$, so $\sigma(x_0) \leq \widetilde\sigma$. 

$\bullet\, \, $  If $\widetilde\sigma = 2$, then  $|x_0 -  b_n/q_n| = 1/q_n^{\mu_n} > 1/q_n^{2 + \epsilon}$, hence $\mu_n \leq 2+\epsilon$, for all approximations by continued fractions with $q_n \not\in 4\mathbb N$.
Therefore, $\sigma(x_0) \leq 2$. Since $\sigma(x_0) \geq 2$ always holds, we conclude. 

Therefore, 
let $x_0 \in (0,1) \setminus \mathbb Q$. 
Then, $x_0 \in A_{\sigma,  \, \mathbb N \setminus 4\mathbb N}$ for all $\sigma < \sigma(x_0)$, 
so the conclusion of Theorem~\ref{thm:Theorem_For_Irrationals}
holds for $2 \leq \mu < 2\sigma$, for all $\sigma < \sigma(x_0)$. 
That implies that for every $\delta > 0$, 
\[
\frac{1}{\mu} \leq \operatorname{dim}_{\mathcal H} \left\{   \, t \, : \frac12 + \frac{1}{4\mu} - \delta  \leq \alpha_{x_0}(t) \leq \frac12 + \frac{1}{2\mu}    \right\}  \leq \frac{2}{\mu}, \qquad \text{ for all } \qquad 2 \leq \mu < 2\sigma(x_0). 
\]

\section{Proof of Theorem~\ref{thm:Main_Theorem_Rationals_Intermittency} - The high-pass filters when $x_0 \in \mathbb Q$}
\label{sec:Rational_High_Pass_Filters}
In this section we prove Theorem~\ref{thm:Main_Theorem_Rationals_Intermittency}. 
For that,  we compute the $L^p$ norms of the high-pass filters of $R_{x_0}$
when $x_0 \in \mathbb Q$. 
In Section~\ref{sec:High_Pass_Filter_And_Localized_Estimates}
we define Fourier high-pass filters using smooth cutoffs, 
reduce the computation of their $L^p$ norms to the study of Fourier localized $L^p$ estimates, 
state such localized estimates and deduce Theorem~\ref{thm:Main_Theorem_Rationals_Intermittency} from them.  
We prove such localized estimates in Section~\ref{sec:Proof_Of_Lp_Norms_Localized_Rationals}.

\subsection{High-pass filters and frequency localization}
\label{sec:High_Pass_Filter_And_Localized_Estimates}
We begin with the definition of high-pass filters we use in the proofs. 
Let $\phi \in C^\infty$ a positive and even cutoff with support on $[-1,1]$ and such that $\phi(x) = 1$ on $x\in [-1/2,1/2]$.  
Let $\psi(x) = \phi(x/2) - \phi(x)$, 
and 
 \begin{equation}\label{eq:Partition_Of_Unity}
\psi_{-1}(x) = \frac{\phi(x)}{\phi(x) + \sum_{i \in \mathbb N} \psi ( x/2^i )},
 \qquad 
\psi_k(x) = \frac{\psi ( x/2^k )}{\phi(x) + \sum_{i \in \mathbb N} \psi( x/2^i )}, 
\qquad \text{ for } k \geq 0, 
\end{equation}
so that we have the partition of unity $\sum_{k = -1}^\infty \psi_k(x) = 1$.
For $k \geq 0$, $\psi_k$ is supported on $[-2^{k+1},-2^{k-1}] \cup [2^{k-1},2^{k+1}]$. 
Let $f$ be a periodic function with Fourier series $f(t) = \sum_{ n \in \mathbb Z} a_n e^{2\pi i n t}$.
With the partition of unity above, we perform a Littlewood-Paley decomposition 
\begin{equation}
f(t) = \sum_{k = -1}^\infty P_kf(t), \qquad \text{ where } \qquad P_kf(t) = \sum_{n \in \mathbb Z} \psi_k(n) a_n e^{2\pi i n t}. 
\end{equation} 
The Fourier high-pass filter at frequency $N \in \mathbb N$
is roughly $P_{\geq N} f(t) = \sum_{k \geq \log N} P_kf(t)$. 
Let us be more precise working directly with $R_{x_0}$, 
whose frequencies in $t$ are squared.  
Let $N \in \mathbb N$ be large, 
and define $k_N$ to be the unique $k_N \in \mathbb N$ such that $2^{k_N} \leq \sqrt{N} < 2^{k_N+1}$. 
We define the high-pass filter of $R_{x_0}$ at frequency $N$ as 
\begin{equation}\label{eq:High_Pass_Filter}
P_{\geq N}  R_{x_0}(t) = \sum_{k \geq k_N} P_k R_{x_0}(t), 
\qquad \text{ where } \qquad P_kR_{x_0}(t) = \sum_{n \in \mathbb N } \psi_k(n) \frac{e^{2\pi i (n^2t + nx_0)}}{n^2}. 
\end{equation}
We first estimate $\lVert P_k R_{x_0} \rVert_p$
and  then extend the result to estimate $\lVert P_{\geq N} R_{x_0} \rVert_p$. 

\begin{rem}
At a first glance, using pure Littlewood-Paley blocks in the definition for high-pass filters in \eqref{eq:High_Pass_Filter} may seem restrictive, 
since it is analogue to estimating high-frequency cutoffs only for a sequence $N_k \simeq 2^k \to \infty$.  
However, the estimates we give depend only on the $L^1$ norm of the cutoff $\psi$, 
so slightly varying the definition and support of $\psi$ does not affect the estimates. 
This is analogous to having a cutoff $\Phi(x/N)$ for a fixed $\Phi$ as we state in the introduction.
\end{rem}

We now state the estimates for the frequency localized $L^p$ estimates. 
For the sake of generality, 
let $\Psi \in C^\infty$ be compactly supported outside the origin
and bounded below in an interval of its support (for instance, $\psi$ defined above).  
\begin{thm}\label{thm:Lp_Norms_Upper_And_Rationals}
Let $x_0 \in \mathbb R$.
Then, for $N \gg 1$, 
 \begin{equation}\label{eq:LpNorm_Upper}
\Big\lVert  \sum_{n \in \mathbb Z} \Psi \big(\frac{n}{N}\big)\,  e^{2\pi i (n^2 \, t + n \, x_0)}  \Big\rVert_{L^p(0,1)}^p
\lesssim 
  \left\{
 \begin{array}{ll}
 N^{p-2}, & \text{ when } p > 4, \\
 N^2 \log N, & \text{ when } p=4, \\
 N^{p/2}, & \text{ when } p < 4. 
 \end{array}
 \right.
 \end{equation}
 When $p=2$, the upper bound is sharp, that is, 
$
\big\lVert  \sum_{n \in \mathbb Z} \Psi (n/N)\,  e^{2\pi i (n^2 \, t + n \, x_0)}  \big\rVert_{L^2(0,1)}^2
\simeq N
$. 

If $x_0 \in \mathbb Q$, then the upper bound is sharp. 
That is, if $x_0 = P/Q$ with $(P,Q)=1$, then 
 \begin{equation}\label{eq:LpNorm_Rationals}
\Big\lVert  \sum_{n \in \mathbb Z} \Psi\big(\frac{n}{N}\big)\,  e^{2\pi i (n^2 \, t + n \, x_0)}  \Big\rVert_{L^p(0,1)}^p
\simeq_Q
  \left\{
 \begin{array}{ll}
 N^{p-2}, & \text{ when } p > 4, \\
 N^2 \log N, & \text{ when } p=4, \\
 N^{p/2}, & \text{ when } p < 4. 
 \end{array}
 \right.
 \end{equation}
\end{thm}
\begin{rem}
All estimates in Theorem~\ref{thm:Lp_Norms_Upper_And_Rationals} depend on $\lVert \Psi \rVert_1$ due to Lemma~\ref{thm:BourgainLemma}. 
\end{rem}
We postpone the proof of Theorem~\ref{thm:Lp_Norms_Upper_And_Rationals} to Section~\ref{sec:Proof_Of_Lp_Norms_Localized_Rationals}. 
and use it now to compute the $L^p$ norms of the high-pass filters 
$\lVert P_{\geq N} R_{x_0} \rVert_p$
and therefore to prove Theorem~\ref{thm:Main_Theorem_Rationals_Intermittency}. 

\begin{proof}[Proof of Theorem~\ref{thm:Main_Theorem_Rationals_Intermittency}]
Denote the estimate for $x_0 \in \mathbb Q$ 
on \eqref{eq:LpNorm_Rationals} 
in Theorem~\ref{thm:Lp_Norms_Upper_And_Rationals}
by
\begin{equation}\label{eq:Lp_Rationals_G}
\big\lVert  \sum_{n \in \mathbb Z} \Psi(n/N)\,  e^{2\pi i (n^2 \, t + n \, x_0)}  \big\rVert_{L^p(0,1)}^p
\simeq G_p(N).
\end{equation}
First, use the triangle inequality in \eqref{eq:High_Pass_Filter} to bound
\begin{equation}
\lVert P_{\geq N} R_{x_0} \rVert_p
\leq 
\sum_{k \geq k_N} \lVert P_k R_{x_0} \rVert_p
= \sum_{k \geq k_N} \Big\lVert \sum_{n \in \mathbb Z} \psi_k(n) \, \frac{e^{2\pi i (n^2 t + nx_0)}}{n^2} \Big\rVert_p.
\end{equation}
Since $\psi_k$ is supported on $[2^{k-1}, 2^{k+1}]$, 
we can take the denominator $n^2$ out of the $L^p$ norm to get
\begin{equation}
\lVert P_{\geq N} R_{x_0} \rVert_p
\lesssim 
\sum_{k \geq k_N} \frac{1}{2^{2k}} \, \Big\lVert \sum_{n \in \mathbb Z} \psi_k(n) \, e^{2\pi i (n^2 t + nx_0)} \Big\rVert_p,
\end{equation}
for example using \cite[Lemma 3.1, Corollary 3.2]{EceizabarrenaVilacaDaRocha2022}. 
We can now use \eqref{eq:Lp_Rationals_G} 
to get\footnote{The estimates in Theorem~\ref{thm:Lp_Norms_Upper_And_Rationals} depend on $\lVert \Psi \rVert_1$, so strictly speaking we need to check that for large enough $k \gg 1$, the norm $\lVert \psi_k(2^k \cdot) \rVert_1$ does not depend on $k$. This is the case, since 
\begin{equation}
\int \psi_k(2^k x) \, dx = \int_{1/2}^2 \frac{\psi(x)}{\phi(2^k x) + \sum_{i=0}^\infty \psi(2^k x /2^i)} \, dx
= \int_{1/2}^2 \frac{\psi(x)}{\psi(x/2) + \psi(x) + \psi(2x) }\, dx 
= C_\psi. 
\end{equation}  
}
\begin{equation}\label{eq:Proof_UpperBound}
\lVert P_{\geq N} R_{x_0} \rVert_p
\lesssim 
\sum_{k \geq k_N} \frac{G_p(2^k)^{1/p}}{2^{2k}} 
\simeq \frac{G_p(2^{k_N})^{1/p}}{2^{2k_N}},
\end{equation} 
where the last equality follows by direct calculation because 
the defintion of $G_p$ makes the series be geometric. 
For the lower bound, as long as $ 1 < p < \infty$, 
the Mihklin multiplier theorem\footnote{Apply Mihklin's theorem in $\mathbb R$ to the operator $P_{k_N}$ in \eqref{eq:High_Pass_Filter} to get $\lVert P_{k_N} f \rVert_p \simeq \lVert P_{k_N} P_{\geq N} f \rVert_p \lesssim \lVert P_{\geq N} f \rVert_p$, and then periodize the result using a theorem by Stein and Weiss \cite [Chapter 7, Theorem 3.8]{SteinWeiss1971}.  }
combined again with 
\cite[Lemma 3.1, Corollary 3.2]{EceizabarrenaVilacaDaRocha2022}
and \eqref{eq:Lp_Rationals_G} gives
\begin{equation}\label{eq:Proof_LowerBound}
\lVert P_{\geq N} R_{x_0} \rVert_p 
\gtrsim \lVert P_{k_N}R_{x_0} \rVert_p
\simeq \frac{1}{2^{2k_N}}  \,   \Big\lVert \sum_{n} \psi_{k_N}(n) \, e^{2\pi i (n^2 t + nx_0)} \Big\rVert_p
\simeq \frac{G_p(2^{k_N})^{1/p}}{2^{2k_N}}. 
\end{equation}
Joining \eqref{eq:Proof_UpperBound} and \eqref{eq:Proof_LowerBound}
and recalling that $2^{k_N} \simeq \sqrt{N}$, 
we conclude that
\begin{equation}
\lVert P_{\geq N} R_{x_0} \rVert_p  
\simeq \frac{G_p(2^{k_N})^{1/p}}{2^{2k_N}}
\simeq \left\{
\begin{array}{ll}
N^{-1/2 - 1/p}, & p > 4, \\
N^{-3/4} \, (\log N)^{1/4}, & p = 4, \\
N^{-3/4}, & p < 4,
\end{array}\right.
\end{equation}
from which we immediately get
\begin{equation}
\eta(p) 
= \lim_{N \to \infty} \frac{\log (\lVert  P_{\geq N} R_{x_0} \rVert_p^p ) }{\log (1/N)}
= \left\{
\begin{array}{ll}
p/2 + 1, & p > 4, \\
3p/4, & p \leq 4.
\end{array}
\right.
\end{equation}
\end{proof}

\subsection{Frequency localized $L^p$ norms 
}\label{sec:Proof_Of_Lp_Norms_Localized_Rationals}
In this section we prove Theorem~\ref{thm:Lp_Norms_Upper_And_Rationals}. 
The $L^2$ estimate, which holds for all $x_0$,  
follows from Plancherel's theorem. 
For $p \neq 2$, we use the following well-known lemma, 
whose proof can be found in \cite[Lemma 3.18]{Bourgain1993} (see also \cite[Lemma 4.4]{BanicaVega2022}).
\begin{lem}\label{thm:BourgainLemma}
Let $\Psi \in C^\infty_0(\mathbb R)$. 
Let $N \in \mathbb N$ and $q \in \mathbb N$ such that $q \leq N$. 
Let also $a\in \mathbb Z$ such that $(a,q) = 1$. 
Then, 
\begin{equation}\label{eq:BourgainLemma_UpperBound}
\Big| t - \frac{a}{q} \Big| \leq \frac{1}{qN} \quad \Longrightarrow \quad \Big|  \sum_{n \in \mathbb Z} \Psi\left( \frac{n}{N} \right) \, e^{2 \pi i (n^2 t + n x)} \, \Big|  \lesssim_{\lVert \Psi \rVert_1} \frac{N}{\sqrt{q} \, \left( 1 +  N\, \sqrt{ |t -  a/q|  } \right)}.
\end{equation}
Moreover, there exist $\delta, \epsilon \leq 1$ only depending on $\Psi$ such that 
if 
\begin{equation}
q \leq \epsilon N, \qquad \Big| t - \frac{a}{q} \Big| \leq \frac{\delta}{N^2}, \qquad \Big| x - \frac{b}{q} \Big| \leq \frac{\delta}{N}
\end{equation}
for some $b \in \mathbb Z$, then
\begin{equation}\label{eq:GaussSumEstimate}
\Big|  \sum_{n \in \mathbb Z} \Psi\left( \frac{n}{N} \right) \, e^{2 \pi i (n^2 t + n x)} \, \Big| \simeq_{\lVert \Psi \rVert_1} \frac{N}{\sqrt{q}}.
\end{equation}
\end{lem}

We are now ready to prove Theorem~\ref{thm:Lp_Norms_Upper_And_Rationals}. 
\begin{proof}[Proof of Theorem~\ref{thm:Lp_Norms_Upper_And_Rationals}]
Let $x_0 \in \mathbb R$. 
For simplicity, we prove the $L^2$ estimate for a symmetric $\Psi$. 
Considering $f$ as a Fourier series in $t$,  
by Plancherel's theorem we write
	\begin{equation}
	\begin{split}
	\Big\lVert  \sum_{n \in \mathbb Z} \Psi\big(\frac{n}{N}\big)\,  e^{2\pi i (n^2 \, t + n \, x_0)}  \Big\rVert_{L^2(0,1)}^2 
	& = \sum_{n =1}^\infty \left|  \Psi\big(\frac{n}{N}\big)\,e^{2\pi i  n \, x_0} + \Psi\big(-\frac{n}{N}\big)\,e^{-2\pi i  n \, x_0}  \right|^2 \\
	& = \sum_{n=1}^\infty \Psi\big(\frac{n}{N}\big)^2 \left|  e^{2\pi i n x_0} + e^{-2\pi i n x_0} \right|^2 
	 \simeq  \sum_{n=1}^\infty \Psi\big(\frac{n}{N}\big)^2 \cos^2 (2\pi n x_0) 
	\end{split}
	\end{equation}
This sum is upper bounded by $N$ by the triangle inequality. 
	If $x_0$ is rational, say $x_0 = P/Q$, 
	the bound from below follows\footnote{Without loss of generality assume that $\Psi(x) \simeq 1$ for $x \in (1,2)$.}
	by summing only over multiples of $Q$ in $[N,2N]$, so that 
	\begin{equation}\label{eq:LowerSumWithCosine}
	\Big\lVert  \sum_{n \in \mathbb Z} \Psi\big(\frac{n}{N}\big) \,  e^{2\pi i (n^2 \, t + n \, x_0)}  \Big\rVert_{L^2(0,1)}^2 
	 \gtrsim \sum_{ k = N/Q}^{2N/Q} \cos^2 (2\pi kQ x_0)
	 = \frac{N}{Q} \simeq_Q N. 
	\end{equation}		
	If $x_0$ is irrational, it is known that the sequence $(nx_0)_n$ is equidistributed in the torus, 
	which means that for any continuous $p$-periodic function 
	\begin{equation}
	\lim_{N \to \infty } \frac{1}{N} \sum_{n=1}^N f(nx_0)  = \int_0^p f.
	\end{equation}
	In particular, since for $f(y) = \cos(4\pi y)$ we have $\int_0^{1/2} f(y)\, dy=0$, we get\footnote{Using the trigonometric identity $\cos^2(x) = (1 + \cos (2x))/2$.}
	for large $N$ that
	\begin{equation}
	\Big\lVert  \sum_{n \in \mathbb Z} \Psi\big(\frac{n}{N}\big)\,  e^{2\pi i (n^2 \, t + n \, x_0)}  \Big\rVert_{L^2(0,1)}^2
	\gtrsim \sum_{n = N}^{2N} \cos^2 (2\pi n x_0)
	\simeq N +  \sum_{n = N}^{2N} \cos (4\pi n x_0)
	\simeq N.
	\end{equation}
	
	We now prove the upper bound \eqref{eq:LpNorm_Upper} for any $x_0 \in \mathbb R$. 
The Dirichlet approximation theorem implies that any $t \in \mathbb R \setminus \mathbb Q$ can be approximated as follows:
\begin{equation}
\forall N \in \mathbb N, \quad \exists q \leq N, \quad 1 \leq a \leq q \quad \text{ such that } \quad \Big| t - \frac{a}{q} \Big| \leq \frac{1}{qN},  
\end{equation}
which can be rewritten as $\mathbb R \setminus \mathbb Q \subset \bigcup_{q=1}^N \bigcup_{a=1}^q B \big( \frac{a}{q}, \frac{1}{qN} \big)$ for all $N \in \mathbb N$. 
Therefore, for any $N \in \mathbb N$, 
\begin{equation}\label{eq:DecomposingUpperBound}
 \Big\lVert  \sum_{n \in \mathbb Z} \Psi(n/N)\,  e^{2\pi i (n^2 \, t + n \, x_0)}  \Big\rVert_{L^p(0,1)}^p 
 \leq \sum_{q=1}^N \sum_{a=1}^q \int_{B(  \frac{a}{q}, \frac{1 }{qN} )}   \Big|  \sum_{n \in \mathbb Z}  \Psi(n/N)\,  e^{2\pi i (n^2 \, t + n \, x_0)}  \Big|^p\, dt.
\end{equation}
We split each integral according to the two situations in \eqref{eq:BourgainLemma_UpperBound} 
in Lemma~\ref{thm:BourgainLemma}:
\begin{equation}\label{eq:EstimatingAbove}
\begin{split}
& \int_{|t - \frac{a}{q}| < \frac{1}{N^2} }   \Big|  \sum_{n \in \mathbb Z}  \Psi(n/N) \,  e^{2\pi i (n^2 \, t + n \, x_0)}  \Big|^p\, dt  + \int_{ \frac{1}{N^2}  < |t - \frac{a}{q}| < \frac{1}{qN} }   \Big|  \sum_{n \in \mathbb Z}  \Psi(n/N) \,  e^{2\pi i (n^2 \, t + n \, x_0)}  \Big|^p\, dt \\
& \qquad \qquad \leq \int_{|t - \frac{a}{q}| < \frac{1}{N^2} }  \Big( \frac{N}{\sqrt{q} } \Big)^p\, dt  + \int_{\frac{1}{N^2} <|t - \frac{a}{q}| < \frac{1}{qN} }  \Big( \frac{1}{\sqrt{q} \,  |t - \frac{a}{q}|^{1/2 } }  \Big)^p\, dt \\
& \qquad \qquad \simeq  \frac{N^{p-2}}{q^{p/2}} + \frac{1}{q^{p/2}} \, \int_{\frac{1}{N^2}}^{\frac{1}{qN}}  \frac{1}{h^{p/2}} \, dh.
\end{split}
\end{equation} 
The behavior of that last integral changes depending on $p$ being greater or smaller than 2.  
\begin{itemize}
	\item If $p < 2$, 
	\begin{equation}
	\eqref{eq:EstimatingAbove} \simeq \frac{N^{p-2}}{q^{p/2}} + \frac{1}{q^{p/2}} \left( \left( \frac{1}{qN} \right)^{1-p/2} - \left( \frac{1}{N^2} \right)^{1-p/2} \right)
	\leq \frac{N^{p-2}}{q^{p/2}} +  \frac{1}{q\, N^{1-p/2}},
	\end{equation}
	so 
	\begin{equation}
	\eqref{eq:DecomposingUpperBound} 
	 \leq N^{p-2}\, \sum_{q=1}^N \sum_{a=1}^q \frac{1}{q^{p/2}} + \frac{1}{ N^{1-p/2}}\,  \sum_{q=1}^N \sum_{a=1}^q \frac{1}{q}
	\lesssim N^{p/2}. 
	\end{equation}
	
	\item If $p=2$, 
	\begin{equation}
	\eqref{eq:EstimatingAbove} 
	\simeq  \frac{1}{q} \Big( 1 +  \int_{\frac{1}{N^2}}^{\frac{1}{qN}}  \frac{dh}{h} \Big) \lesssim \frac{1}{q} \left( 1 + \log(N^2) - \log(qN) \right)
	= \frac{1 +  \log (N/q) }{q}, 
	\end{equation}
	hence 
	\begin{equation}
	\eqref{eq:DecomposingUpperBound} 
	\lesssim \sum_{q=1}^N \Big( 1 -  \log ( q/N )  \Big) \simeq N - \int_1^N \log(x/N) \, dx \simeq N \Big( 1  -\int_{\frac{1}{N}}^1 \log(y)\, dy  \Big)
	\simeq N. 
	\end{equation}
	
	\item If $p > 2$, 
	\begin{equation}
	\eqref{eq:EstimatingAbove} 
	\simeq  \frac{N^{p-2}}{q^{p/2}} + \frac{ \left( N^2 \right)^{p/2 - 1} -  (qN )^{p/2 - 1} }{q^{p/2}} 
	\lesssim \frac{N^{p-2}}{q^{p/2}} 
	\quad \Longrightarrow \quad 
	\eqref{eq:DecomposingUpperBound} \lesssim N^{p-2} \, \sum_{q=1}^N \frac{1}{q^{p/2 - 1}}.  
	\end{equation}
	This series converges if and only if $p > 4$, and more precisely, 
	\begin{equation}
	\eqref{eq:DecomposingUpperBound} \lesssim 
	\left\{  \begin{array}{ll}
	N^{p-2}, & p > 4, \\
	N^2 \, \log N, & p = 4, \\
	N^{p-2} \, N^{2 - p/2} = N^{p/2}, &  p < 4. 
	\end{array}\right.
	\end{equation}
	This concludes the proof of \eqref{eq:LpNorm_Upper}. 
\end{itemize}

We now prove the lower bound in \eqref{eq:LpNorm_Rationals} for $x_0 \in \mathbb Q$. 
Let $x_0 = P/Q$ with $(P,Q) = 1$. 
Let $\delta, \epsilon > 0$ as given in Lemma~\ref{thm:BourgainLemma}, 
and let $N \in \mathbb N$ be such that $Q \leq \epsilon N$. 
Bound the $L^p$ norm from below by 
\begin{equation}\label{eq:LowerBoundSimple}
 \Big\lVert  \sum_{n \in \mathbb Z} \Psi(n/N)\,  e^{2\pi i (n^2 \, t + n \, x_0)}  \Big\rVert_{L^p(0,1)}^p 
 \geq  \int_{B\big(  \frac{a}{Q}, \frac{\delta }{N^2} \big)}   \Big|  \sum_{n \in \mathbb Z} \Psi(n/N) \,  e^{2\pi i (n^2 \, t + n \, x_0)}  \Big|^p\, dt,
\end{equation}
where $a$ is any $1 \leq a \leq Q$ such that $(a,Q)=1$. 
Use Lemma~\ref{thm:BourgainLemma} with $q=Q$ and $b = P$, 
for which the condition $0 = |x_0 - P/Q| < \delta / N$ is satisfied trivially, 
and $|t - a/Q| < \delta / N^2$, which is valid on the domain of integration.
Then, for $N \geq Q / \epsilon$, 
\begin{equation}
 \Big\lVert  \sum_{n \in \mathbb Z} \Psi(n/N) \,  e^{2\pi i (n^2 \, t + n \, x_0)}  \Big\rVert_{L^p(0,1)}^p  
 \gtrsim \int_{B\big(  \frac{a}{Q}, \frac{\delta }{N^2} \big)} \Big( \frac{N}{\sqrt{Q}} \Big)^p\, dt 
 \simeq \frac{N^p}{Q^{p/2}} \, \frac{\delta}{N^2} \simeq_Q N^{p-2}. 
\end{equation} 
In view of the upper bound in \eqref{eq:LpNorm_Upper}, this is optimal when $p > 4$. 
When $p \leq 4$, we refine the bound in \eqref{eq:LowerBoundSimple} as follows. 
Define the set
\begin{equation}
\mathcal Q_N = \{  \, q \in \mathbb N \, : \,  Q \mid q \, \text{ and } \, q \leq \epsilon N \,  \}, 
\end{equation}
whose cardinality $ \simeq \epsilon N /Q$ is as large as needed if $N \gg 1$. 
Observe that 
\begin{equation}
B\Big( \frac{a}{q}, \frac{\delta}{N^2} \Big)  \cap  B\Big( \frac{a'}{q'}, \frac{\delta}{N^2} \Big) =  \emptyset , \qquad \forall q,q' \in \mathcal Q_N , \quad (a,q) = 1 = (a',q'), 
\end{equation}
as long as $a/q \neq a'/q'$. 
Indeed, the distance from the centers is $\frac{|aq' - a'q|}{q \, q'} \geq \frac{1}{q \, q'} \geq \frac{1}{\epsilon^2 N^2}$, 
while the radius is $\frac{\delta}{N^2} < \frac{1}{\epsilon^2 N^2}$ (choosing a smaller $\delta > 0$ if needed). 
Hence the balls in the family  $\{  B(a/q, \delta/N^2)  \, : \, q \in \mathcal Q_N, \, (a,q)=1 \, \} $ are pairwise disjoint, 
and we can bound 
\begin{equation}\label{eq:Case_P_Small}
 \Big\lVert  \sum_{n \in \mathbb Z} \Psi(n/N) \,  e^{2\pi i (n^2 \, t + n \, x_0)}  \Big\rVert_{L^p(0,1)}^p  \gtrsim \sum_{q \in \mathcal Q_N} \, \sum_{a : (a,q) = 1 } \int_{B\left(  \frac{a}{q}, \frac{\delta }{N^2} \right)} \Big|  \sum_{n \in \mathbb Z} \Psi(n/N)  \,  e^{2\pi i (n^2 \, t + n \, x_0)}   \Big|^p   \, dt.
\end{equation} 
For each of those integrals we have $q = Qn$ for some $n \in \mathbb N$. 
To use Lemma~\ref{thm:BourgainLemma}
we chose $b = Pn$ so that $0 = |x_0 - b/q| < \delta / N$, 
hence
\begin{equation}\label{eq:Before_Lemma}
 \eqref{eq:Case_P_Small}
 \gtrsim \sum_{q \in \mathcal Q_N} \, \sum_{a : (a,q) = 1 } \int_{B\left(  \frac{a}{q}, \frac{\delta }{N^2} \right)} \Big( \frac{N}{\sqrt{q}} \Big)^p  \, dt 
 \simeq  \delta \, N^{p-2}\, \sum_{q \in \mathcal Q_N} \,   \frac{\varphi(q)}{q^{p/2}} 
 \simeq \frac{N^{p-2}}{Q^{p/2}} \, \sum_{n=1}^{\epsilon N / Q} \,   \frac{\varphi(Qn)}{n^{p/2}}.
\end{equation} 
We estimate this sum in the following lemma, which we prove in Appendix~\ref{sec:SumsEulerFunctions}, Corollary~\ref{thm:Appendix_Corollary}.
\begin{lem}\label{thm:SumOfEulerQWithWeight}
Let $Q \in \mathbb N$. Then, for $N \gg 1$, 
\begin{equation}
\sum_{n=1}^N \frac{\varphi(Qn)}{n^2} \simeq \log N, \qquad \text{ and } \qquad \sum_{n=1}^N \frac{\varphi(Qn)}{n^\alpha} \simeq N^{2 -  \alpha}, \quad \text{ for } \quad \alpha < 2,
\end{equation}
where the implicit constants depend on $Q$ and $\alpha$. 
\end{lem}
Using this lemma in \eqref{eq:Before_Lemma}, when $p < 4$ we get
\begin{equation}
 \Big\lVert  \sum_{n \in \mathbb Z} \Psi(n/N)\,  e^{2\pi i (n^2 \, t + n \, x_0)}  \Big\rVert_{L^p(0,1)}^p
 \simeq_{p,Q}  \frac{N^{p-2}}{Q^{p/2}} \, \Big( \frac{\epsilon N}{Q} \Big)^{2-\frac{p}{2}} 
 \simeq_{p,Q} N^{p/2}. 
\end{equation}
Similarly, when $p = 4$ we get
\begin{equation}
 \Big\lVert  \sum_{n \in \mathbb Z} \Psi(n/N)\,  e^{2\pi i (n^2 \, t + n \, x_0)}  \Big\rVert_{L^4(0,1)}^4 
\simeq_{Q}  \frac{N^2}{Q^2} \, \log \Big( \frac{\epsilon N}{Q} \Big)
 \simeq_{Q} N^2 \, \log N.
\end{equation}
Together with the upper bounds in \eqref{eq:LpNorm_Upper}, this completes the proof. 
\end{proof}

\appendix

\section{Sums of Euler's totient function}\label{sec:SumsEulerFunctions}
Sums of the Euler totient function play a relevant role in this article, 
especially in Lemma~\ref{thm:SumOfEulerQWithWeight}.
In Section~\ref{sec:SumsEulerClassical} we state the classical results
and briefly prove them for completeness.  
In Section~\ref{sec:SumsEulerModQ}
we adapt these classical proofs to sums modulo $Q$ that we need in this article.  
Throughout this appendix, $\varphi$ denotes the Euler totient function and $\mu$ denotes the M\"obius function\footnote{For $n \in \mathbb N$, $\mu(n) = 1$ if $n$ is has no squared prime factor and if it has an even number of prime factors; $\mu(n) = -1$ if $n$ is has no squared prime factor and if it has an odd number of prime factors; and $\mu(n) = 0$ if it has a squared prime factor.}.

\subsection{Sums of Euler's totient function}\label{sec:SumsEulerClassical}
Define the sum function
\begin{equation}\label{eq:EulerSumFunction}
\Phi(N) = \sum_{n=1}^N \varphi(n), \qquad N \in \mathbb N. 
\end{equation}

\begin{prop}\label{thm:SumOfEulerFunction}
For $N \gg 1$,  
\begin{equation}
\Phi(N) 
= C N^2 + O\Big( N \log N \Big), 
\qquad \text{ where } \qquad C = \frac12 \, \sum_{n=1}^\infty \frac{\mu(n)}{n^2} = \frac{3}{\pi^2}
\end{equation}
\end{prop}
\begin{proof}
By the M\"obius inversion formula, 
\begin{equation}
\Phi(N) 
= \sum_{n = 1}^N \varphi(n) 
= \sum_{n=1}^N n \bigg( \sum_{d \mid n} \frac{\mu(d)}{d} \bigg) 
= \sum_{n=1}^N \sum_{d \mid n} \frac{n}{d}\, \mu(d).
\end{equation}
Calling $n/d = d'$, the sum is in all natural numbers $d$ and $d'$ such that $d d'  \leq N$. 
Therefore, 
\begin{equation}
\Phi(N) = \sum_{d, d' \, : \, dd' \leq N} d' \mu(d) = \sum_{d=1}^N \mu(d) \, \sum_{d'=1}^{\lfloor N/d \rfloor } d'  
= \sum_{d=1}^N \mu(d) \, \frac{ \lfloor N/d \rfloor \left( \lfloor N/d \rfloor + 1 \right) }{2}.
\end{equation}
For $x \in \mathbb R$, write $x = \lfloor x \rfloor + \{ x \}$, where $0 \leq \{ x \} < 1$ is the fractional part of $x$. Then, direct computation shows that $\lfloor x \rfloor \left( \lfloor x \rfloor + 1 \right) = x^2 + O(x)$ when $x \geq 1$, so 
\begin{equation}
\Phi(N) = \frac{1}{2} \, \sum_{d=1}^N \mu(d) \, \bigg( \Big( \frac{N}{d} \Big)^2 + O\Big(\frac{N}{d} \Big) \bigg) = \frac{N^2}{2} \sum_{d=1}^N \frac{\mu(d)}{d^2} + O\left( N \, \sum_{d=1}^N \frac{1}{d} \right).
\end{equation}
The series $\sum_{d=1}^\infty \mu(d)/d^2$ is absolutely convergent,
and its value is known to be $2C = 6/\pi^2$, 
so write 
\begin{equation}
\sum_{d=1}^N \frac{\mu(d)}{d^2} 
= 2C - \sum_{d=N+1}^\infty \frac{\mu(d)}{d^2} 
= 2C + O\bigg(  \sum_{d=N+1}^\infty \frac{1}{d^2} \bigg) 
= 2C + O\Big(  \frac{1}{N} \Big).
\end{equation}
Since 
$\sum_{d=1}^N 1/d \simeq \log N$, 
we get $\Phi(N) = C\, N^2 + O( N ) + O( N \log N) = CN^2 + O( N \log N)$. 
\end{proof}

As a Corollary of Lemma~\ref{thm:SumOfEulerFunction}
we obtain the analogue result for the sums
weighted by $n^{-\alpha}$. 
Observe that when $\alpha > 2$ the sum is convergent. 
\begin{cor}\label{thm:SumOfEulerWithWeight}
Let $\alpha \leq 2$. For $N \gg 1$, 
\begin{equation}
\sum_{n=1}^N \frac{\varphi(n)}{n^2} \simeq \log N, \qquad \text{ and } \qquad \sum_{n=1}^N \frac{\varphi(n)}{n^\alpha} \simeq N^{2 -  \alpha}, \quad \text{ if } \, \alpha < 2.
\end{equation}
\end{cor}
\begin{proof}
Upper bounds immediately follow from $\varphi(n) \leq n$. 
For lower bounds, 
assume first that $\alpha \geq 0$. 
From Proposition~\ref{thm:SumOfEulerFunction} 
we directly get 
\begin{equation}
\sum_{n=1}^N \frac{\varphi(n)}{n^\alpha} 
\geq \frac{1}{N^\alpha} \sum_{n=1}^N \varphi(n)
= \frac{1}{N^\alpha} \Phi(N) \simeq N^{2- \alpha},  
\end{equation}
which is optimal when $\alpha < 2$. 
For the case $\alpha = 2$ we use the summation by parts formula\footnote{Let $a_n$ and $b_n$ be two sequences,
and let $B_N = \sum_{n=1}^N b_n$. 
Then,  
$
\sum_{n=1}^N a_n b_n = a_N B_N - \sum_{n=1}^{N-1} B_n (a_{n+1} - a_n).
$
}
to get 
\begin{equation}\label{eq:Sum_By_Parts_Applied}
\sum_{n=1}^N \frac{\varphi(n)}{n^2} 
 = \frac{\Phi(N)}{N^2} - \sum_{n=1}^{N-1} \Phi(n) \Big(\frac{1}{(n+1)^2 } - \frac{1}{n^2} \Big) 
 = \frac{\Phi(N)}{N^2} + \sum_{n=1}^{N-1} \Phi(n) \frac{ 2n+1}{n^2 \, (n+1)^2 }.
\end{equation}
Restrict the sum to 
$\log N \leq n \leq N-1$, 
and combine it with $\Phi(n) \simeq n^2$ for $n \gg 1$ from Proposition~\ref{thm:SumOfEulerFunction}
to get
\begin{equation}
\sum_{n=1}^N \frac{\varphi(n)}{n^2} 
\gtrsim 1 + \sum_{n \geq \log N}^{N-1}  \frac{ 1}{n}
\simeq \log N - \log \log N \simeq \log N, 
\qquad \text{ for } \, N \gg 1. 
\end{equation}
When $\alpha < 0$, restrict the sum to $n \in [N/2, N]$ and use $\Phi(N) = C N^2 + O(N\log N)$ in Proposition~\ref{thm:SumOfEulerFunction} to get
\[
\sum_{n=1}^N \frac{\varphi(n)}{n^\alpha} 
= \sum_{n=1}^N \varphi(n) \, n^{|\alpha|} 
\geq \Big(\frac{N}{2}\Big)^{|\alpha|} \, \sum_{n \geq N/2}^N \varphi(n)
\simeq_{|\alpha|} \frac{\Phi(N) - \Phi(N/2)}{N^\alpha}
\simeq N^{2 - \alpha}. 
\qedhere
\]
\end{proof}

\subsection{Sums of Euler's totient function modulo $\boldsymbol{Q}$}\label{sec:SumsEulerModQ}
For $Q \in \mathbb N$, let
\begin{equation}
\Phi_Q(N) = \sum_{n=1}^N \varphi(Qn) \qquad \text{ when } \, N \gg 1,
\end{equation}
To estimate the behavior when $N \to \infty$
we adapt the proofs of Proposition~\ref{thm:SumOfEulerFunction}
and Corollary~\ref{thm:SumOfEulerWithWeight}. 
\begin{prop}\label{thm:SumOfEulerFunctionQ}
Let $Q \in \mathbb N$. 
Then, $\Phi_Q(N) \leq QN^2$, and there exists a constant $c_Q > 0$ such that 
\begin{equation}
\Phi_Q(N) \geq c_Q N^2 + O_Q( N \log N).  
\end{equation}
Consequently,
 $\Phi_Q(N) \simeq _Q N^2$ when $N \gg 1$.  
\end{prop}
\begin{proof}
The upper bound follows directly from $\varphi(n) < n$ for all $n \in \mathbb N$, 
so it suffices to prove the lower bound. 
For that, 
first restrict the sum to $n \leq N$ such that $(Q,n) = 1$. 
By the multiplicative property of the Euler function, 
we get  
\begin{equation}\label{eq:Phi_Q_Lower_Bound}
\Phi_Q(N) \geq \sum_{\substack{ n=1 \\ (Q,n) = 1  } }^N \varphi(Qn)  =  \varphi(Q) \sum_{\substack{ n=1 \\  (Q,n) = 1  } }^N \varphi(n). 
\end{equation}
The proof now follows the same strategy as in Proposition~\ref{thm:SumOfEulerFunction}. 
Use M\"obius inversion to write
\begin{equation}
\sum_{\substack{ n=1 \\  (Q,n) = 1  } }^N \varphi(n)
= \sum_{\substack{ n=1 \\ (Q,n) = 1  } }^N \Bigg( n \sum_{d \mid n} \frac{\mu(d)}{d} \Bigg)
= \sum_{\substack{ n=1 \\ (Q,n) = 1  } }^N \sum_{d \mid n} \, \frac{n}{d} \,  \mu(d). 
\end{equation}
Observe that if $(Q,n)=1$ and if we decompose $n = d \, d'$, then both $d$ and $d'$ are coprime with $Q$. 
Conversely, if $d$ and $d'$ are coprime with $Q$, then so is $n = d \, d'$. Thus, 
\begin{equation}\label{eq:LowerBoundCoprimes}
\sum_{\substack{ n=1 \\  (Q,n) = 1  } }^N \varphi(n) 
=  \sum_{\substack{ d, d' \, : \, d \, d' \leq N \\  (Q,d) = 1 =  (Q,d')  } } d' \,  \mu(d)
=  \sum_{\substack{ d=1 \\  (Q,d) = 1  } }^N \mu(d)  \Bigg( \sum_{\substack{ d'=1 \\  (Q,d') = 1  } }^{\lfloor N/d \rfloor} d' \Bigg). 
\end{equation}
In the following lemma we give a closed formula  for the inner sum. 
We postpone its proof. 
\begin{lem}\label{thm:LemmaModM1}
Let $Q \in \mathbb N$, $Q \geq 2$. Then, 
\begin{equation}
S_Q = \sum_{\substack{ n = 1 \\ (Q,n) = 1 }}^{Q-1} n = \frac{Q \,  \varphi(Q)}{2},  \qquad \text{ and } \qquad S_{Q,k} = \sum_{\substack{ n = 1 \\ (Q,n) = 1 }}^{kQ - 1} n = \frac{Q \, \varphi(Q)}{2} \, k^2, \quad \forall  k \in \mathbb N.
\end{equation}
\end{lem}
Now, 
for every $d \leq N$, 
find $k_d \in \mathbb N\cup \{ 0 \}$ such that $k_dQ \leq \lfloor N/d \rfloor < (k_d+1)Q$, and write 
\begin{equation}\label{eq:InnerSumCoprimes}
\sum_{\substack{ d'=1 \\  (Q,d') = 1  } }^{\lfloor N/d \rfloor} d'  
= \sum_{\substack{ d'=1 \\  (Q,d') = 1  } }^{k_dQ - 1} d'  + \sum_{\substack{ d' = k_dQ + 1 \\  (Q,d') = 1  } }^{\lfloor N/d \rfloor} d'  = S_{Q,k_d} + O\Big(  (k_d+1)Q^2 \Big)
= \frac{Q \, \varphi(Q)}{2} \, k_d^2 + O\Big( (k_d+1)Q^2 \Big).
\end{equation}
Since the definition of $k_d$ is equivalent to 
$\frac{1}{Q} \, \lfloor  N/d\rfloor - 1  <  k_d \leq \frac{1}{Q} \, \lfloor  N/d\rfloor$,
we deduce that
$k_d = \lfloor \frac{1}{Q} \lfloor N/d \rfloor  \rfloor$.
Consequently, 
since $\lfloor x \rfloor = x + O(1)$ 
and $\lfloor x \rfloor^2 = x^2 + O(x)$, we get
\begin{equation}\label{eq:EstimateK_d_And_KdSquare}
k_d 
= \frac{N}{Qd} + O(1)
\qquad \text{ and } \qquad 
k_d^2 = \frac{N^2}{Q^2d^2} + \frac{1}{Q} \, O\Big( \frac{N}{d} \Big).
\end{equation}
Hence, from \eqref{eq:InnerSumCoprimes} and \eqref{eq:EstimateK_d_And_KdSquare} we get 
\begin{equation}
\sum_{\substack{ d'=1 \\  (Q,d') = 1  } }^{\lfloor N/d \rfloor} d'  
= \frac{ \varphi(Q)}{2Q}\, \frac{N^2}{ d^2} 
+ O \left(   
\varphi(Q)\, \frac{N}{d} + Q\frac{N}{d} + Q^2
\right)
= \frac{ \varphi(Q)}{2Q}\, \frac{N^2}{ d^2} 
+ Q^2 \, O \bigg( \frac{N}{d}\bigg).
\end{equation}
We plug this in \eqref{eq:LowerBoundCoprimes} to get
\begin{equation}
\sum_{\substack{ n=1 \\  (Q,n) = 1  } }^N \varphi(n) 
= \frac{\varphi(Q)}{2Q} N^2 \sum_{\substack{ d=1 \\  (Q,d) = 1  } }^N 
\frac{\mu(d)}{ d^2} 
+   O \Big(  Q^2 N \sum_{\substack{ d=1 \\  (Q,d) = 1  } }^N \frac{\mu(d)}{d} 
\Big). 
\end{equation}
The sum $\sum_{n=1}^\infty \mu(d)/d^2$ is absolutely convergent, and $c_Q := \sum_{ d=1, \,  (Q,d) = 1  }^\infty \mu(d)/d^2 > 0$ because 
\begin{equation}
c_Q = 1 + \sum_{\substack{ d=2 \\  (Q,d) = 1  } }^\infty  \frac{\mu(d)}{d^2} 
\qquad \text{ and } \qquad 
\Bigg| \sum_{\substack{ d=2 \\  (Q,d) = 1  } }^\infty  \frac{\mu(d)}{d^2}  \Bigg| \leq \frac{\pi^2}{6} - 1 < 1.
\end{equation}
Hence, 
\begin{equation}
\sum_{\substack{ d=1 \\  (Q,d) = 1  } }^N  \frac{\mu(d)}{d^2} = c_Q - \sum_{\substack{ d=N+1 \\  (Q,d) = 1  } }^\infty  \frac{\mu(d)}{d^2}
= c_Q + O\Big( \sum_{ d=N+1 }^\infty  \frac{1}{d^2}  \Big) = c_Q + O(1/N). 
\end{equation}
Together with $|\sum_{d=1, \, (Q,d)=1}^N \mu(d)/d| \lesssim \log N$, 
this implies
\begin{equation}
\sum_{\substack{ n=1 \\  (Q,n) = 1  } }^N \varphi(n) 
= c_Q\, \frac{\varphi(Q)}{2Q} N^2 + O\Big( \frac{\varphi(Q)}{Q} N \Big) + O (Q^2 N \log N)
 =  c_Q\, \frac{\varphi(Q)}{2Q} N^2 + O_Q(N \log N).
\end{equation}
Together with \eqref{eq:Phi_Q_Lower_Bound} we conclude $\Phi_Q(N) \geq c_Q\, \frac{\varphi(Q)^2}{2Q} N^2 + O_Q(N \log N)$. 
\end{proof}

\begin{proof}[Proof of Lemma~\ref{thm:LemmaModM1}]
We begin with $k=1$. 
When $Q=2$, we have $S_{2,1} = 1 = 2\, \varphi(2)/2$, 
so we may assume $Q \geq 3$. 
We first observe that $\varphi(Q)$ is even, 
because if $Q$ has an odd prime factor $p$, 
then $\varphi(p) = p-1$, which is even, 
is a factor of $\varphi(Q)$. 
Otherwise, $Q = 2^r$ with $r \geq 2$, 
so $\varphi(Q) = 2^{r-1}$ is even.
Now, the observation that 
$(Q,n) = 1 \, \Longleftrightarrow \, (Q,Q-n)=1$
implies
\begin{equation}
S_{Q,1}
= \sum_{\substack{ n = 1 \\ (Q,n) = 1 }}^{\lfloor Q/2 \rfloor} n +  \sum_{\substack{ n = \lfloor Q/2 \rfloor + 1 \\ (Q,n) = 1 }}^{Q - 1} n 
= \sum_{\substack{ n = 1 \\ (Q,n) = 1 }}^{\lfloor Q/2 \rfloor} \big(n + (Q - n) \big)
= Q \,  \frac{\varphi(Q)}{2}.
\end{equation}
Let now $k \geq 2$, so that
\begin{equation}
 \sum_{\substack{ n = (k-1)Q + 1 \\ (Q,n) = 1 }}^{kQ-1} n
= \sum_{\substack{ n =  1 \\ (Q,n) = 1 }}^{Q-1} \bigg( n + (k-1)Q \bigg)
= S_{Q,1} + (k-1)Q\varphi(Q)
= Q \varphi(Q) \Big( k - \frac12 \Big).
\end{equation}
Consequently, 
\[
S_{Q, k} 
= \sum_{\ell=1}^k \Bigg(  \sum_{\substack{ n = (\ell - 1) Q + 1 \\ (Q,n) = 1 }}^{\ell Q} n \Bigg)
= \sum_{\ell=1}^k  Q \varphi(Q) \Big( \ell -  \frac12 \Big)
= \frac{Q\varphi(Q)}{2} k^2.
\qedhere
\]
\end{proof}

To conclude, 
we prove the estimates for the weighted sums 
that we needed in Lemma~\ref{thm:SumOfEulerQWithWeight}
as a corollary of Proposition~\ref{thm:SumOfEulerFunctionQ}. 
As before, when $\alpha > 2$ the sums are absolutely convergent. 
\begin{cor}[Lemma~\ref{thm:SumOfEulerQWithWeight}]\label{thm:Appendix_Corollary}
Let $Q \in \mathbb N$ and $\alpha \leq 2$. 
For $N \gg 1$,  
\begin{equation}
\sum_{n=1}^N \frac{\varphi(Qn)}{n^2} \simeq \log N, 
\qquad \text{ and } \qquad 
\sum_{n=1}^N \frac{\varphi(Qn)}{n^\alpha} 
\simeq N^{2 -  \alpha} \quad \text{ for } \quad \alpha < 2.
\end{equation}
The implicit constants depend on $Q$, and also on $\alpha$ when $\alpha < 0$.  
\end{cor}
\begin{proof}
Upper bounds follow directly from $\varphi(n) \leq n$.
Lower bounds follow from Proposition~\ref{thm:SumOfEulerFunctionQ} with the same strategy as in the proof of Corollary~\ref{thm:SumOfEulerWithWeight}. 
If $\alpha \geq 0$, by Proposition~\ref{thm:SumOfEulerFunctionQ} we get
\begin{equation}
\sum_{n=1}^N \frac{\varphi(Qn)}{n^\alpha} 
\geq \frac{1}{N^\alpha}\, \Phi_Q(N) 
\simeq_Q N^{2-\alpha}, \qquad \text{ when } N \gg 1. 
\end{equation}
When $\alpha = 2$, combine 
Proposition~\ref{thm:SumOfEulerFunctionQ} with
summing by parts
as in \eqref{eq:Sum_By_Parts_Applied}
to get 
\begin{equation}
\sum_{n=1}^N \frac{\varphi(Qn)}{n^2} 
 = \frac{\Phi_Q(N)}{N^2} + \sum_{n=1}^{N-1} \Phi_Q(n) \frac{ 2n+1}{n^2 \, (n+1)^2 }
 \gtrsim 1 + \sum_{n=\log N}^{N-1}  \frac{ 1}{n }
 \simeq \log N. 
\end{equation} 
When $\alpha < 0$, choosing $\delta > 0$ small enough depending on $Q$,  
Proposition~\ref{thm:SumOfEulerFunctionQ}  implies
\[
\sum_{n=1}^N \frac{\varphi(Qn)}{n^\alpha} 
\geq_\alpha N^{|\alpha|} \sum_{n=\delta N}^N \varphi(Qn)
= N^{|\alpha|} \Big( \Phi_Q(N) - \Phi_Q(\delta N) \Big) 
\simeq_{Q,\alpha} N^{|\alpha|} N^2 
= N^{2-\alpha}. 
\qedhere
\]
\end{proof}

\section{Alternative asymptotic behavior of $R_{x_0}$ around rational $t$.}
\label{sec:Asymptotic_Alternative}

Following Duistermaat \cite{Duistermaat1991}, 
we give an alternative asymptotic behavior of $R_{x_0}$ around rationals that complements Corollary~\ref{thm:CorollaryAsymptotic} and allows us to prove Propositions~\ref{thm:Holder_Lower_Bound_At_Rationals}
and \ref{thm:UpperBound34MainText}.

\begin{prop}\label{thm:Asymptotic_Alternative}
Let $x_0 \in \mathbb R$. Let $p,q \in \mathbb N$ be such that $(p,q)=1$. 
Let $x_q = \operatorname{dist}(x_0, \mathbb Z/q)$.
Let $h \neq 0$ and denote $\operatorname{sign}(h) = \pm$ so that $h = \pm |h|$. 
If $x_q = 0$,
	\begin{equation}\label{eq:Asymptotic_xq_Equal_0}
\begin{split}
&R_{x_0}\Big( \frac{p}{q} + h \Big) - R_{x_0}\Big( \frac{p}{q} \Big) + 2\pi i h   \\
& \qquad   \quad = 2\pi (- 1 \pm i) \frac{\sqrt{|h|}}{\sqrt{q}} \frac{ G(p,m_q,q)}{\sqrt{q}} 
+  2 (1 \pm i) \, q^{3/2} |h|^{3/2} \,\sum_{m \neq 0} \, \frac{ G(p,m_q + m,q)}{\sqrt{q}} \,  \frac{e^{- 2\pi i \frac{m^2}{4 q^2 h}}}{ m^2}
 + O\left( q^{7/2} h^{5/2}  \right),
 \end{split}
\end{equation}
If $x_q \neq 0$,
	\begin{equation}\label{eq:Asymptotic_xq_Not_0}
\begin{split}
& R_{x_0}\Big( \frac{p}{q} + h \Big) - R_{x_0}\Big( \frac{p}{q} \Big) + 2\pi i h  \\
& \qquad \qquad 
= 2 (1 \pm i)  \, q^{3/2} |h|^{3/2} \,\sum_{m \in \mathbb Z} \, \frac{ G(p,m_q + m,q)}{\sqrt{q}} \,  
\frac{e^{- 2\pi i \frac{(m - qx_q)^2}{4 q^2 h}}}{ (m - qx_q)^2} \,  + O\left( q^{7/2} h^{5/2} \sum_{m \in \mathbb Z} \frac{1 }{(m - qx_q)^4}  \right).
\end{split}
\end{equation}
\end{prop}

\begin{proof}
From the definition $R_{x_0}(t) = \sum_{n \neq 0}e^{2\pi i ( n^2 t + n x_0 ) }  /n^2$, we first write
\begin{equation}
\begin{split}
R_{x_0}\Big( \frac{p}{q} + h \Big) - R_{x_0}\Big( \frac{p}{q} \Big) + 2\pi i h
& = 2 \pi i h +   \sum_{n \neq 0} \frac{e^{2\pi i n^2 h} - 1}{n^2}\, e^{2\pi i \frac{pn^2}{q}} \, e^{2\pi i n x_0} \\
& = 2\pi i h \, \sum_{n \in \mathbb Z} \Big( \int_0^1 e^{2\pi i n^2 h \tau} \, d\tau \Big) \, e^{2\pi i \frac{pn^2}{q}} \, e^{2\pi i n x_0}.
\end{split}
\end{equation}
Split the sum modulo $q$ by writing $n = mq+r$ and use the Poisson summation formula to obtain
\begin{equation}
\begin{split}
R_{x_0}\Big( \frac{p}{q} + h \Big) - R_{x_0}\Big( \frac{p}{q} \Big) + 2\pi i h 
& = 2\pi i h \, \sum_{r = 0}^{q-1} e^{2\pi i r^2 p/q} \,  \sum_{m \in \mathbb Z} \Big( \int_0^1 e^{2\pi i (mq+r)^2 h \tau} \, d\tau \Big) \,  e^{2\pi i (mq+r) x_0} \\
& = 2\pi i h \, \sum_{r = 0}^{q-1} e^{2\pi i r^2 p/q} \,  \sum_{m \in \mathbb Z} \int \Big( \int_0^1 e^{\pm 2\pi i (zq+r)^2 |h| \tau} \, d\tau \Big) \,  e^{2\pi i (zq+r) x_0} \, e^{-2\pi i m z} \, dz \\
& = \pm 2\pi i \frac{\sqrt{|h|}}{q} \,\sum_{m \in \mathbb Z} \,  \sum_{r = 0}^{q-1} e^{2\pi i \frac{p r^2 + m r }{q}} \,   \int_0^1 \int  e^{\pm 2\pi i y^2 \tau}  \,  e^{2\pi i \frac{ y}{\sqrt{|h|}} ( x_0 - \frac{m}{q})} \, dy \, d\tau.
\end{split}
\end{equation}
where we changed variables $(zq+r)^2 |h| = y^2$. Now complete the square to get
\begin{equation}\label{eq:After_Heuristic_Computations}
\begin{split}
R_{x_0}\Big( \frac{p}{q} + h \Big) - R_{x_0}\Big( \frac{p}{q} \Big) + 2\pi i h 
& = \pm 2\pi i \frac{\sqrt{|h|}}{q} \,\sum_{m \in \mathbb Z} \, G(p,m,q) \,   \int_0^1 \Big( \int  e^{\pm 2\pi i \tau \Big( y \pm \frac{x_0 - m/q}{2 \tau\sqrt{|h|}} \Big)^2}   \,   dy \Big) \, e^{\mp 2\pi i \frac{(x_0 - m/q)^2}{4\tau |h|} } d\tau  \\
& = \pm 2\pi i \frac{1\pm i}{2} \, \frac{\sqrt{|h|}}{q} \,\sum_{m \in \mathbb Z} \,  G(p,m,q) \,   \int_0^1 \frac{1}{\sqrt{\tau}} \, e^{ \mp 2\pi i \frac{(x_0 - m/q)^2}{4\tau |h|} } d\tau.
\end{split}
\end{equation}
By changing variables, and defining $x_q = \min_{m\in\mathbb Z} |x_0 - m/q| = |x_0 - m_q/q|$ as in \eqref{eq:x_q_AND_m_q}, we write
\begin{equation}\label{eq:Asymptotic_Second}
R_{x_0}\Big( \frac{p}{q} + h \Big) - R_{x_0}\Big( \frac{p}{q} \Big) + 2\pi i h  
= \pi  (-1 \pm i) \, \frac{\sqrt{|h|}}{\sqrt{q}} \,\sum_{m \in \mathbb Z} \, \frac{ G(p,m_q + m,q)}{\sqrt{q}} \,   \int_1^\infty \frac{1}{ \xi^{3/2}} \, e^{- 2\pi i \frac{(x_q - m/q)^2}{4 h} \xi} d\xi.
\end{equation}
We now separate cases.
If $x_q = 0$, the integral of the term $m = 0$ is $\int_1^\infty \xi^{-3/2}d\xi =2$.
In all other cases, that is, if either $x_q \neq 0$ or $m \neq 0$, integration by parts implies
\begin{equation}\label{eq:IntByPartsOnce}
\begin{split}
  \int_1^\infty \frac{1}{ \xi^{3/2}} \, e^{- 2\pi i \frac{(x_q - m/q)^2}{4 h} \xi} d\xi
 & = \frac{2}{\pi i}\frac{ q^2 h}{(m - qx_q)^2} \left( \, e^{- 2\pi i \frac{(m - qx_q)^2}{4 q^2 h}} + \frac{3}{2} \, \int_1^\infty \frac{1}{ \xi^{5/2}} \, e^{- 2\pi i \frac{(x_q - m/q)^2}{4 h} \xi} d\xi \right) \\
 &  = O\left( \frac{q^2 h}{(m-qx_q)^2} \right) .
  \end{split}
\end{equation}
What is more, integrating by parts again we obtain
\begin{equation}\label{eq:IntByPartsTwice}
\begin{split}
\int_1^\infty \frac{1}{ \xi^{3/2}} \, e^{- 2\pi i \frac{(x_q - m/q)^2}{4 h} \xi} d\xi
& = \frac{2}{\pi i} \, \frac{q^2 h}{ (m - qx_q)^2} \, \left(  e^{- 2\pi i \frac{(m - qx_q)^2}{4 q^2 h}} + O\Big( \frac{ q^2 h}{(m - qx_q)^2} \Big) \right).
\end{split}
\end{equation}
Combining these with \eqref{eq:Asymptotic_Second} give the desired expressions. 
\end{proof}
\begin{rem}
Computations for \eqref{eq:After_Heuristic_Computations} are made rigorous to avoid convergence problems by writing
\begin{equation}
\sum_{n \in \mathbb Z} \frac{e^{2\pi i n^2 h} - 1}{n^2}\, e^{2\pi i n^2 p/q} \, e^{2\pi i n x_0} = \lim_{\epsilon \to 0} \sum_{n \in \mathbb Z} \frac{e^{2\pi i n^2 h(1+i\epsilon)} - 1}{n^2}\, e^{2\pi i n^2 p/q} \, e^{2\pi i n x_0}.
\end{equation} 
\end{rem}

Proposition~\ref{thm:Asymptotic_Alternative} 
will allow us to give upper bounds of $\alpha_{x_0}(t)$ for general $t$. 

\begin{prop}\label{thm:UpperBound34}
Let $x_0 \in \mathbb Q$ and $t \not\in \mathbb Q$. Then, $\alpha_{x_0}(t) \leq  3/4$. 
\end{prop}
\begin{proof}
Set $x_0 = P/Q$ with $P,Q \in \mathbb N$ and $(P,Q) = 1$. 
Let $t \not\in \mathbb Q$ and let $p_n/q_n$ be its approximations by continued fractions. 
It is well-known\footnote{Because two consecutive denominators $q_n$ and $q_{n+1}$ are never both even.} that there is a subsequence of odd denominators $q_{n_k}$. 
Renaming that subsequence back to $q_n$, we may assume that all $q_n$ are odd. 
Consequently, $|G(p_n, m, q_n)| = \sqrt{q}$ for all $m, n \in \mathbb N$. 
As usual, let 
\begin{equation}
h_n =  t - \frac{p_n}{q_n}, 
\qquad|h_n| < \frac{1}{q_n^2}, 
\qquad x_{q_n} = \min_{m \in \mathbb Z} \Big| \frac{P}{Q} - \frac{m}{q_n} \Big| = \Big| \frac{P}{Q} - \frac{m_{q_n}}{q_n} \Big|,
\end{equation}
and we immediately deduce that either $x_{q_n} = 0$ or $1/Q \leq q_n x_{q_n} \leq 1/2$. 
We separate cases:
\begin{itemize}
	\item[\textbf{Case 1}] We have $x_{q_n} = 0$ for infinitely many $n \in \mathbb N$. 
	Rename that subsequence and rewrite \eqref{eq:Asymptotic_xq_Equal_0} as 
	\begin{equation}\label{eq:Asymptotic_Case_1}
\Big| R_{x_0} \Big( \frac{p_n}{q_n} + h_n \Big)  - R_{x_0}\Big( \frac{p_n}{q_n} \Big) + 2\pi i h_n   \Big| 
 = 2\pi \sqrt{2} \frac{\sqrt{|h_n|}}{\sqrt{q_n}}  + O\left( q_n^{3/2} h_n^{3/2}  \right) 
 \simeq \frac{\sqrt{|h_n|}}{\sqrt{q_n}} \left(  1 + O\big( q_n^2 h_n \big) \right).
\end{equation}
Let $\delta > 0$ which we determine later. Separate cases again:
\begin{itemize}
	\item[\textbf{Case 1.1.}] Suppose that $\left|  1 + O\big( q_n^2 h_n \big) \right| \geq \delta$ for infinitely many $n \in \mathbb N$. 
	Then, 
	\begin{equation}
	\left| R_{x_0} (t) - R_{x_0}( t - h_n ) + 2\pi i h_n   \right| 
	\geq \delta \frac{\sqrt{|h_n|}}{\sqrt{q_n}}  
	\geq \delta \, |h_n|^{3/4}, 
	\end{equation}
	because $q_n^2 |h_n| \leq 1$. 
	Hence $| R_{x_0} (t) - R_{x_0}( t - h_n )| \geq (\delta/2) \, |h_n|^{3/4}$ for infinitely many $n \in \mathbb N$,
	and consequently $\alpha_{x_0}(t) \leq 3/4$. 
	
	\item[\textbf{Case 1.2.}] We have $\left|  1 + O\big( q_n^2 h_n \big) \right| < \delta$ for all large enough $n$.
	In that case, we evaluate \eqref{eq:Asymptotic_Case_1} at a point closer to $p_n/q_n$.
	Let $\epsilon > 0$ and write \eqref{eq:Asymptotic_xq_Equal_0} for $\epsilon h_n$, so that instead of  \eqref{eq:Asymptotic_Case_1} we get
	\begin{equation}
\Big|   R_{x_0} \Big( \frac{p_n}{q_n} + \epsilon h_n \Big)  - R_{x_0}\Big( \frac{p_n}{q_n} \Big) + 2\pi i \epsilon h_n \Big|   
 \simeq \sqrt{\epsilon} \, \frac{\sqrt{|h_n|}}{\sqrt{q_n}} \left(  1 + \epsilon O\big( q_n^2 h_n \big) \right).
\end{equation}
Since $q_n^2 |h_n| < 1$ and the constant underlying the big-$O$ is universal, say $C$, choose $\epsilon \leq 1/(2C)$, in such a way that 
\begin{equation}
\Big| R_{x_0} \Big( \frac{p_n}{q_n} + \epsilon h_n \Big)  - R_{x_0}\Big( \frac{p_n}{q_n} \Big) + 2\pi i \epsilon h_n  \Big|   
\gtrsim  \frac{\sqrt{\epsilon}}{2} \, \frac{\sqrt{|h_n|}}{\sqrt{q_n}}.
\end{equation}
From this and \eqref{eq:Asymptotic_Case_1}, we write
\begin{equation}
\begin{split}
\frac{\sqrt{\epsilon}}{2} \, \frac{\sqrt{|h_n|}}{\sqrt{q_n}} 
& \lesssim \Big| R_{x_0} \Big( \frac{p_n}{q_n} + \epsilon h_n \Big)  - R_{x_0}\Big( \frac{p_n}{q_n} \Big) \Big| + 2\pi \epsilon |h_n|  \\
& \leq \Big| R_{x_0} \Big( \frac{p_n}{q_n} + \epsilon h_n \Big)  - R_{x_0}( t ) \Big| + \Big| R_{x_0} ( t )  - R_{x_0}\Big( \frac{p_n}{q_n} \Big) \Big| + 2\pi \epsilon |h_n| \\
& \lesssim \Big| R_{x_0} \Big( \frac{p_n}{q_n} + \epsilon h_n \Big)  - R_{x_0}( t ) \Big| + \frac{\sqrt{|h_n|}}{\sqrt{q_n}} \left(  1 + O\big( q_n^2 h_n \big) \right) + 2\pi (1 +  \epsilon) |h_n| \\
& \leq \Big| R_{x_0} \Big( \frac{p_n}{q_n} + \epsilon h_n \Big)  - R_{x_0}( t ) \Big| + 2\delta\, \frac{\sqrt{|h_n|}}{\sqrt{q_n}}.
\end{split}
\end{equation} 
In the last line we used the hypothesis of Case 1.2 and $|h_n| \leq \frac{\sqrt{|h_n|}}{\sqrt{q_n}} \, \frac{1}{\sqrt{q_n}}$.
Hence, 
 \begin{equation}
 \Big| R_{x_0} \Big( \frac{p_n}{q_n} + \epsilon h_n \Big)  - R_{x_0}( t ) \Big| 
 \gtrsim \left(  \frac{\sqrt{\epsilon}}{2}  - C\delta \right)   \frac{\sqrt{|h_n|}}{\sqrt{q_n}},
\end{equation}
for some $C>0$.
Fix $\sqrt{\epsilon} = 4C \delta$ small enough.
Writing $p_n/q_n + \epsilon h_n = t - (1-\epsilon) h_n$ and observing that $(1-\epsilon)|h_n| \simeq |h_n|$, we conclude that 
\begin{equation}
 \Big| R_{x_0} \big( t - (1- \epsilon) h_n \big)  - R_{x_0}( t ) \Big| 
 \gtrsim \delta \,  \frac{\sqrt{|h_n|}}{\sqrt{q_n}}
 \geq \delta \, |h_n|^{3/4} 
 \simeq | (1-\epsilon) h_n |^{3/4}, 
 \quad \text{ for large enough } n. 
\end{equation}
Hence $\alpha_{x_0}(t) \leq 3/4$. 
 
\end{itemize}

	\item[\textbf{Case 2}] We have $x_{q_n} \neq 0$ for all large enough $n \in \mathbb N$,
	hence $1/Q \leq q_n x_{q_n} \leq 1/2$. 
	We now use \eqref{eq:Asymptotic_xq_Not_0} which
	has no leading $h^{1/2}$ term. 
	Rewrite it\footnote{
	When $q$ is odd and coprime with $p$, the inverses of 2 and $p$ modulo $q$ exist.
	Therefore, 
	\begin{equation}
	G(p,m,q) = \sum_{r=1}^q e^{2\pi i \frac{pr^2 + mr}{q}}
	= e^{2\pi i (4p)^{-1}\frac{m^2}{q}} \, \sum_{r=1}^q e^{2\pi i p \, \frac{(r + (2p)^{-1}m)^2}{q} } 
	= e^{2\pi i (4p)^{-1}\frac{m^2}{q}} \, G(p,0,q).
	\end{equation}	
	}, 
	assuming $1/Q \leq qx_q \leq 1/2$, 
	as	
	\begin{equation}
	\begin{split}
& R_{x_0}\Big( \frac{p}{q} + h \Big) - R_{x_0}\Big( \frac{p}{q} \Big) + 2\pi i h  \\
& \qquad \quad = 2(1 \pm i) \, \frac{G(p,0,q)}{\sqrt{q}}\, \frac{\sqrt{|h|}}{\sqrt{q}} \,  q^{2} |h| \, \left[  \sum_{m \in \mathbb Z} \, e^{2\pi i (4p)^{-1} \frac{(m_q + m)^2}{q}} \,  
\frac{e^{- 2\pi i \frac{(m - qx_q)^2}{4 q^2 h}}}{ (m - qx_q)^2} \,  + O_Q\left( q^{2} h  \right)
\right].
\end{split}
	\end{equation}
	Define the auxiliary function
	\begin{equation}\label{eq:f_q}
	f_q(y) = \sum_{m \in \mathbb Z} \, e^{2\pi i (4p)^{-1} \frac{ m^2 + 2m_q m}{q}} \,  
\frac{e^{- 2\pi i (m^2  - 2 qx_q m)  y }}{ (m - qx_q)^2}.
	\end{equation}
	Take absolute values and write
	\begin{equation}\label{eq:Asymptotic_f_q}
	\Big|  R_{x_0}\Big( \frac{p}{q} + h \Big) - R_{x_0}\Big( \frac{p}{q} \Big) + 2\pi i h  \Big|
	= 2\sqrt{2}  \frac{\sqrt{|h|}}{\sqrt{q}} \,  q^{2} |h| \,
\left|  f_q\Big( \frac{1}{4q^2 h}  \Big)  + O_Q\left( q^{2} h  \right)
\right|.
	\end{equation}
We now state the properties of this function, whose proof we postpone. 
\begin{lem}\label{thm:Lemma_f_q}
Let $q \in \mathbb N$, let $p \in \mathbb N$ be coprime with $q$ and $f_q$ defined in \eqref{eq:f_q}. 
Then, 
\begin{enumerate}
	\item $f_q$ is periodic of period $Q$.
	\item there exists $y_0^q \in [0,Q]$ depending on $q$ (and on $p$) such that $|f_q(y_0^q) | \geq 5$.
	\item The sequence defined by $y_k^q = y_0^q + kQ$ satisfies 
	\begin{equation}
	\lim_{k \to \infty} y_k^q = \infty, \qquad \text{ and } \qquad  |f_q(y_k^q)| \geq 5, \quad \forall k \in \mathbb N. 
	\end{equation}
\end{enumerate}
\end{lem}
\begin{rem}
The dependence on $p$ of the point $y_0^q$ is irrelevant for our purposes.
Indeed, once we fix $t \not\in \mathbb Q$, we get the sequence of approximations $p_n/q_n$, 
hence each $q_n$ comes with one and only one $p_n$. 
Hence, we can assume that the sequence $f_{q_n}$ only depends on $q_n$. 
\end{rem}
We now evaluate \eqref{eq:Asymptotic_f_q} at $p_n/q_n$ and $h_n = t - p_n/q_n$ and we separate two cases:
\begin{itemize}
	\item[\textbf{Case 2.1.}] Suppose $\limsup_{n \to \infty} q_n^2 |h_n| > 0$, so that there exists $c > 0$ and a subsequence for which $c < q_n^2 |h_n| \leq 1$.
	Then, from \eqref{eq:Asymptotic_f_q} we get
		\begin{equation}
	\Big|  R_{x_0}( t ) - R_{x_0}\Big( \frac{p_n}{q_n} \Big) + 2\pi i h_n  \Big|
	\geq c  \frac{\sqrt{|h_n|}}{\sqrt{q_n}} \,
\left|  f_{q_n}\Big( \frac{1}{4q_n^2 h_n}  \Big)  + O_Q\left( q_n^2 h_n  \right)
\right|.
	\end{equation}
	Fix $\delta>0$ which we later determine. Proceeding like in \textbf{Case 1}, we separate two cases:
	\begin{itemize}
		\item[\textbf{Case 2.1.1.}] Suppose $\left|  f_{q_n}\Big( \frac{1}{4q_n^2 h_n}  \Big)  + O_Q\left( q_n^2 h_n  \right) \right| \geq \delta$ for infinitely many $n$. Then, 
		\begin{equation}
	\Big|  R_{x_0}( t ) - R_{x_0}\Big( \frac{p_n}{q_n} \Big) + 2\pi i h_n  \Big|
	\geq c \delta \frac{\sqrt{|h_n|}}{\sqrt{q_n}} 
	\geq c\delta |h_n|^{3/4} 
	\end{equation}
		for infinitely many $n$, which implies $\alpha_{x_0}(t) \leq 3/4$.
				
		\item[\textbf{Case 2.1.2.}] Suppose $\left|  f_{q_n}\Big( \frac{1}{4q_n^2 h_n}  \Big)  + O_Q\left( q_n^2 h_n  \right) \right| < \delta$ for all large enough $n$. 
		Then, let $\epsilon_n$ be a sequence which we determine later,
		and define $\eta_n = \epsilon_n /q_n^2$. 
		Observe that $\eta_n = \epsilon_n |h_n| / (q_n^2 |h_n|) \simeq \epsilon_n |h_n|$. 
		Evaluate \eqref{eq:Asymptotic_f_q} at $\eta_n$ to get 
		\begin{equation}
		\begin{split}
		\Big|  R_{x_0}\Big( \frac{p_n}{q_n} + \eta_n \Big) - R_{x_0}\Big( \frac{p_n}{q_n} \Big) + 2\pi i \eta_n  \Big|
	& = 2\sqrt{2}  \frac{\sqrt{\eta_n}}{\sqrt{q_n}} \,  q_n^2 \eta_n \,
\left|  f_q\Big( \frac{1}{4q_n^2 \eta_n}  \Big)  + O_Q\left( q_n^{2} \eta_n  \right)
\right| \\
& = 2\sqrt{2} \,   \epsilon_n \,  \frac{\sqrt{\eta_n}}{\sqrt{q_n}} \,  
\left|  f_{q_n}\Big( \frac{1}{4 \epsilon_n}  \Big)  + O_Q\left( \epsilon_n  \right)
\right|. 
\end{split}
		\end{equation}
		Fix $k \in \mathbb N$ large enough and set $\epsilon_n = 1/(4y_k^{q_n})$.
		Then, by Lemma~\ref{thm:Lemma_f_q} (c), 
		\begin{equation}
		\Big|  f_{q_n}\Big( \frac{1}{4 \epsilon_n}  \Big) \Big| = \left|  f_{q_n}( y_K^{q_n}) \right| \geq 5,  \quad \forall n \text{ large enough.}
		\end{equation}
		Since $\epsilon_n \simeq 1/(kQ)$, 
		if $k \in \mathbb N$ is large enough we get $ O_Q\left( \epsilon_n  \right) \leq C_Q \epsilon_n \leq 1$.
		In particular, $|h_n| \simeq_Q k \eta_n$. 
		Therefore, 
		\begin{equation}
		\Big|  R_{x_0}\Big( \frac{p_n}{q_n} + \eta_n \Big) - R_{x_0}\Big( \frac{p_n}{q_n} \Big) + 2\pi i \eta_n  \Big|
 \geq  \epsilon_n\,  \frac{\sqrt{\eta_n}}{\sqrt{q_n}} 
 \simeq \epsilon_n^{3/2} \frac{\sqrt{|h_n|}}{\sqrt{q_n}}. 
		\end{equation}
		With this, and using the assumption of this case in \eqref{eq:Asymptotic_f_q}, we write
		\begin{equation}
		\begin{split}
		\epsilon_n^{3/2} \frac{\sqrt{|h_n|}}{\sqrt{q_n}}
		&  \lesssim \Big|  R_{x_0}\Big( \frac{p_n}{q_n} + \eta_n \Big) - R_{x_0}\Big( \frac{p_n}{q_n} \Big)\Big| + 2\pi \eta_n   \\
		& \leq \Big|  R_{x_0}\Big( \frac{p_n}{q_n} + \eta_n \Big) - R_{x_0}(t) \Big| 
		+ \Big| R_{x_0}(t) - R_{x_0}\Big( \frac{p_n}{q_n} \Big)\Big| 
		+ 2\pi \eta_n \\
		& \lesssim \Big|  R_{x_0}\Big( \frac{p_n}{q_n} + \eta_n \Big) - R_{x_0}(t) \Big| 
		+ \delta \frac{\sqrt{|h_n|}}{\sqrt{q_n}} 
		+ 2\pi ( |h_n|  + \eta_n)   \\
		& \lesssim \Big|  R_{x_0}\Big( \frac{p_n}{q_n} + \eta_n \Big) - R_{x_0}(t) \Big| 
		+ \delta \frac{\sqrt{|h_n|}}{\sqrt{q_n}},
		\end{split}
		\end{equation}
	for large enough $n$, 
	where in the last line we used $\eta_n \simeq |h_n|/k \leq |h_n|$
	and $|h_n| \leq  \frac{\sqrt{|h_n|}}{\sqrt{q_n}} \, \frac{1}{\sqrt{q_n}} \ll  \frac{\sqrt{|h_n|}}{\sqrt{q_n}}$. 
	Since $\epsilon_n \simeq_Q 1/k$, set  $\delta = 1/(c_Qk^{3/2})$ with some small enough $c_Q>0$ so that
	\begin{equation}
	\Big|  R_{x_0}\Big( \frac{p_n}{q_n} + \eta_n \Big) - R_{x_0}(t) \Big| 
	\gtrsim \Big(  \epsilon_n^{3/2} - C\delta \Big) \frac{\sqrt{|h_n|}}{\sqrt{q_n}} 
	\gtrsim \delta \frac{\sqrt{|h_n|}}{\sqrt{q_n}} 
	\geq \delta  |h_n|^{3/4}. 
	\end{equation}
	Write $p_n/q_n + \eta_n = t - (h_n -  \eta_n)$. 
	Since $|h_n - \eta_n| \leq 2|h_n|$, 
	we get 
	\begin{equation}
	\Big|  R_{x_0}\Big(  t - (h_n -  \eta_n) \Big) - R_{x_0}(t) \Big| 
	\geq \delta |h_n|^{3/4} 
	\gtrsim \delta |h_n -  \eta_n|^{3/4}, \quad \text{ for large enough } n, 
	\end{equation}
	which implies $\alpha_{x_0}(t) \leq 3/4$. 
		
	\end{itemize}

	\item[\textbf{Case 2.2.}] Suppose $\lim_{n \to \infty} q_n^2 |h_n| = 0$. 
	In this case, the term $q_n^2 |h_n|$ in \eqref{eq:Asymptotic_f_q} tends to zero, which kills the desired $|h_n|^{3/4}$ that came from $\sqrt{h_n}/\sqrt{q_n}$. 
	To counter that, define $\eta_n = \epsilon_n / q_n^2$ 
	as in Case 2.1.2. By \eqref{eq:Asymptotic_f_q}, 
	\begin{equation}
	\Big|  R_{x_0}\Big( \frac{p_n}{q_n} + \eta_n \Big) - R_{x_0}\Big( \frac{p_n}{q_n} \Big) + 2\pi i \eta_n  \Big|
	= 2\sqrt{2} \epsilon_n \frac{\sqrt{\eta_n}}{\sqrt{q_n}} \,
\left|  f_{q_n}\Big( \frac{1}{4\epsilon_n}  \Big)  + O_Q\left(\epsilon_n  \right)
\right|.
	\end{equation}
	Fix $k \in \mathbb N$ large enough and set $\epsilon_n = 1/(4y_k^{q_n})$. Then,
	\begin{equation}
	\left|  f_{q_n}\Big( \frac{1}{4\epsilon_n}  \Big)  \right| 
	= \left|  f_{q_n}\Big( y_k^{q_n} \Big)  \right| \geq 5, 
	\quad \text{ and } \quad
	O_Q(\epsilon_n) \leq C_Q \epsilon_n = \frac{C_Q}{4y_k^{q_n}} \simeq \frac{C_Q}{kQ} \leq 1,
	\end{equation}
	so
	\begin{equation}
	\Big|  R_{x_0}\Big( \frac{p_n}{q_n} + \eta_n \Big) - R_{x_0}\Big( \frac{p_n}{q_n} \Big) + 2\pi i \eta_n  \Big|
	\geq  \epsilon_n \frac{\sqrt{\eta_n}}{\sqrt{q_n}}. 
	\end{equation}
	With this and \eqref{eq:Asymptotic_f_q}, we can write
	\begin{equation}
	\begin{split}
	\epsilon_n \frac{\sqrt{\eta_n}}{\sqrt{q_n}} 
	&  \leq \Big|  R_{x_0}\Big( \frac{p_n}{q_n} + \eta_n \Big) - R_{x_0}\Big( \frac{p_n}{q_n} \Big) \Big| + 2\pi \eta_n  \\
	 &  \leq \Big|  R_{x_0}\Big( \frac{p_n}{q_n} + \eta_n \Big) -R_{x_0}(t) \Big| + \Big| R_{x_0}(t) -  R_{x_0}\Big( \frac{p_n}{q_n} \Big) \Big| + 2\pi \eta_n  \\ 
	 &  \lesssim \Big|  R_{x_0}\Big( \frac{p_n}{q_n} + \eta_n \Big) -R_{x_0}(t) \Big| + \frac{\sqrt{|h_n|}}{\sqrt{q_n}} \, q_n^2 |h_n| + 2\pi (\eta_n + |h_n|).
	 \end{split}
	\end{equation}
	Since $\lim_{n \to \infty} q_n^2 |h_n|=0$ implies $h_n = o(\eta_n)$, and $\eta_n =  \frac{\sqrt{\eta_n}}{\sqrt{q_n}} \frac{\sqrt{\epsilon_n}}{\sqrt{q_n}}$, 
	we get
	\begin{equation}
	\Big|  R_{x_0}\Big( \frac{p_n}{q_n} + \eta_n \Big) -R_{x_0}(t) \Big|
	\gtrsim \Big( \epsilon_n - q_n^2 h_n - \frac{\sqrt{\epsilon_n}}{ \sqrt{q_n} } \Big) \frac{\sqrt{\eta_n}}{\sqrt{q_n}}
	\geq \frac{\epsilon_n}{2} \, \frac{\sqrt{\eta_n}}{\sqrt{q_n}}
	=  \frac{\epsilon_n^{3/4}}{2} \, \eta_n^{3/4}.
	\end{equation}
	Write $p_n/q_n + \eta_n = t + (\eta_n - h_n)$. 
	Recalling $\epsilon_n \simeq 1/(kQ)$ for all $n$, and since $h_n = o(\eta_n)$ implies $|\eta_n - h_n| \simeq \eta_n$,  we conclude
	\begin{equation}
	\Big|  R_{x_0}\Big( t + (\eta_n - h_n) \Big) -R_{x_0}(t) \Big|
	\geq  \frac{\epsilon_n^{3/4}}{2} \, \eta_n^{3/4}
	\simeq_Q |\eta_n - h_n|^{3/4}, 
	\end{equation}
	and therefore $\alpha_{x_0}(t) \leq 3/4$. 
\end{itemize}

\end{itemize}
\end{proof}

We now prove Lemma~\ref{thm:Lemma_f_q}. 
\begin{proof}[Proof of Lemma~\ref{thm:Lemma_f_q}]
$(a)$ Write first 
\begin{equation}
0 \neq qx_q 
= q \min_{m \in \mathbb Z} \Big|x_0 - \frac{m}{q} \Big|
= q \,  \Big|x_0 - \frac{m_q}{q} \Big|
= \frac{1}{Q} \,  | Pq - Q m_q | 
= \frac{m_q'}{Q},
\end{equation}
where $m_q' = |Pq - Qm_q| \in \mathbb N \setminus \{0\}$. 
Hence, the variable $y$ in \eqref{eq:f_q} only appears in  
\begin{equation}
e^{2\pi i (m^2 - 2qx_q m) \, y } = e^{2\pi i (Qm^2 - 2 m_q' m) \frac{y}{Q} }, 
\end{equation}
which is $Q$-periodic. Hence $f_q$ has period $Q$. 

$(b)$ Split the sum in $f_p$ in the terms $m=0,1$ and the rest,  
\begin{equation}
f_q(y) 
=   \frac{1}{ (qx_q)^2} 
+ e^{2\pi i (4p)^{-1} \frac{ 1 + 2m_q }{q}} \,  \frac{e^{- 2\pi i (1  - 2 qx_q)  y }}{ (1 - qx_q)^2}
+ \text{Error}
\end{equation}
where $1/Q \leq qx_q \leq 1/2$ implies
\begin{equation}
\begin{split}
|\text{Error}| 
& = \Big| \sum_{m \neq 0,1} \, e^{2\pi i (4p)^{-1} \frac{ m^2 + 2m_q m}{q}} \,  
\frac{e^{- 2\pi i (m^2  - 2 qx_q m)  y }}{ (m - qx_q)^2} \Big|  \\
& \leq  \sum_{m=2}^\infty \frac{1}{ (m - qx_q)^2} +  \sum_{m=1}^\infty \frac{1}{ (m + qx_q)^2} 
\leq \sum_{m=2}^\infty \frac{1}{ (m - 1/2)^2} +  \sum_{m=1}^\infty \frac{1}{ m^2} 
= \frac{\pi^2}{2} - 4 + \frac{\pi^2}{6} \leq 3. 
\end{split}
\end{equation}
On the other hand, the phase in 
\begin{equation}
e^{2\pi i (4p)^{-1} \frac{ 1 + 2m_q }{q}} \, e^{- 2\pi i (1  - 2 qx_q)  y }.
\end{equation}
is continuous, decreasing, and $Q$-periodic. 
That implies that there exists $y_0^q \in [0,Q]$ such that  
$e^{2\pi i (4p)^{-1} \frac{ 1 + 2m_q }{q}} \, e^{- 2\pi i (1  - 2 qx_q)  y_0^q } = 1$, 
and consequently, 
\begin{equation}
|f_q(y_0^q)| \geq \frac{1}{ (qx_q)^2} +  \frac{1}{ (1 - qx_q)^2} - 3 \geq \frac{1}{ (1/2)^2} +  \frac{1}{ (1 - 1/2)^2} - 3 = 5
\end{equation} 
because in $(0,1)$ the function $1/x^2 + 1/(1-x)^2$ has a minimum in $x=1/2$. 

$(c)$ The fact that $f_q$ is $Q$-periodic implies that 
$|f_q(y_n^q)| = |f_q(y_0^q + nQ)| =  |f_q(y_0^q )| \geq 5$. 
\end{proof}

We now complete the proof of Proposition~\ref{thm:Holder_Lower_Bound_At_Rationals}.
\begin{prop}\label{thm:Rationals32}
Let $x_0 \in \mathbb R$ and $t \in \mathbb Q$. 
If $\alpha_{x_0}(t) \neq 1/2$, then $\alpha_{x_0}(t) = 3/2 $. 
\end{prop}
\begin{proof}
By Proposition~\ref{thm:Asymptotic_Alternative}, 
$\alpha_{x_0}(t) = 1/2$ happens only if $x_q = 0$ and $G(p,m_q,q) \neq 0$. 

$\bullet$ If $x_q = 0$ and $G(p,m_q,q) = 0$, then $x_0 \in \mathbb Q$ and $q \in 2\mathbb N$.  
From \eqref{eq:Asymptotic_xq_Equal_0}
and the fact that
\begin{equation}\label{eq:Gauss_Sums_Complete}
G(p,m,q) = \left\{ 
\begin{array}{ll}
e^{2\pi i (4p)^{-1} m^2/q} \, G(p,0,q), & q \text{ odd,}\\
e^{2\pi i p^{-1} (m/2)^2/q} \,  G(p,0,q), & q \equiv 0 \pmod{4} \text{ and } m \text{ even}, \\
e^{2\pi i p^{-1} ((m-1)/2)^2/q} e^{2\pi i p^{-1} ((m-1)/2)/q} \, G(p,1,q), & q \equiv 2 \pmod{4} \text{ and } m \text{ odd}, 
\end{array}
\right.
\end{equation}
and $G(p,m,q) = 0$ otherwise, we have 
\begin{equation}\label{eq:Extra_Rationals_xq_0}
R_{x_0}\Big( \frac{p}{q} + h \Big) - R_{x_0}\Big( \frac{p}{q} \Big) + 2\pi i h  
= 2(1 \pm i)\, q^{3/2} |h|^{3/2} \,\sum_{m \text{ odd}} \, \frac{ G(p,m_q + m,q)}{\sqrt{q}} \,  \frac{e^{- 2\pi i \frac{m^2}{4 q^2 h}}}{ m^2}
 + O\left( q^{7/2} h^{5/2}  \right).
\end{equation}
It suffices to find a sequence $y_k \to \infty$ such that $|g(y_k)| \geq c > 0$ for some $c > 0$, where
\begin{equation}
g(y) = \sum_{m \text{ odd}} \, \frac{ G(p,m_q + m,q)}{\sqrt{q}} \,  \frac{e^{- 2\pi i m^2 y}}{ m^2},
\end{equation}
because that way, defining $h_k = 1/(4q^2 y_k)$, we get 
\begin{equation}
\Big| R_{x_0}\Big( \frac{p}{q} + h_k \Big) - R_{x_0}\Big( \frac{p}{q} \Big) + 2\pi i h_k \Big|
 \gtrsim_q h_k^{3/2} |g(y_k)| - O(h_k^{5/2}) 
\gtrsim_q h_k^{3/2}
\end{equation}
for all $k$ large enough, hence $\alpha_{x_0}(t) \leq 3/2$. 
So let us find that sequence $y_k$. 
According to \eqref{eq:Gauss_Sums_Complete},  
if $q \equiv 0 \pmod{4}$, by symmetry we can write
\begin{equation}
\begin{split}
g(y) & = \frac{G(p,0,q)}{\sqrt{q}} \,  \sum_{m \geq 0 \text{ odd}} \frac{e^{- 2\pi i m^2 y}}{ m^2} \Big(  e^{2\pi i p^{-1} \left( \frac{m_q + m}{2}  \right)^2 \frac{1}{q} } + e^{2\pi i p^{-1} \left( \frac{m_q - m}{2}  \right)^2 \frac{1}{q} }   \Big) \\
& =  2 \, \frac{G(p,0,q)}{\sqrt{q}} \, e^{2\pi i p^{-1} \frac{m_q^2}{4q}  } \, 
\sum_{m \geq 0 \text{ odd}} \frac{e^{- 2\pi i m^2 (y - \frac{ p^{-1} }{4q} ) }}{ m^2} \cos\left(  2\pi  \frac{p^{-1}  m_q }{2q} \, m \right).
\end{split}
\end{equation}
On the other hand, if $q \equiv 2 \pmod{4}$, then 
\begin{equation}
\begin{split}
g(y) & = \frac{G(p,1,q)}{\sqrt{q}} \,  \sum_{m \geq 0 \text{ odd}} \frac{e^{- 2\pi i m^2 y}}{ m^2}  \Big( e^{2\pi i p^{-1} \Big[ \left( \frac{m_q + m - 1}{2}  \right)^2 + \frac{m_q + m - 1}{2} \Big] \frac{1}{q} } + e^{2\pi i p^{-1} \Big[ \left( \frac{m_q - m - 1}{2}  \right)^2 + \frac{m_q - m - 1}{2} \Big] \frac{1}{q} }   \Big) \\
& =  2 \, \frac{G(p,1,q)}{\sqrt{q}} \, e^{2\pi i p^{-1} \frac{(m_q - 1)^2 + 2(m_q - 1)}{4q}  } \, 
\sum_{m \geq 0 \text{ odd}} \frac{e^{- 2\pi i m^2 (y - \frac{ p^{-1} }{4q} ) }}{ m^2} \, \cos\left(  2\pi   \frac{ p^{-1} m_q   }{2q} \, m \right).
\end{split}
\end{equation}
Choose the sequence $y_k = p^{-1}/(4q) + k $ for $k \in \mathbb N$. 
Then, since $x_q = |x_0 - m_q/q| = 0$ implies $x_0 = m_q/q$, 
but also $x_0 = P/Q$ in its reduced form, we get
\begin{equation}\label{eq:g_as_Fourier_Series}
|g(y_k)|  \simeq  
\left| \sum_{m = 0}^\infty \frac{\cos\left(  \pi  \frac{p^{-1}  P }{Q} \, (2m+1) \right)}{ (2m+1)^2} \right|, \qquad \forall k \in \mathbb N. 
\end{equation}
Define the Fourier series 
\begin{equation}
G(z)  =  
\sum_{m = 0}^\infty \frac{\cos\left(  (2m+1) \pi  z \right)}{ (2m+1)^2} = \frac{ \pi^2}{8} (1 - |2z|) \qquad z \in (-1,1),
\end{equation}
so that, after extending periodically to $\mathbb R$, in view of \eqref{eq:g_as_Fourier_Series}, we have  
$|g(y_n)| = |G(p^{-1}P/Q)|$ for all $n \in \mathbb N$. 
Observe that the only zeros of $G$ are $(2m+1)/2$ for $m \in \mathbb Z$.
We separate two cases again. If $q \equiv 0 \pmod{4}$, by \eqref{eq:Gauss_Sums_Complete} $m_q$ must be odd. 
Then $Qm_q = Pq$ implies $4 \mid Q$, hence both $p^{-1}$ and $P$ are odd.
We deduce $p^{-1}P/Q \neq (2m+1)/2$ for any $m \in \mathbb Z$, because otherwise $p^{-1}P = (2m+1) Q/2$ for some $m$, so $p^{-1}P$ would be even. 
If $q \equiv 2 \pmod{4}$, then $m_q$ is even and $Q(m_q/2) = P(q/2)$ implies that $Q$ is odd. Hence $p^{-1}P/Q \neq (2m+1)/2$ for any $m\in\mathbb Z$.  
In both cases, this implies that $|g(y_k)| = |G(p^{-1}P/Q)| \neq 0$
for all $k$, which is what we wanted to prove. 

$\bullet$ If $x_q \neq 0$, according to \eqref{eq:Asymptotic_xq_Not_0} we get
\begin{equation}\label{eq:Extra_Rationals_xq_Not0}
\begin{split}
\Big| R_{x_0}\Big( \frac{p}{q} + h \Big) - R_{x_0}\Big( \frac{p}{q} \Big) + 2\pi i h \Big|
\simeq \Big| (qh)^{3/2} \,\sum_{m \in \mathbb Z} \, \frac{ G(p,m_q + m,q)}{\sqrt{q}} \,  
\frac{e^{- 2\pi i \frac{(m - qx_q)^2}{4 q^2 h}}}{ (m - qx_q)^2} \,  + O\left( q^{7/2} h^{5/2}  \right) \Big|
\end{split}
\end{equation}
because $0 < qx_q \leq 1/2$. 
If $q$ is odd, we use \eqref{eq:Gauss_Sums_Complete} and the definition of $f_q$ in \eqref{eq:f_q}
to write 
\begin{equation}\label{eq:Extra_Rationals_xq_Not0_fq}
 \Big|R_{x_0}\Big( \frac{p}{q} + h \Big) - R_{x_0}\Big( \frac{p}{q} \Big) + 2\pi i h \Big|  
\simeq q^{3/2} h^{3/2} \,\Big| f_q\Big( \frac{1}{4q^2h} \Big) + O\big( q^2h\big) \Big|. 
\end{equation}
With the definition of $y_k^q$ in Lemma~\ref{thm:Lemma_f_q},  choose the sequence $h_k = 1/(4q^2y_k^q)$ that tends to zero and for which $|f_q(1/(4q^2h_k^q))| = |f_q(y^q_k)| \geq 5$. This and \eqref{eq:Extra_Rationals_xq_Not0_fq} show that $\alpha_{x_0}(t) = 3/2$. 
When $q$ is even,
by \eqref{eq:Gauss_Sums_Complete},
the sum in \eqref{eq:Extra_Rationals_xq_Not0}  only has either even or odd terms.
The main term is $m=0$ if even terms survive, and $m=1$ if odd terms survive,
and crude estimates in the error suffice to conclude. 

\end{proof}

\bibliographystyle{acm}
\bibliography{Multifractality_R_x0.bib}

%

\section*{Acknowledgements}
L. Vega is thankful to S. Jaffard and S. Seuret for insightful conversations.

\section*{Funding}
V. Banica is partially supported by the Institut Universitaire de France, by the French ANR project SingFlows. 
D. Eceizabarrena is funded in part by the 
Simons Foundation Collaboration Grant on Wave Turbulence (Nahmod’s award
ID 651469) and by the American Mathematical Society and the Simons Foundation under an AMS-Simons Travel Grant for the period 2022-2024. 
A. Nahmod is funded in part by NSF DMS-2052740, NSF DMS-2101381
and the Simons Foundation Collaboration Grant on Wave Turbulence (Nahmod’s award
ID 651469).
L. Vega is funded in part by MICINN (Spain) projects Severo Ochoa CEX2021-001142, and PID2021-126813NB-I00 (ERDF A way of making Europe), and by Eusko Jaurlaritza project IT1615-22 and BERC program.

\section*{Statements}
On behalf of all authors, the corresponding author states that there is no conflict of interest.

\end{document}